\documentclass{amsart}
\usepackage{clayams}
\usepackage{subcaption}

\newcommand{\n}[1]{\left\|#1\right\|}
\newcommand{\br}[1]{\left(#1\right)}
\newcommand{\proj}{\operatorname{proj}}
\newcommand{\dist}{\operatorname{dist}}
\newcommand{\spn}{\operatorname{span}}

\theoremstyle{plain}

\theoremstyle{plain}

\graphicspath{{./},{./figures/}}
\title[\sf The Natural Greedy Algorithm]{The Natural Greedy Algorithm\\for reduced bases in Banach spaces}

\author[\sf A.~Dereventsov]{\sf Anton~Dereventsov$^\star$}
\address{$^\star$\sf Department of Computational and Applied Mathematics, Oak Ridge National Laboratory \\
Oak Ridge TN 37831-6164 ({\tt dereventsova@ornl.gov}).}

\author[\sf C.~G.~Webster]{\sf Clayton G.~Webster$^\dagger$}
\address{$^\dagger$\sf Department of Mathematics, University of Tennessee, Knoxville, TN 37996, and
Department of Computational and Applied Mathematics, Oak Ridge National Laboratory, Oak Ridge TN 37831-6164
({\tt cwebst13@math.utk.edu}).}

\date{\today}


\begin{document}

\maketitle
\begin{abstract}
In this effort we introduce and analyze a novel reduced basis approach, used to construct an approximating subspace for a given set of data.
Our technique, which we call the Natural Greedy Algorithm (NGA), is based on a recursive approach for iteratively constructing such subspaces, and coincides with the standard, and the extensively studied, Orthogonal Greedy Algorithm (OGA) in a Hilbert space, defined in~\cite{buetal2012}.
However, for a given set of data in a general Banach space, the NGA is straightforward to implement and overcomes the explosion in computational effort introduced by the OGA, as we utilize an entirely new technique for projecting onto appropriate subspaces.
We provide a rigorous analysis of the NGA, and demonstrate that it's theoretical performance is similar to the OGA, while the realization of the former results in significant computational savings through a substantially improved numerical procedure.  Furthermore, we show that the empirical interpolation method (EIM) can be viewed as a special case of the NGA.
Finally, several numerical examples are used to illustrate the advantages of our NGA compared with 
other greedy algorithms and additional popular reduced bases methods, including EIM and proper orthogonal decomposition.

\end{abstract}

\section{Introduction}\label{section_introduction}
Reduced basis methods offer approaches of obtaining subspaces that approximate given sets of data.
This is often used in applications for reducing the dimensionality of a problem, and has become particularly popular recently as a way of obtaining solutions for parametric PDEs (see, e.g., \cite{mapatu2002b,mapatu2002a,pretal2002,veetal2003,roetal2007,se2008}).
For instance, it is generally unfeasible to employ a standard highly accurate numerical solution of a parametrized PDE for each parameter value, but one can obtain solutions for specifically selected values beforehand, and then use them to span a low-dimensional subspace in which the solution of the PDE can be approximated quickly (e.g., by employing a Galerkin procedure on the constructed subspace).
The rationale of this approach can be explained by C\'ea's lemma (see e.g.~\cite{ce1964}), which states that the solution obtained in this manner is quasi-optimal.

However, in this effort we consider a more general problem setting: let $\mathcal{F}$ be a set of elements of a Banach space $\mathcal{X}$ (e.g. a family of functions, a set of vectors, etc.) for which we need to construct an approximating subspace.
This set is specific for each problem setting and therefore usually cannot be approximated well by predetermined bases, including, e.g., wavelet, trigonometric, or polynomial bases in a case of Banach function spaces, or canonical disjoint bases in a case of discrete Banach spaces.
The idea of a reduced basis approach is to select key elements $f_0,\ldots,f_{n-1}$ of the set $\mathcal{F}$ and use them to construct an $n$-dimensional subspace $V_n = \spn\{f_0,\ldots,f_{n-1}\} \subset \mathcal{X}$ that approximates the set $\mathcal{F}$.
The objective of a reduced basis method is to provide a strategy for selecting these key elements $f_0,\ldots,f_{n-1}$.

One straightforward approach is to select the next basis element based on previously selected elements $f_0,\ldots,f_{n-1}$ by picking an element $f_n \in \mathcal{F}$ that is not approximated well by the current subspace $V_n = \spn\{f_0,\ldots,f_{n-1}\}$, i.e., to take
\[
	f_n = \mathop{\mathrm{argmax}}_{f \in \mathcal{F}} \dist(f,V_n).
\]
This method is known as the Orthogonal Greedy Algorithm (formally known as the Greedy Algorithm, however we add the adjective `Orthogonal' to avoid the confusion) and was formally stated and analyzed in the Hilbert space setting in~\cite{buetal2012}, where the first estimate on the rate of convergence was also provided.
This estimate was improved in~\cite{bietal2011} and then improved further and extended to the Banach space setting in~\cite{depewo2013}.
The most comprehensive technique for investigating the convergence of the Orthogonal Greedy Algorithm (OGA) was developed in~\cite{wo2015}, and the most recent improvement was given in~\cite{ng2018}, where the new estimate on the rate of convergence in a special case is shown.
For more details, see Section~\ref{section_ga}, where we discuss the OGA and the corresponding convergence results discussed above.

A key feature of the OGA is that it constructs a reduced basis with respect to the desired norm, thus specifically producing an approximation for the problem under consideration.
Nonetheless, on each iteration of the OGA it is required to compute the distance from an element $f \in \mathcal{F}$ to the previously constructed subspace $V_n = \spn\{f_0,\ldots,f_{n-1}\}$, which is equivalent to finding the orthogonal projection onto the subspace $V_n$ with respect to the fixed norm.
In a Hilbert space one generally approaches such a problem by orthogonalizing the vectors $\{f_k\}_{k=0}^{n-1}$ to obtain an orthogonal basis $\{f_k^\prime\}_{k=0}^{n-1}$ for $V_n$.
Then, the projections of each $f \in \mathcal{F}$ onto $V_n$ are easily found as $\sum_{k=0}^{n-1} \langle f,f_k^\prime \rangle f_k^\prime$.
However, in a Banach space this problem is difficult since there does not exist a straightforward procedure for calculating the orthogonal projection onto a subspace.
Thus for every $\mathcal{F}$ one has to solve an $n$-parameter minimization problem
\[
	\min_{\alpha_0,\ldots,\alpha_{n-1} \in \mathbb{R}} \n{f - \sum_{k=0}^{n-1} \alpha_k f_k},
\]
which, even though convex, becomes more complicated with each iteration and is not always feasible.

Another reduced basis algorithm that uses a recursive greedy strategy is the Empirical Interpolation Method (EIM), which was first introduced in~\cite{paetal2004} as an interpolation procedure for a family of functions $\mathcal{F} \subset L_\infty(\Omega)$.
This method constructs the basis elements $h_0,\ldots,h_{n-1} \in \mathcal{F}$ and the interpolation points $z_0,\ldots,z_{n-1} \in \Omega$ (in which the basis elements attain their maximal values) as $z_m = \operatorname{argmax}_{z \in \Omega} |h_m(z)|$.
At every iteration the EIM builds an approximation for each function $f \in \mathcal{F}$ by constructing linear combinations of the basis elements $h_0,\ldots,h_{n-1}$, which agrees with $f$ on points $z_0,\ldots,z_{n-1}$, i.e., $\sum_{k=0}^{n-1} \beta_k h_k$ such that
\[
	f(z_m) = \sum_{k=0}^{n-1} \beta_k h_k(z_m)
	\ \text{ for any }\  0 \le m \le n-1.
\]
A function $f_n \in \mathcal{F}$ that is not well approximated by such a linear combination is then selected, and the next interpolation point $z_n$ is chosen correspondingly, i.e.,
\[
	f_n = \mathop{\mathrm{argmax}}_{f\in\mathcal{F}} \n{f - \sum_{k=0}^{n-1} \beta_k h_k}_{L_\infty(\Omega)}
	\text{ and }\ 
	z_n = \mathop{\mathrm{argmax}}_{z\in\Omega} \left| f(z) - \sum_{k=0}^{n-1} \beta_k h_k(z) \right|,
\]
and the next basis element $h_n$ is constructed as the normalized remainder of the approximation of $f_n$, i.e.,
\[
	h_n = \frac{f_n - \sum_{k=0}^{n-1} \beta_k h_k}{f_n(z_n) - \sum_{k=0}^{n-1} \beta_k h_k(z_n)}.
\]
This procedure of obtaining a reduced basis is attractive from a computational perspective since in order to construct the approximation for an element $f \in \mathcal{F}$ one has to find the coefficients $\beta_0,\ldots,\beta_{n-1}$ by solving a linear system of $n$ equations, which is significantly simpler than solving an $n$-parameter optimization problem.
On the other hand, the simplicity of realization comes at the cost of freedom in the choice of norm: the reduced basis $\{h_k\}_{k=0}^{n-1}$ is designed to minimize $L_\infty(\Omega)$-norm and offers no flexibility in that matter.
For more details, see Section~\ref{section_other_rb}, where we briefly discuss the EIM and its relationship to the new approach we introduce in this paper.

In this paper we introduce a novel greedy algorithm for constructing reduced bases in Banach spaces, which does not require a predetermined norm, nor the solution of an optimization problem, and thus combines the flexibility of the OGA with the computational simplicity of the EIM.
Similar to the methods, described above, our algorithm constructs the reduced basis by employing a recursive greedy strategy: once basis elements $g_0,\ldots,g_{n-1}$ are constructed, they are used to approximate elements of the set $\mathcal{F}$ in order to choose the next element.
Similarly to the OGA, our method builds an approximation of an element $f \in \mathcal{F}$ by projecting $f$ onto the previously constructed subspace $V_n = \spn\{g_0,\ldots,g_{n-1}\}$; however it does not use the orthogonal projection but instead utilizes specifically constructed operators $\mathcal{R}_n : \mathcal{X} \to \mathcal{X}$, which will be introduced shortly.
In addition, similarly to the EIM, our way of obtaining an approximation $f \approx \sum_{k=0}^{n-1} \alpha_k g_k$ does not involve solving an optimization problem; instead we calculate the coefficients $\alpha_0,\ldots,\alpha_{n-1}$ by iteratively applying the norming functionals of the basis elements $g_0,\ldots,g_{n-1}$.
We call our method the Natural Greedy Algorithm (NGA) because our approach for approximating elements of the set $\mathcal{F}$ seems ``natural'' from several perspectives, as we will discuss below and detailed in Section~\ref{section_nga}.
In particular, the NGA does not just combine features of the OGA and the EIM, but in fact coincides with the former in $L_2$-spaces, and with the latter in $L_\infty$-spaces.
Thus the NGA can be thought of as an alternative generalization of the OGA from the Hilbert space setting to the Banach space setting.

\noindent
We propose here a new way of projecting onto a subspace of a Banach space by utilizing norming functionals.
While our method does not necessarily minimize the norm of the remainder, as the orthogonal projection does, it results in a significantly easier computationally approach since the norming functionals are known explicitly in many Banach spaces of interest (e.g., $L_p$-spaces), and in a general case can be easily computed as the derivative of the norm.
Namely, for given elements $g_0,\ldots,g_{n-1} \in \mathcal{X}$ we define operators $r_0,\ldots,r_{n-1}$ as $r_k(f) = f - F_{g_k}(f) \, g_k$, where $F_{g_k}$ is the norming functional of $g_k$.
These operators represent a remainder of a projection onto the one-dimensional subspace $\spn\{g_k\}$ and, thus, applying them consecutively results in an analogue of the remainder of the orthogonal projection onto the $n$-dimensional subspace $V_n = \spn\{g_0,\ldots,g_{n-1}\}$.
Furthermore, denote by $\mathcal{R}_n$ the combination of $r_{n-1},\ldots,r_0$, i.e.,
\[
	\mathcal{R}_n(f) = r_{n-1} \circ \ldots \circ r_0(f),
\]
then $\n{\mathcal{R}_n(f)}$ can be used as a simplified computational substitute for the distance from $f \in \mathcal{F}$ to the subspace $V_n = \spn\{g_0,\ldots,g_{n-1}\} \subset \mathcal{X}$.
We will see that such a method of projecting appears to be ``natural'' and possesses interesting analytical properties.
For instance, it is easy to see that the operator $\mathcal{R}_n(\cdot)$ is linear (while the orthogonal projector $\proj(\cdot,V_n)$ is not), and we will prove that $\mathcal{R}_n$ is an actual projector (i.e., $\mathcal{R}_n^2 = \mathcal{R}_n$).

Evidently, the quality of such a substitute depends critically on the choice of the vectors $g_0,\ldots,g_{n-1}$.
In fact, the value of $\mathcal{R}_n(f)$ is dictated only by the elements $g_0,\ldots,g_{n-1}$ and the geometry of the space in these points (which is, e.g., not the case for the orthogonal projection $\proj(f,V_n)$ as it is determined by the whole subspace $V_n$ rather than the individual elements).
Note also that if $\mathcal{X}$ is a Hilbert space and $\{g_k\}_{k=0}^{n-1}$ is an orthonormal system, then the value of $\n{\mathcal{R}_n(f)}$ is exactly the distance from $f$ to the subspace $V_n = \spn\{g_0,\ldots,g_{n-1}\}$; however in a general Banach space the value of the ratio $\n{\mathcal{R}_n(f)}/\dist(f,V_n)$ on a certain element $f \in \mathcal{X}$ can be as big as $2^n$, unless the vectors $g_0,\ldots,g_{n-1}$ are selected carefully.
In Section~\ref{section_nga} we describe the detailed process of obtaining appropriate elements $g_0,g_1,g_2,\ldots$ and 
formally introduce and analyze the NGA.
We state the convergence results corresponding to those of the OGA and note that every known estimate on the convergence rate for the OGA can also be stated for the NGA with an additional multiplicative constant.
Appendix~\ref{section_proofs} includes the details of the proofs of such theorems.

In addition to theoretical estimates, we compare the OGA and the NGA on a series of numerical examples and observe virtually no difference in the quality of the reduced bases generated by the two algorithms, while the computational complexity is significantly smaller for the NGA.
In Section~\ref{section_numerics} we present our numerical experiments for both greedy algorithms, compare them to the other known methods for constructing reduced bases, and discuss the computational complexity, which we measure in terms of the time each algorithm spends on the central processing unit.
Such a measurement is appropriate as it is directly related to a number of computations performed by each algorithm.

\section{The Orthogonal Greedy Algorithm}
\label{section_ga}

We begin this section by recalling standard notations that will be used throughout the paper.
Let $(\mathcal{X},\n{\cdot})$ be a Banach space.
For an element $f \in \mathcal{X}$ and a closed subspace $V \subset \mathcal{X}$ denote by $\proj(f,V)$ the orthogonal projection of $f$ onto $V$ and denote by $\dist(f,V)$ the distance from $f$ to $V$, i.e.,
\[
	\proj(f,V) = \mathop{\mathrm{argmin}}_{v \in V} \n{f - v}
	\text{ and }
	\dist(f,V) = \inf_{v \in V} \n{f - v} = \n{f - \proj(f,V)}.
\]

Let $\mathcal{F} \subset \mathcal{X}$ be a set of elements that we want to approximate by a low-dimensional subspace.
The following greedy algorithm (formally defined in~\cite{buetal2012}) constructs a sequence of nested subspaces $\{V_n\}_{n=0}^\infty$ that are designed to approximate the set $\mathcal{F}$.
\begin{definition}[{{\bf Orthogonal Greedy Algorithm}}]
For $\mathcal{F} \subset \mathcal{X}$ consider the following iterative procedure: 
\begin{enumerate}
\item[{\bf Step ${\bm 0}$.}]
    set $V_0 = \{0\}$ and find $f_0 = \operatorname{argmax}_{f\in\mathcal{F}} \dist(f,V_0)$;
    \smallskip
\item[{\bf Step ${\bm 1}$.}]
    set $V_1 = \spn\{f_0\}$ and find $f_1 = \operatorname{argmax}_{f\in\mathcal{F}} \dist(f,V_1)$;
    \smallskip
\item[{\bf Step ${\bm 2}$.}]
    set $V_2 = \spn\{f_0,f_1\}$ and find $f_2 = \operatorname{argmax}_{f\in\mathcal{F}} \dist(f,V_2)$;
\item[${\dots}$]
\item[{\bf Step ${\bm n}$.}]
    set $V_n = \spn\{f_0,\ldots,f_{n-1}\}$ and find $f_n = \operatorname{argmax}_{f\in\mathcal{F}} \dist(f,V_n)$.
\end{enumerate}
\end{definition}

\noindent
Note that generally sequences $\{f_n\}_{n=0}^\infty$ and $\{V_n\}_{n=0}^\infty$ are not unique and might not even exist due to the impossibility of finding an element $f_n \in \mathcal{F}$ that satisfies the greedy criterion.
The standard remedy to guarantee existence is the assumption of compactness of the set $\mathcal{F}$, which is widely used and is natural in PDE-based applications.
Under such a condition the process of greedy approximation is feasible and generally infinite (unless the set $\mathcal{F}$ is finite dimensional); in practice this procedure is stopped after the desired approximation accuracy has been achieved or the iteration limit has been reached.

Throughout the paper we will assume that $\mathcal{F}$ is a compact subset of $\mathcal{X}$, and thus at least one realization of the greedy algorithm is achievable; however uniqueness is still not guaranteed and, unless stated otherwise, variables $\{f_n\}_{n=0}^\infty$ and $\{V_n\}_{n=0}^\infty$ will denote any possible realization of the greedy algorithm.
Additionally, we assume for convenience that $\mathcal{F}$ is contained in the unit ball of $\mathcal{X}$, i.e., $\n{f} \le 1$ for any $f \in \mathcal{F}$.

Let us define for each $n \ge 0$ the number $\sigma_n = \sigma_n(\mathcal{F,X})$ that represent how well the set $\mathcal{F}$ is approximated by the subspace $V_n$:
\[
	\sigma_n = \sigma_n(\mathcal{F,X}) = \sup_{f \in \mathcal{F}} \dist(f,V_n).
\]
It is easy to see that the sequence $\{\sigma_n\}_{n=0}^\infty$ is monotone and that for any compact set $\mathcal{F} \subset \mathcal{X}$ one has $\sigma_n \to 0$ as $n \to \infty$, i.e., the OGA converges for any compact set $\mathcal{F} \subset \mathcal{X}$; however no direct estimates on the rate of convergence can be given without additional assumptions on $\mathcal{F}$.
Indeed, take any positive decreasing (to zero) sequence $\{a_n\}_{n=0}^\infty$ and consider the following compact set $\mathcal{F} \subset \ell_2$:
\[
	\mathcal{F} = \{a_n e_n\}_{n=0}^\infty,
\]
where $\{e_n\}_{n=0}^\infty$ is the canonical basis in $\ell_2$.
Then one achievable realization of the greedy algorithm is $f_n = a_n e_n$ for each $n \ge 0$, which implies $\sigma_n(\mathcal{F},\ell_2) = a_n$.

Nevertheless we can compare the approximation of the set $\mathcal{F}$ by the subspace $V_n$ with the best possible approximation by an $n$-dimensional subspace $X_n \subset \mathcal{X}$ given by the Kolmogorov $n$-width $d_n$:
\[
	d_n = d_n(\mathcal{F,X}) = \inf_{\substack{X_n \subset \mathcal{X} \\ \dim X_n=n}} \sup_{f \in \mathcal{F}} \dist(f,X_n).
\]
We note that while Kolmogorov $n$-widths have been introduced in 1935 (see~\cite{ko1936}) and extensively studied ever since, they still remain a mostly theoretical concept.
Namely, the width $d_n(\mathcal{F,X})$ might not be realizable on any subspace, and, more significantly, there are no constructive methods of obtaining an almost optimal subspace in a general setting.
Nonetheless, Kolmogorov widths are inherent in comparing the theoretical performance of approximating algorithms with the ideal, even if unattainable, case.
For a detailed discussion of $n$-widths and related concepts we refer the reader to the books~\cite{pi1985} and~\cite[Chapters~13--14]{logoma1996}.

\subsection{Convergence rate of the Orthogonal Greedy Algorithm}
We now return to estimating the convergence rate of the OGA.
The first direct comparison between $\sigma_n$ and $d_n$ in a Hilbert space is given in~\cite{buetal2012}, where it is shown (even though not stated explicitly) that
\[
	\sigma_n(\mathcal{F,H}) \le (n+1) \, 2^{n+1} d_n(\mathcal{F,H}).
\]
This estimate is improved in~\cite{bietal2011}, where it is shown that 
\[
	\sigma_n(\mathcal{F,H}) \le \frac{2^{n+1}}{\sqrt{3}} \, d_n(\mathcal{F,H})
\]
and that this estimate is sharp up to the factor of $2/\sqrt{3}$.

One technique that allows us to extend this result to a general Banach space is proposed in~\cite{wo2015}, where it is essentially shown that in a Banach space the above estimate holds with an additional factor (which depends on $\mathcal{X}$) that represents how isomorphic the subspaces of $\mathcal{X}$ are to the subspaces of $L_2$, which is characterized by the sequence $\{\gamma_n(\mathcal{X})\}_{n=1}^\infty$ defined below.

In particular, for Banach spaces $\mathcal{X}$ and $\mathcal{Y}$ denote by $GL(\mathcal{X,Y})$ the family of all linear isomorphisms from $\mathcal{X}$ to $\mathcal{Y}$.
Then the multiplicative Banach--Mazur distance between $\mathcal{X}$ and $\mathcal{Y}$ is given by
\[
	d(\mathcal{X,Y}) = \left\{\begin{array}{ll}
	\inf\{\|T\| \|T^{-1}\| : T \in GL(\mathcal{X,Y})\} & \text{if } \mathcal{X} \text{ is isomorphic to } \mathcal{Y},
	\\
	\infty & \text{if } \mathcal{X} \text{ is not isomorphic to } \mathcal{Y}.
	\end{array}\right.
\]
Denote by $\gamma_n(\mathcal{X})$ the supremum of multiplicative Banach--Mazur distances between $n$-dimensional subspaces of $\mathcal{X}$ and $\ell_2^n$, i.e.,
\[
	\gamma_n = \gamma_n(\mathcal{X}) = \sup\{d(V,\ell_2^n) : V \subset \mathcal{X} \text{ and } \dim V = n\}.
\]
It is easy to see that the sequence $\{\gamma_n(\mathcal{X})\}_{n=1}^\infty$ is non-decreasing and that $\gamma_1(\mathcal{X}) = 1$.
For any Hilbert space $\gamma_n(\mathcal{H}) = 1$ and for any Banach space $\gamma_n(\mathcal{X}) \le \sqrt{n}$ (it follows, for instance, from the Pietsch factorization theorem, see, e.g.,~\cite[Corollary~16.12.1]{ga2007}).
Moreover, it is known (see, e.g.,~\cite[III.B.9]{wo1991}) that in $L_p$-spaces 
\begin{equation}\label{gamma_n(L_p)}
	\gamma_n(L_p) \le n^{|\frac{1}{2}-\frac{1}{p}|} \text{ for any } 1 \le p \le \infty.
\end{equation}
Note that due to the compactness of the space of isometry classes of $n$-dimensional Banach spaces, for any $n$-dimensional subspace $Y_n \subset \mathcal{X}$ there exists an isomorphism $T \in GL(Y_n,\ell_2^n)$ with $\|T^{-1}\| = 1$ such that
\[
	\n{y} = \n{T^{-1}Ty} \le \n{Ty}_2 \le d(Y_n,\ell_2^n) \n{y} \le \gamma_n(\mathcal{X}) \n{y}.
\]
Thus, for any $Y \subset \mathcal{X}$ there exists a Euclidean norm $\n{\cdot}_e$ on $Y$ given by $\n{y}_e = \n{Ty}_2$ such that
\begin{equation}\label{e_norm}
	\n{y} \le \n{y}_e \le \gamma_{\dim Y}(\mathcal{X}) \n{y}
	\text{ for any } y \in Y_n.
\end{equation}

The direct comparison between $\sigma_n$ and $d_n$ for $L_p$-spaces is given in~\cite{wo2015}; however a similar technique can be used for a general Banach space $\mathcal{X}$ with minor changes in the proof.
Thus we state the generalized theorem and deduce the original result as a corollary (with a slight improvement in the multiplicative constant).
We demonstrate how to obtain the proof of this theorem in Appendix~\ref{section_proofs}.
\begin{theorem}\label{theorem_ga_direct}
For a compact set $\mathcal{F} \subset \mathcal{X}$ the Orthogonal Greedy Algorithm provides
\[
	\sigma_n(\mathcal{F,X}) \le \gamma_{2n+1}(\mathcal{X}) \, 2^{n+1} d_n(\mathcal{F,X}).
\]
\end{theorem}
\begin{corollary}
For a compact set $\mathcal{F} \subset L_p$ with any $1 \le p \le \infty$ the Orthogonal Greedy Algorithm provides
\[
	\sigma_n(\mathcal{F},L_p) \le 2\sqrt{3} \, n^{|\frac{1}{2}-\frac{1}{p}|} \, 2^n d_n(\mathcal{F},L_p).
\]
\end{corollary}

\noindent
An alternative approach to estimating the rate of convergence of the OGA is proposed in~\cite{depewo2013}, where a delayed estimate is proven, i.e., a comparison between $\sigma_n$ and $d_m$ for $m < n$.
In particular, the following estimate is proven in a Hilbert space setting,
\[
	\sigma_n(\mathcal{F,H}) \le \sqrt{2} \min_{0 < m < n} d_m^{1-m/n}(\mathcal{F,H}),
\]
and for a general Banach space,
\[
	\sigma_n(\mathcal{F,X}) \le \sqrt{2n} \min_{0 < m < n} d_m^{1-m/n}(\mathcal{F,X}).
\]
It is conjectured in~\cite{depewo2013} that the additional factor of $\sqrt{n}$ in the last estimate, while unable to be removed in general, can be improved for some Banach spaces; for instance, in the case of $L_p$-spaces one expects to replace $\sqrt{n}$ with $n^{|1/2-1/p|}$.
This conjecture is proven in~\cite{wo2015}, where the following result is stated.
\begin{theorem}\label{theorem_ga_indirect}
For a compact set $\mathcal{F} \subset \mathcal{X}$ the Orthogonal Greedy Algorithm provides
\[
	\sigma_n(\mathcal{F,X}) \le \sqrt{2} \min_{0 < m < n} \gamma_{n+m}(\mathcal{X}) \, d_m^{1-m/n}(\mathcal{F,X}).
\]
\end{theorem}
\begin{corollary}
For a compact set $\mathcal{F} \subset L_p$ with any $1 \le p \le \infty$ the Orthogonal Greedy Algorithm provides
\[
	\sigma_n(\mathcal{F},L_p) \le \sqrt{2} \min_{0 < m < n} (n+m)^{|\frac{1}{2}-\frac{1}{p}|}  \, d_m^{1-m/n}(\mathcal{F},L_p).
\]
In particular,
\[
	\sigma_{2n}(\mathcal{F},L_p) \le \sqrt{6} \, n^{|\frac{1}{2}-\frac{1}{p}|} \sqrt{d_n(\mathcal{F},L_p)}.
\]
\end{corollary}

\noindent
Additional estimates can be given when $d_n$ satisfies a special decay rate.
For instance, if $d_n(\mathcal{F}) \le A \exp(-a n^\alpha)$ with some constants $0 < \alpha,a,A < \infty$, then it is shown in~\cite{bietal2011} that in a Hilbert space the OGA guarantees
\[
	\sigma_n(\mathcal{F,H}) \le B \exp(-b n^\beta),
\]
with $B = B(\alpha,a,A)$, $b = b(\alpha,a)$, and $\beta = \frac{\alpha}{\alpha+1}$.
This estimate is improved and extended to the Banach space setting in~\cite{depewo2013}, that is
\[
	\sigma_n(\mathcal{F,X}) \le B \exp(-b n^\alpha),
\]
with $B = B(A)$ and $b = b(\alpha,a)$.
Similarly, if $d_n(\mathcal{F}) \le A n^{-\alpha}$ with some constants $0 < \alpha,A < \infty$, then it is shown in~\cite{bietal2011} that in a Hilbert space the OGA satisfies
\[
	\sigma_n(\mathcal{F,H}) \le B n^{-\alpha},
\]
with $B = B(\alpha,A)$.
The corresponding estimate for a Banach space is proven in~\cite{depewo2013}:
\[
	\sigma_n(\mathcal{F,X}) \le B n^{-\alpha + 1/2 + \epsilon},
\]
with any $\epsilon > 0$ and $B = B(\epsilon,\alpha,A)$.
This estimate is improved in~\cite{wo2015}, where it is shown that if, in addition, one has a bound $\gamma_n(\mathcal{X}) \le C n^{\mu}$ with some $C > 0$ and $0 \le \mu \le \min\{\alpha,1/2\}$, then
\[
	\sigma_n(\mathcal{F,X}) \le B (\ln n)^\alpha \, n^{-\alpha + \mu},
\]
with $B = B(\alpha,\mu,A,C)$.
The most recent bound of this type is given in~\cite{ng2018}, where it is shown that the power of the logarithm can be replaced with $1/2$, which essentially results in the estimate
\[
	\sigma_n(\mathcal{F,X}) \le B (\ln n)^{\min\{\alpha,1/2\}} \, n^{-\alpha + \mu},
\]
with $B = B(\alpha,\mu,A,C)$.

\subsection{Comments on the Weak Orthogonal Greedy Algorithm}\label{remark_woga}
In general the OGA is computationally expensive since at the $n$-th iteration of the algorithm one has to solve an $n$-dimensional minimization problem for each element of the set $\mathcal{F}$.
While this is a manageable task in a Hilbert space, where one can orthogonolize the basis and compute the projection, in a Banach space this is a complicated problem, which increases its computational complexity exponentially with each consecutive iteration.

We conclude this section by commenting on two common ways to simplify the realization of the algorithm.
The first method, which is a classical approach in the greedy approximation theory (see, e.g.,~\cite{te2011}), is to consider a weak version of OGA, where instead of finding $f_n = \operatorname{argmax}_{f\in\mathcal{F}} \dist(f,V_n)$ it is sufficient to find any such $f_n^w \in \mathcal{F}$ that
\[
    \dist(f_n^w,V_n) \ge \gamma \sup_{f\in\mathcal{F}} \dist(f,V_n),
\]
where $\gamma \in (0,1]$ is a fixed constant which is generally referred to as the weakness parameter.
Second, it might be possible to find for each $n \ge 1$ an easy-to-compute surrogate $s_n : \mathcal{F} \to \mathbb{R}$ such that
\begin{equation}\label{ga_surrogate}
	c_s s_n(f) \le \dist(f,V_n) \le C_s s_n(f) \ \text{ for all } f \in \mathcal{F} \text{ and } n \ge 0,
\end{equation}
with some constants $0 < c_s \le C_s < \infty$.
In this case, instead of finding $f_n$ one can search for $f_n^s = \operatorname{argmax}_{f\in\mathcal{F}} s_n(f)$, which would immediately satisfy 
\[
	\dist(f_n^s,V_n) \ge \frac{c_s}{C_s} \sup_{f\in\mathcal{F}} \dist(f,V_n),
\]
which is essentially equivalent to the weak version of the OGA with $\gamma = c_s / C_s$.

The analysis of the Weak OGA is essentially the same as that of the OGA, and all the results stated in this section hold for the weak version as well with an additional factor of $\gamma^{-1}$.
We note that while the Weak OGA results in reduced computational complexity, it requires additional information on the problem either in terms of a sufficiently tight upper estimate of $\sup_{f\in\mathcal{F}} \dist(f,V_n)$, or in terms of suitable surrogates $s_1,s_2,s_3,\ldots$, neither of which is generally available, and thus the Weak OGA cannot be employed universally.


\section{The Natural Greedy Algorithm}\label{section_nga}
In this section we introduce the Natural Greedy Algorithm (NGA)~--- an alternative to the OGA, in which instead of maximizing the distance to the constructed subspace, one maximizes the norm of a specifically constructed and easily-computable operator.
The behavior of the NGA depends on the smoothness of the space, thus we recall the related concepts.

For a non-zero element $x \in \mathcal{X}$ a norming functional $F_x$ is an element of $\mathcal{X}^*$ such that 
\[
	\n{F_x}_{\mathcal{X}^*} = 1 \text{ and } F_x(x) = \n{x}.
\]
Existence of norming functionals in Banach spaces is a simple corollary of the Hahn--Banach theorem (see e.g.~\cite[Corollary~II.3]{be1982}) however their uniqueness is generally not guaranteed.
Banach spaces in which norming functionals are unique are the smooth spaces; in this case for any elements $x,y \in \mathcal{X}$ the value of $F_x(y)$ can be found by computing the G\^ateaux derivative of the norm (see, e.g.,~\cite[Chapter~I.1]{degozi1993}), i.e.,
\begin{equation}\label{norming_functional_smooth}
	F_x(y) = \left.\frac{d}{du} \n{x+uy}\right|_{u=0} = \lim_{u \to 0} \frac{\n{x+uy}-\n{x}}{u}.
\end{equation}
Moreover, norming functionals are known for Banach spaces that are commonly utilized in applications. 
In particular, for the following smooth spaces:
\begin{enumerate}
	\item \text{Hilbert space: }\ 
	$\displaystyle{F_x(y) = \frac{\langle x, y \rangle}{\n{x}}}$;
	\item \text{$\ell_p$-space, $1 < p < \infty$: }\ 
	$\displaystyle{F_x(y) = \frac{\sum_{n=1}^\infty \operatorname{sgn}(x_n) |x_n|^{p-1} y_n}{\n{x}_p^{p-1}}}$;
	\item \text{$L_p(\Omega,\mu)$-space, $1 < p < \infty$: }\ 
	$\displaystyle{F_x(y) = \int\limits_{\Omega} \frac{\operatorname{sgn}(x) |x|^{p-1} y}{\n{x}_p^{p-1}} \, d\mu}$.
\end{enumerate}
We remark that in cases $p = 1$ and $p = \infty$ the spaces $\ell_p$ and $L_p(\Omega,\mu)$ are not smooth and thus norming functionals are not unique.
Throughout the paper we will be using the following functionals in those settings:
\begin{align}
	\label{norming_functional_l1}
	F_x^{\ell_1}(y) &= \sum_{n=1}^\infty \operatorname{sgn}(x_n) \, y_n;
	\\
	\label{norming_functional_linf}
	F_x^{\ell_\infty}(y) &= \operatorname{sgn}(x_m) \, y_m,
	\text{ where } m \in \mathbb{N} : |x_m| = \n{x}_{\ell_\infty};
	\\
	\nonumber
	F_x^{L_1}(y) &= \int_\Omega \operatorname{sgn}(x) \, y \, d\mu;
	\\
	\nonumber
	F_x^{L_\infty}(y) &= \operatorname{sgn}(x(t)) \, y(t),
	\text{ where } t \in \Omega : |x(t)| = \n{x}_{L_\infty(\Omega)}.
\end{align}

The smoothness of a Banach space $\mathcal{X}$ is characterized by the modulus of smoothness $\rho(u)$, which is defined as
\begin{equation}\label{modulus_of_smoothness}
	\rho(u) = \rho(u,\mathcal{X}) = \sup_{\n{x}=\n{y}=1} \frac{\n{x+uy} + \n{x-uy}}{2} - 1.
\end{equation}
The modulus of smoothness is an even and convex function and, therefore, $\rho(u)$ is non-decreasing on $(0, \infty)$.
A Banach space is uniformly smooth if $\rho(u) = o(u)$ as $u \to 0$.
We say that the modulus of smoothness $\rho(u)$ is of power type $q \in [1;2]$ if $\rho(u) \leq c_\rho u^q$ with some constant $c_\rho = c_\rho(\mathcal{X}) > 0$ for all $u \ge 0$.
One easily concludes by using the triangle inequality that any Banach space has a modulus of smoothness of power type $1$ and that any Hilbert space has a modulus of smoothness of power type $2$ by applying the parallelogram law.
Additionally, it is known (see, e.g.,~\cite[p.~63]{litz1979}) that the moduli of smoothness $\rho_p(u)$ of $L_p$-spaces satisfy the asymptotic estimate
\[
	\rho_p(u) = \left\{
	\begin{array}{ll}
		\frac{1}{p} \, u^p + o(u^p), & 1 < p \leq 2
		\\
		\frac{p-1}{2} \, u^2 + o(u^2), & 2 \leq p < \infty
	\end{array}
	\right.,
\]
i.e., $L_p$-spaces have the modulus of smoothness of power type $q = \min\{p,2\}$.

\noindent
For a more detailed discussion on the smoothness of norms and their relation to the other geometrical aspects of Banach spaces we refer the reader to books~\cite{be1982} and~\cite{degozi1993}.

We now define the operator sequences which are foundational for the NGA.
For a given normalized sequence of elements $\{g_n\}_{n=0}^\infty \in \mathcal{X}$ denote
\begin{align}
	\nonumber
	r_n(f) &= f - F_{g_n}(f) \, g_n,
	\ \text{ and}
	\\
	\label{nga_R_n}
	\mathcal{R}_n(f) &= \left\{
	\begin{array}{ll}
		f, & n = 0,
		\\
		r_{n-1} \circ r_{n-2} \circ \ldots \circ r_1 \circ r_0 \, (f), & n \ge 1.
	\end{array}
	\right.
\end{align}
Essentially, operators $r_n$ and $\mathcal{R}_n$ represent the remainders of certain projections of an element $f \in \mathcal{X}$ onto $\spn\{g_n\}$ and $\spn\{g_0,\ldots,g_{n-1}\}$ respectively.
Indeed, if we view $F_{g_n}(\cdot) \, g_n$ as a projector onto $g_n$ then $r_n(\cdot)$ is the remainder of this projection, and $\mathcal{R}_n(\cdot)$ is the combination of $n$ remainders of such one-dimensional projections.
In fact, if $\{g_k\}_{k=0}^{n-1}$ is an orthonormal system in a Hilbert space $\mathcal{H}$, then for any $f \in \mathcal{H}$ we have $\mathcal{R}_n(f) = f - \proj(f,\spn\{g_0,\ldots,g_{n-1}\})$.
While it is not exactly the case in a general Banach space, with the right sequence $\{g_n\}_{n=0}^\infty$ the value of $\n{\mathcal{R}_n(f)}$ can be used in place of the distance from $f$ to $\spn\{g_0,\ldots,g_{n-1}\}$, as we show later.

This new concept of projecting in Banach spaces is the basis of the NGA; it is computationally advantageous since it only requires the knowledge of the norming functionals of the reduced basis (which can often be stated explicitly or easily calculated as~\eqref{norming_functional_smooth} in other cases), as opposed to the orthogonal projection, which requires the solution of an optimization problem.
Note that in general the operators $\{r_n\}_{n=0}^\infty \in \mathcal{X}^*$ (and hence $\{\mathcal{R}_n\}_{n=0}^\infty$) are not unique since the corresponding norming functionals might lack uniqueness.
In such cases, unless stated otherwise, we consider any suitable operator sequences.

The NGA defined below recurrently constructs sequences $\{f_n\}_{n=0}^\infty$ and $\{g_n\}_{n=0}^\infty$, that in turn define operators $\{\mathcal{R}_n\}_{n=0}^\infty$ and span nested subspaces $\{V_n\}_{n=0}^\infty$ which approximate the set $\mathcal{F}$.
\begin{definition}[{{\bf Natural Greedy Algorithm}}]
Given the definition of the operators in~\eqref{nga_R_n} and for $\mathcal{F} \subset X$ consider the following iterative procedure: 
\begin{enumerate}
\item[{\bf Step ${\bm 0}$.}]
    set $V_0 = \{0\}$, find $f_0 = \operatorname{argmax}_{f\in\mathcal{F}} \|\mathcal{R}_0(f)\|$
    \\
    \phantom{set $V_0 = \{0\}$, find $f_0 = \operatorname{argmax}_{f\in\mathcal{F}}$}
    and let $\displaystyle{g_0 = \frac{\mathcal{R}_0(f_0)}{\n{\mathcal{R}_0(f_0)}}}$;
    \smallskip
\item[{\bf Step ${\bm 1}$.}]
    set $V_1 = \spn\{f_0\}$, find $f_1 = \operatorname{argmax}_{f\in\mathcal{F}} \|\mathcal{R}_1(f)\|$
    \\
    \phantom{set $V_1 = \spn\{f_0\}$, find $f_1 = \operatorname{argmax}_{f\in\mathcal{F}}$}
    and let $\displaystyle{g_1 = \frac{\mathcal{R}_1(f_1)}{\n{\mathcal{R}_1(f_1)}}}$;
    \smallskip
\item[{\bf Step ${\bm 2}$.}]
    set $V_2 = \spn\{f_0,f_1\}$, find $f_2 = \operatorname{argmax}_{f\in\mathcal{F}} \|\mathcal{R}_2(f)\|$
    \\
    \phantom{set $V_2 = \spn\{f_0,f_1\}$, find $f_2 = \operatorname{argmax}_{f\in\mathcal{F}}$}
    and let $\displaystyle{g_2 = \frac{\mathcal{R}_2(f_2)}{\n{\mathcal{R}_2(f_2)}}}$;
    \smallskip
\item[${\dots}$]
\item[{\bf Step ${\bm n}$.}]
    set $V_n = \spn\{f_0,\ldots,f_{n-1}\}$, find $f_n = \operatorname{argmax}_{f\in\mathcal{F}} \|\mathcal{R}_n(f)\|$
    \\
    \phantom{set $V_n = \spn\{f_0,\ldots,f_{n-1}\}$, find $f_n = \operatorname{argmax}_{f\in\mathcal{F}}$}
    and let $\displaystyle{g_n = \frac{\mathcal{R}_n(f_n)}{\n{\mathcal{R}_n(f_n)}}}$.
\end{enumerate}
\end{definition}

\noindent
Similarly to the OGA, on the $n$-th iteration for each element of $\mathcal{F}$ we construct an approximation from $V_n = \spn\{f_0,\ldots,f_{n-1}\} = \spn\{g_0,\ldots,g_{n-1}\}$ in the form of a linear combination of already chosen elements.
The key difference is that in order to obtain the coefficient of this linear combination the NGA does not require solving an $n$-dimensional optimization problem; instead it calculates them directly by computing the values of norming functionals of the basis elements $\{g_k\}_{k=0}^{n-1}$.
While such technique of obtaining coefficients is generally not optimal in the sense of minimizing the norm of the remainder, it is inherent to the space as the the coefficients are dictated by the geometry of the norm on the elements of reduced basis.

\begin{remark}
Note that the NGA cannot proceed with the construction of reduced basis if at some step $n$ we have $\operatorname{argmax}_{f\in\mathcal{F}} \|\mathcal{R}_n(f)\| = 0$.
As we will see later, this situation occurs if and only if the set $\mathcal{F}$ is finite dimensional.
Moreover, in this case, just as for the OGA, we have $V_n = \spn\{g_0,\ldots,g_{n-1}\} \supset \mathcal{F}$ and $n = \dim\mathcal{F}$.
For convenience, from now on we assume that $\mathcal{F}$ is compact and infinite dimensional: compactness guarantees the possibility of choosing an element $f_n$, and infinite dimensions ensures the existence of a corresponding vector $g_n$.
Hence we presume that the NGA recursively constructs sequences $\{f_n\}_{n=0}^\infty$, $\{g_n\}_{n=0}^\infty$, $\{\mathcal{R}_n\}_{n=0}^\infty$ and, $\{V_n\}_{n=0}^\infty$.
Alternatively, if $\dim \mathcal{F} < \infty$ then the trivial modification in form of replacing $\infty$ with $\dim \mathcal{F}$ in sequence indexing needs to be performed.
\end{remark}

\subsection{Analysis of the operators $\mathcal{R}_n$}
To understand the primary concept behind the NGA we analyze the behavior of the operators $\{\mathcal{R}_n\}_{n=0}^\infty$, and show that our method of obtaining reduced bases, despite seeming unintuitive, is actually quite natural from several perspectives (hence the name of the algorithm).
In addition, we will see that in a Hilbert space the NGA coincides with the OGA.

We begin by establishing an important property of the operators $\{\mathcal{R}_n\}_{n=0}^\infty$ that will be our main tool in understanding the behavior of the constructed reduced basis $\{g_n\}_{n=0}^\infty$.
Specifically, we use induction to show that 
\begin{equation}\label{nga_F_k(R_m)=0}
	F_{g_k}(\mathcal{R}_m(f)) = 0
	\text{ for any }
	m > k \ge 0
	\text{ and any }
	f \in \mathcal{X}.
\end{equation}
Recall that from the definition~\eqref{nga_R_n} of operators $\{\mathcal{R}_n\}_{n=0}^\infty$ for any $m > k$ we have 
\[
	\mathcal{R}_m(f) = r_{m-1} \circ \ldots \circ r_0(f)
	= r_{m-1} \circ \ldots \circ r_k \circ \mathcal{R}_k(f).
\]
It is clear that the base of induction $m = k+1$ holds since $r_k(f) = f - F_{g_k}(f) \, g_k$ and hence
\[
	F_{g_k}(\mathcal{R}_{k+1}(f)) = F_{g_k} \big( r_k(\mathcal{R}_k(f)) \big) = 0.
\]
Assume that the hypothesis holds for some $m > k$, i.e. $F_{g_k}(\mathcal{R}_m(f)) = 0$ for any $f \in \mathcal{X}$.
Then, since by construction $g_m = \mathcal{R}_m(f_m) / \n{\mathcal{R}_m(f_m)}$, we deduce
\begin{align*}
	F_{g_k}(\mathcal{R}_{m+1}(f))
	&= F_{g_k} \big( r_m(\mathcal{R}_m(f)) \big)
	= F_{g_k} \Big( \mathcal{R}_m(f) - F_{g_m}(\mathcal{R}_m(f)) \, g_m \Big)
	\\
	&= F_{g_k} \big( \mathcal{R}_m(f) \big) - F_{g_m}(\mathcal{R}_m(f)) \, F_{g_k}\br{\frac{\mathcal{R}_m(f_m)}{\n{\mathcal{R}_m(f_m)}}}
	= 0,
\end{align*}
which concludes the proof.

Condition~\eqref{nga_F_k(R_m)=0} implies that $\mathcal{R}_n^2 = \mathcal{R}_n$, i.e., the operator $\mathcal{R}_n : \mathcal{X} \to \mathcal{X}$ is a projector.
Moreover, note that for any $m > k \ge 0$ we have $F_{g_k}(g_m) = 0$, and hence $r_k(g_m) = g_m$.
Therefore $\mathcal{R}_n(V_n) = 0$ since for any coefficients $\{\alpha_k\}_{k=0}^{n-1}$ we have
\begin{equation}\label{nga_R_n(V_n)=0}
	\mathcal{R}_n\br{\sum_{k=0}^{n-1} \alpha_k g_k}
	= r_{n-1} \circ \ldots \circ r_0 \br{\sum_{k=0}^{n-1} \alpha_k g_k}
	= r_{n-1} \circ \ldots \circ r_j \br{\sum_{k=j}^{n-1} \alpha_k g_k}
	= 0.
\end{equation}

\noindent
Thus, we have shown that in order to construct a reduced basis the NGA uses a linear projector $\mathcal{R}_n : \mathcal{X} \to \mathcal{X}$ with $\ker \mathcal{R}_n = V_n$.
In particular, it guarantees that the NGA will not, on any step, select a vector from the already generated subspace $V_n$, which implies that $\dim V_n = n$, i.e., the chosen elements $\{f_k\}_{k=0}^{n-1}$ are linearly independent.
Essentially, operator $\mathcal{R}_n$ decomposes any $f \in \mathcal{X}$ into the sum $f = \mathcal{R}_n(f) + \br{f - \mathcal{R}_n(f)}$, where $f - \mathcal{R}_n(f) \in V_n$, and $\mathcal{R}_n(f)$ represents the remainder of the projection of $f$ onto $V_n$ and determines which element of the set $\mathcal{F}$ will be chosen on this iteration of the algorithm.
From this point of view, the OGA uses the same approach as it relies on the orthogonal projector $\mathcal{P}_n(f) = f - \proj(f,V_n)$.
Although both algorithms make use of projectors ($\mathcal{R}_n$ and $\mathcal{P}_n$ respectively) which are independent of the translation about the current constructed subspace $V_n$, the main difference is that operator $\mathcal{P}_n$ is generally nonlinear (with the exception of the Hilbert space setting), while operator $\mathcal{R}_n$ is linear regardless of the choice of a space.

\subsection{Relations to orthogonality}
We now discuss the OGA and the NGA from the perspective of orthogonality of the reduced bases.
To begin, let us recall the concept of orthogonality in Banach spaces.
We consider here Birkhoff--James orthogonality, i.e. for any $x,y \in \mathcal{X}$ we say that $x$ is orthogonal to $y$ (denoted $x \perp y$) if $\n{x + \lambda y} \ge \n{x}$ for any $\lambda \in \mathbb{R}$.
This concept is directly related to the orthogonal projection onto subspaces of Banach spaces since $x - \proj(x,\mathcal{Y}) \perp \mathcal{Y}$ for any $x \in \mathcal{X}$ and any $\mathcal{Y} \subset \mathcal{X}$; however, in order to achieve orthogonality even for just two vectors $x$ and $y$, one generally has to solve the minimization problem $\min_{\lambda \in \mathbb{R}} \n{x + \lambda y}$ (which is how approximations are constructed in the OGA).

We propose here an alternative approach to projecting in Banach spaces, which is based on the relation between norming functionals and orthogonality.
First, it is easy to see that for any elements $x,y \in \mathcal{X} \setminus \{0\}$ and a norming funcitonal $F_x$ the condition $F_x(y) = 0$ implies $x \perp y$.
Indeed, for any $\lambda \in \mathbb{R}$ we have
\[
	\n{x + \lambda y} \ge F_x(x + \lambda y) = \n{x}.
\]
The inverse implication can be shown in uniformly smooth Banach spaces with the use of the well-known inequality (see e.g.~\cite[Lemma~6.1]{te2011}), which follows directly from the definition of modulus of smoothness~\eqref{modulus_of_smoothness}, namely,
\[
	\n{x - \lambda y} \le \n{x} - \lambda F_x(y) + 2\n{x} \rho\br{\frac{\lambda \n{y}}{\n{x}}}.
\]
Assume that $x \perp y$ but $F_x(y) \ne 0$.
Then for sufficiently small $\lambda > 0$ the previous inequality and the uniform smoothness of $\mathcal{X}$ provide
\[
	\n{x - \operatorname{sgn}(F_x(y)) \, \lambda y} \le \n{x} - \lambda |F_x(y)| + o(\lambda) < \n{x},
\]
which contradicts the orthogonality.

Together with these relations, property~\eqref{nga_F_k(R_m)=0} implies that the vectors $\{g_n\}_{n=0}^\infty$ constructed by the NGA form a semi-orthogonal basis for $V_n$ in the sense
\[
	g_k \perp g_m
	\text{ for any } m > k \ge 0.
\]
Similarly, if during the realization of the OGA one considers vectors $h_0,h_1,h_2,\ldots$ defined by $h_n = (f_n - \proj(f_n,V_n)) / \dist(f_n,V_n)$, then $\{h_k\}_{k=0}^\infty$ will also form a semi-orthogonal basis for $V_n$, but in the reverse order:
\[
	h_k \perp h_m
	\text{ for any } k > m \ge 0.
\]
Note that in a Hilbert space the orthogonality relation is symmetric, i.e., $x \perp y \iff y \perp x$, both algorithms produce the same sequence $\{f_n\}_{n=0}$ and $\{h_n\}_{n=0}^\infty = \{g_n\}_{n=0}^\infty$ is the orthogonalization of $\{f_n\}_{n=0}^\infty$.

Thus, when the NGA computes the remainder of the projection onto the subspace $V_n$, it subsequently computes $n$ remainders of one-dimensional projections on vectors $g_0,\ldots,g_{n-1}$ (which are given by $F_{g_k}(\cdot) \, g_k$).
This procedure seems appropriate since $\{g_k\}_{k=0}^{n-1}$ is a semi-orthogonal basis for $V_n$ and therefore in a Hilbert space we would get $\mathcal{R}_n(f) = f - \proj(f,V_n)$.
In a general Banach space, however, the concept of orthogonality is more complex and thus we only get $\{g_k\}_{k=0}^{n-1} \perp \mathcal{R}_n(f)$, while the remainder of the orthogonal projection provides $(f - \proj(f,V_n)) \perp \{g_k\}_{k=0}^{n-1}$.
Hence, such projection is not optimal in the sense of minimizing the norm of the remainder since it does not provide the condition $\mathcal{R}_n(f) \perp V_n$, and, theoretically, can even increase the norm of the projected element as Lemma~\ref{estimate_norm_R_n} below states; however calculating projection in this way is much simpler computationally as it does not involve solving a minimization problem.
Moreover, our numerical results in Section~\ref{section_numerics} show that, in practice, this imperfection in norm minimization does not deteriorate the quality of the constructed reduced basis, and the convergence rate of the NGA, despite its computational simplicity, is comparable to the one of the OGA.

We note that the realization of the Orthogonal Greedy Algorithm cannot be simplified along the same lines since for any $f \in \mathcal{X}$ one has to find $v_n = \proj(f,V_n)$ so that $(f - v_n) \perp V_n$, which, if $X$ is uniformly smooth, is equivalent to
\[
	F_{f - v_n}(V_n) = 0.
\]
To the best of our knowledge, there are no substantial theoretical findings on the relationship between an element $x \in \mathcal{X}$ and a norming functional $F_x \in \mathcal{X^*}$ (except in a Hilbert space setting), and therefore constructing the required vector $v_n \in V_n$ is only achievable by solving an $n$-parameter optimization problem.

\subsection{Convergence of the Natural Greedy Algorithm}\label{section_nga_convergence}
We now discuss the convergence properties of the NGA.
Similarly to the OGA we impose compactness on $\mathcal{F}$ to guarantee the feasibility of at least one realization of the NGA and existence of sequences $\{f_n\}_{n=0}^\infty$, $\{g_n\}_{n=0}^\infty$, and $\{V_n\}_{n=0}^\infty$.

In the same way we say that the algorithm converges if $\sigma_n = \sup_{f\in\mathcal{F}} \dist(f,V_n) \to 0$ as $n \to \infty$ for each possible realization of the algorithm.
Since it is difficult to measure $\{\sigma_n\}_{n=0}^\infty$ due to the fact that the NGA does not directly calculate the distance to the constructed subspace $V_n$, we introduce the sequence $\{\tau_n\}_{n=0}^\infty$ that represents the performance of the algorithm:
\[
	\tau_n = \tau_n(\mathcal{F,X}) = \sup_{f \in \mathcal{F}} \n{\mathcal{R}_n(f)}.
\]
These values are computed automatically during the realization of the NGA since by construction $\n{\mathcal{R}_n(f_n)} = \sup_{f \in \mathcal{F}} \n{\mathcal{R}_n(f)} = \tau_n$, and thus, no additional computations are required. 
Note that unlike $\{\sigma_n\}_{n=0}^\infty$, the sequence $\{\tau_n\}_{n=0}^\infty$ is not necessarily monotone, however, we will see later that for any compact set $\mathcal{F} \subset \mathcal{X}$ both sequences converge to zero, i.e., the NGA converges.
To show the exact connection between $\sigma_n(\mathcal{F,X})$ and $\tau_n(\mathcal{F,X})$, we first need to estimate the norm of $\mathcal{R}_n$, which depends on the geometry of the space $\mathcal{X}$ that is represented by the modulus of smoothness $\rho(u)$ defined by~\eqref{modulus_of_smoothness}.
\begin{lemma}\label{estimate_norm_R_n}
Let $\mathcal{X}$ be a Banach space, then for any $n \ge 0$ we have
\[
	\n{\mathcal{R}_n}_{\mathcal{X}^*} \le R^n,
\]
where
\[
	R = R(\mathcal{X}) = 1 + \mu,
	\text{ with }
	\mu \in (0,1]
	\text{ such that }
	1 + \mu = 2\mu \, \rho(\mu^{-1}).
\]
Additionally, if $\mathcal{X}$ has the modulus of smoothness of power type $q$, we get
\[
	R(\mathcal{X}) \le \min\{1 + \mu_q, 2\},
	\text{ with }
	\mu_q > 0
	\text{ such that }
	1 + \mu_q = 2c_\rho \, \mu_q^{1-q}.
\]
\end{lemma}

\noindent
The above bound is quite pessimistic since it estimates the norm of $\mathcal{R}_n$ on the whole space $\mathcal{X}$ and thus has to accommodate the worst-case scenario.
While this estimate is attainable in non-smooth Banach spaces, it is generally not sharp.
For instance, in a Hilbert space $\mathcal{H}$, our lemma provides the estimate $R(\mathcal{H}) = (\sqrt{5} + 1)/2$ since the modulus of smoothness of a Hilbert space is $\rho_\mathcal{H}(u) = u^2/2$, however, the actual norm of the operator $\mathcal{R}_n$ in this case is $1$.
Additionally, in Section~\ref{section_numerics} we estimate the norm of operators $\mathcal{R}_n$ on constructed subspaces for concrete numerical examples, and demonstrate that the attainable norm is drastically smaller than Lemma~\ref{estimate_norm_R_n} suggests.

Even though Lemma~\ref{estimate_norm_R_n} shows that in a general Banach space $\mathcal{X}$ the sequence of norms $\{\n{\mathcal{R}_n}_{\mathcal{X}^*}\}_{n=0}^\infty$ might be unbounded, we only operate on the set $\mathcal{F} \subset \mathcal{X}$, thus we need an alternative estimate for the norm of $\mathcal{R}_n$ that is more appropriate for the analysis of the NGA.
In particular, we will show that on a compact set $\mathcal{F}$ the values of $\n{\mathcal{R}_n(f)}$ are uniformly bounded for all $f \in \mathcal{F}$ and $n \ge 0$.

Indeed, let $\{g_n\}_{n=0}^\infty$ be a reduced basis generated by the NGA for the compact set $\mathcal{F}$, and $V$ be a subspace such that $V = \spn\{g_0,g_1,g_2,\ldots\} \subset \mathcal{X}$.
Denote by $C_g = C_g(\mathcal{F,X})$ the basis constant of $\{g_n\}_{n=0}^\infty$ (since it is a Schauder basis for $V$).
Then, similarly to~\eqref{nga_R_n(V_n)=0}, by using semi-orthogonality of the basis we obtain for any $h = \sum_{k=0}^\infty \alpha_k g_k \in V$ and any $n \ge 1$
\[
	\mathcal{R}_n(h) = \mathcal{R}_n \br{\sum_{k=0}^\infty \alpha_k g_k}
	= \sum_{k=n}^\infty \alpha_k g_k
	= h - \sum_{k=0}^{n-1} \alpha_k g_k,
\]
i.e. on the subspace $V$ the operator $\mathcal{R}_n$ acts as the remainder of the $n$-th basis projection of $\{g_n\}_{n=0}^\infty$.
Therefore 
\[
	\n{\mathcal{R}_n}_{V^*} \le C_g + 1,
\]
and by combining this estimate with Lemma~\ref{estimate_norm_R_n}, we arrive at
\begin{equation}\label{estimate_norm_R_n_V}
	\n{\mathcal{R}_n}_{V^*} \le B_n^g = B_n^g(\mathcal{F,X}) = \min\{R(\mathcal{X})^n, \, C_g(\mathcal{F,X}) + 1\}.
\end{equation}
Even though we estimate the norm of operator $\mathcal{R}_n$ only on the subspace $V$, we remark that for any $f \in \mathcal{F}$ and any $n \ge 0$ we have
\[
	\n{\mathcal{R}_n(f)} \le \n{\mathcal{R}_n(f_n)} 
	\le \n{\mathcal{R}_n}_{V_{n+1}^*} \n{f_n} 
	\le \n{\mathcal{R}_n}_{V^*} \le C_g + 1.
\]

We now can establish the relation between $\sigma_n(\mathcal{F,X})$ and $\tau_n(\mathcal{F,X})$ (which are defined as $\sup_{f \in \mathcal{F}} \dist(f,V_n)$ and $\sup_{f \in \mathcal{F}} \n{\mathcal{R}_n(f)}$ respectively).
Note that for the OGA the greedy selection step provides $\sigma_n = \dist(f_n,V_n)$, which is not necessarily the case for the NGA, where instead we have $\tau_n = \n{\mathcal{R}_n(f_n)}$.
Evidently for any $f \in \mathcal{F}$ we have
\[
	\dist(f,V_n) \le \n{\mathcal{R}_n(f)} \le \sup_{f \in \mathcal{F}} \n{\mathcal{R}_n(f)} = \n{\mathcal{R}_n(f_n)}.
\]
On the other hand, by using the property $\mathcal{R}_n(V_n) = 0$ and the estimate~\eqref{estimate_norm_R_n_V}, we obtain
\[
	\n{\mathcal{R}_n(f_n)} = \n{\mathcal{R}_n(f_n - \proj(f_n,V_n))} 
	\le \n{\mathcal{R}_n}_{V^*} \dist(f_n,V_n)
	\le B_n^g \sup_{f \in \mathcal{F}} \dist(f,V_n).
\]
Therefore for any $n \ge 0$ we arrive at the relation
\begin{equation}\label{estimate_sigma_tau}
	\sigma_n(\mathcal{F,X}) 
	\le \tau_n(\mathcal{F,X}) 
	\le B_n^g(\mathcal{F,X}) \, \sigma_n(\mathcal{F,X}).
\end{equation}
Similarly, we can establish the uniform equivalence (i.e., for all $n \ge 0$) of $\n{\mathcal{R}_n(\cdot)}$ and $\dist(\cdot,V_n)$ on the subspace $V \subset \mathcal{X}$, since for each $n \ge 0$ we have for any $f \in V$
\begin{equation}\label{estimate_dist_R_n_V}
	\dist(f,V_n) \le \n{\mathcal{R}_n(f)}
	\le \n{\mathcal{R}_n}_{V^*} \n{f - \proj(f,V_n)}
	\le (C_g + 1) \, \dist(f,V_n).
\end{equation}

We now have the necessary ingredients to discuss the convergence results of the NGA.
First, note that the uniform equivalence of $\dist(\cdot,V_n)$ and $\n{\mathcal{R}_n(\cdot)}$ on $V$ guarantees the convergence of the NGA for any compact set $\mathcal{F} \subset \mathcal{X}$.
Indeed, if for some $\epsilon > 0$ and all $n \ge 0$ we have $\sigma_n(\mathcal{F,X}) > \epsilon$, then relations~\eqref{estimate_sigma_tau} and~\eqref{estimate_dist_R_n_V} provide
\[
	\epsilon < \sigma_n \le \tau_n = \n{\mathcal{R}_n(f_n)}
	\le (C_g+1) \, \dist(f_n,V_n),
\]
i.e., for any $m > n \ge 0$ we have $\n{f_n - f_m} \ge \epsilon / (C_g+1)$, which contradicts the compactness of $\mathcal{F}$.

Furthermore, it might seem that the relation~\eqref{estimate_dist_R_n_V} implies that $\n{\mathcal{R}_n(\cdot)}$ is a surrogate for $\dist(\cdot,V_n)$ (see~\eqref{ga_surrogate}), i.e., that the NGA is a weak version of the Orthogonal Greedy Algorithm.
In Lemma~\ref{lemma_nga!=ga} below we refute this hypothesis by providing an example of a Banach space $\mathcal{X}$ and a compact set $\mathcal{F}$ such that $\n{\mathcal{R}_n(\cdot)}$ is not uniformly equivalent to $\dist(\cdot,V_n)$ on $\mathcal{F}$, thus proving that the Natural Greedy Algorithm is a distinct algorithm.
\begin{lemma}\label{lemma_nga!=ga}
	There exists a Banach space $\mathcal{X}$ and a compact set $\mathcal{F} \subset \mathcal{X}$ such that
	\[
		\lim_{n \to \infty} \sup_{f \in \mathcal{F}} \frac{\n{\mathcal{R}_n(f)}}{\dist(f,V_n)} = \infty.
	\]
\end{lemma}

\noindent
Next, we estimate the rate of convergence of the NGA.
Similarly to the convergence results for the OGA, we provide the direct and delayed comparison between $\sigma_n(\mathcal{F,X})$ and $d_n(\mathcal{F,X})$ for the NGA.
\begin{theorem}\label{theorem_nga_direct}
For a compact set $\mathcal{F} \subset \mathcal{X}$ the Natural Greedy Algorithm provides
\[
	\sigma_n(\mathcal{F,X}) \le B_{n+1}^g(\mathcal{F,X}) \, \gamma_{2n+1}(\mathcal{X}) \, 2^{n+1} d_n(\mathcal{F,X}).
\]
\end{theorem}
\begin{corollary}
For a compact set $\mathcal{F} \subset L_p$ with any $1 \le p \le \infty$ the Natural Greedy Algorithm provides
\[
	\sigma_n(\mathcal{F},L_p) \le 2\sqrt{3} B_{n+1}^g(\mathcal{F,X}) \, n^{|\frac{1}{2}-\frac{1}{p}|} \, 2^n d_n(\mathcal{F},L_p).
\]
\end{corollary}
\begin{theorem}\label{theorem_nga_indirect}
For a compact set $\mathcal{F} \subset \mathcal{X}$ the Natural Greedy Algorithm provides
\[
	\sigma_n(\mathcal{F,X}) \le \sqrt{2} B^g_{n/2}(\mathcal{F,X}) \min_{0 < m < n} \gamma_{n+m}(\mathcal{X}) \, d_m^{1-m/n}(\mathcal{F,X}).
\]
\end{theorem}
\begin{corollary}
For a compact set $\mathcal{F} \subset L_p$ with any $1 \le p \le \infty$ the Natural Greedy Algorithm provides
\[
	\sigma_n(\mathcal{F},L_p) \le \sqrt{2} B^g_{n/2}(\mathcal{F,X}) \min_{0 < m < n} (n+m)^{|\frac{1}{2}-\frac{1}{p}|} \, d_m^{1-m/n}(\mathcal{F},L_p).
\]
In particular,
\[
	\sigma_{2n}(\mathcal{F},L_p) \le \sqrt{6} B^g_n(\mathcal{F,X}) \, n^{|\frac{1}{2}-\frac{1}{p}|} \sqrt{d_n(\mathcal{F},L_p)}.
\]
\end{corollary}

\noindent
Finally, we provide the estimates for convergence rates in special cases when the rate of decay of Kolmogorov $n$-widths is known.
\begin{theorem}\label{theorem_nga_special}
For a compact set $\mathcal{F} \subset \mathcal{X}$ the Natural Greedy Algorithm results in the following consequences:
\begin{enumerate}
	\item[\normalfont{1.}] If $d_n(\mathcal{F,X}) \le A \exp(-a n^\alpha)$ for some constants $0 < \alpha,a,A < \infty$ then
	\[
		\sigma_n(\mathcal{F,X}) \le B \gamma_{3n/2}(\mathcal{X}) \exp(-b n^\alpha),
	\]
	with $B = B(A,\mathcal{F,X}) = \sqrt{2A} \, B_{n/2}^g(\mathcal{F,X})$ and $b = b(\alpha,a) = 2^{-(\alpha+1)} a$.
	Moreover, factor $\gamma_{3n/2}(\mathcal{X})$ can be removed by decreasing the constant $b$.
	\item[\normalfont{2.}] If $d_n(\mathcal{F,X}) \le A n^{-\alpha}$ and $\gamma_n(\mathcal{X}) \le C n^{\mu}$ for some constants $0 < \alpha,A,C < \infty$ and $0 \le \mu \le \min\{\alpha,1/2\}$ then
	\[
		\sigma_n(\mathcal{F,X}) \le B (\ln n)^{\min\{\alpha,1/2\}} \, n^{-\alpha + \mu},
	\]
	with $B = B(\alpha,\mu,A,C)$.
\end{enumerate}
\end{theorem}

\noindent
Proofs of all stated results can be found in Appendix~\ref{section_proofs}.

\subsection{Comments on the Weak Natural Greedy Algorithm}
We conclude this section by pointing out that the NGA admits the same simplification steps as the OGA as discussed in Section~\ref{section_ga}.
Namely, one can consider a weak version of the greedy selection step by taking any such $f_n^w \in \mathcal{F}$ that 
\[
    \n{\mathcal{R}_n(f_n^w)} \ge \gamma \sup_{f\in\mathcal{F}} \n{\mathcal{R}_n(f)},
\]
with some weakness parameter $\gamma \in (0,1]$, or find surrogates $s_n : \mathcal{X} \to \mathbb{R}$ such that
\[
	c_s s_n(f) \le \n{\mathcal{R}_n(f)} \le C_s s_n(f) \ \text{ for any } f \in \mathcal{X} \text{ and } n \ge 0,
\]
with some constants $0 < c_s \le C_s < \infty$ and select $f_n^s = \operatorname{argmax}_{f\in\mathcal{F}} s_n(f)$.
Any such modification would qualify as the Weak NGA and all the results stated in this section would hold with the additional factor of $\gamma^{-1}$ or $C_s / c_s$ respectively.

\section{Other reduced bases methods}\label{section_other_rb}

In this section we briefly discuss two well-known methods for constructing reduced bases, namely, the Proper Orthogonal Decomposition and the Empirical Interpolation Method.

\subsection{The Proper Orthogonal Decomposition}
The Proper Orthogonal Decomposition (POD, also known as the Principal Component Analysis and the Karhunen--Lo\`eve expansion) is applicable to a finite set of discrete data.
Let $f_1, \ldots, f_M$ be vectors in $\mathbb{R}^N$ and let $F \in \mathbb{R}^{N \times M}$ be the matrix, whose columns are formed by the vectors $f_m$, $1 \le m \le M \le N$.
In order to obtain an $n$-dimensional approximating subspace for $F$, the POD performs the compact singular value decomposition of the matrix $F$, i.e.,
\[
	F = U \Sigma V,
	\text{ with }
	U \in \mathbb{R}^{N \times M},\ 
	\Sigma \in \mathbb{R}^{M \times M},
    \text{ and }
	V \in \mathbb{R}^{M \times M},
\]
and defines the reduced basis $h_0,\ldots,h_{n-1}$ to be the first $n \le M$ columns of matrix $U$, or, equivalently,
\[
	h_{k-1} = \sigma_k^{-1} \sum_{j=1}^N v_{kj} f_j,
	\quad
	1 \le k \le n.
\]
The constructed subspace $\spn\{h_0,\ldots,h_{n-1}\}$ is optimal for approximating the space $\spn\{f_1,\ldots,f_M\}$ in $L_2$-norm in the sense that the reduced basis $\{h_k\}_{k=0}^{n-1}$ consists of singular vectors corresponding to the $n$ largest singular values of $F$ and, thus, contains the most information of the matrix $F$.
For a more detailed discussion on the Proper Orthogonal Decomposition see, e.g.,~\cite{ch2000,wipe2002}.

\subsection{The Empirical Interpolation Method}
The Empirical Interpolation Method (EIM) is designed for approximating a compact set of functions $\mathcal{F} \subset L_\infty(\Omega)$ and can be viewed as a particular greedy algorithm since it utilizes the iterative greedy selection approach to construct the reduced basis.
For a given set of functions $h_0,\ldots,h_{n-1} \in L_\infty(\Omega)$ and points $z_0,\ldots,z_{n-1} \in \Omega$ define the operator $\mathcal{I}_n : L_\infty(\Omega) \to L_\infty(\Omega)$ given as
\[
	\mathcal{I}_n(f) = \left\{
	\begin{array}{ll}
		0 & n = 0,
		\\
		\sum_{k=0}^{n-1} \beta_k h_k & n > 0,
	\end{array}
	\right.
\]
where $\{\beta_k\}_{k=0}^{n-1}$ is the solution of the linear system
\begin{equation}\label{eim_linear_system}
	\sum_{k=0}^{n-1} \beta_k h_k(z_m) = f(z_m)
	\ \text{ for any }\ 
	0 \le m \le n-1.
\end{equation}
Then the EIM recursively constructs the sequence of interpolation points $\{z_n\}_{n=0}^\infty \in \Omega$ and the reduced basis sequence $\{h_n\}_{n=0}^\infty \in L_\infty(\Omega)$ via the process defined below.
\begin{definition}[{{\bf Empirical Interpolation Method}}]
For $\mathcal{F} \subset L_\infty(\Omega)$ consider the following iterative procedure: 
\begin{enumerate}
\item[{\bf Step ${\bm 0}$.}]
    find $f_0 = \operatorname{argmax}_{f\in\mathcal{F}} \n{f - \mathcal{I}_0(f)}_{L_\infty(\Omega)}$
    \medskip\\
    \phantom{find $f_0=\operatorname{argmax}_{f\in\mathcal{F}}$}
    and $z_0 = \mathop{\mathrm{argmax}}_{z \in \Omega} |\br{f_0 - \mathcal{I}_0(f_0)}(z)|$,
    \smallskip\\
    \phantom{find $f_0=\operatorname{argmax}_{f\in\mathcal{F}}$ and $z_0=\operatorname{argmax}_{z\in\Omega}$}
    and let $\displaystyle{h_0 = \frac{f_0 - \mathcal{I}_0(f_0)}{\br{f_0 - \mathcal{I}_0(f_0)}(z_0)}}$;
\item[{\bf Step ${\bm 1}$.}]
    find $f_1 = \operatorname{argmax}_{f\in\mathcal{F}} \n{f - \mathcal{I}_1(f)}_{L_\infty(\Omega)}$
    \medskip\\
    \phantom{find $f_1=\operatorname{argmax}_{f\in\mathcal{F}}$}
    and $z_1 = \mathop{\mathrm{argmax}}_{z \in \Omega} |\br{f_1 - \mathcal{I}_1(f_1)}(z)|$,
    \smallskip\\
    \phantom{find $f_1=\operatorname{argmax}_{f\in\mathcal{F}}$ and $z_1=\operatorname{argmax}_{z\in\Omega}$}
    and let $\displaystyle{h_1 = \frac{f_1 - \mathcal{I}_1(f_1)}{\br{f_1 - \mathcal{I}_1(f_1)}(z_1)}}$;
\item[${\dots}$]
\item[{\bf Step ${\bm n}$.}]
    find $f_n = \operatorname{argmax}_{f\in\mathcal{F}} \n{f - \mathcal{I}_n(f)}_{L_\infty(\Omega)}$
    \medskip\\
    \phantom{find $f_n=\operatorname{argmax}_{f\in\mathcal{F}}$}
    and $z_n = \mathop{\mathrm{argmax}}_{z \in \Omega} |\br{f_n - \mathcal{I}_n(f_n)}(z)|$,
    \smallskip\\
    \phantom{find $f_n=\operatorname{argmax}_{f\in\mathcal{F}}$ and $z_n=\operatorname{argmax}_{z\in\Omega}$}
    and let $\displaystyle{h_n = \frac{f_n - \mathcal{I}_n(f_n)}{\br{f_n - \mathcal{I}_n(f_n)}(z_n)}}$.
\end{enumerate}
\end{definition}

\noindent
Thus the constructed subspace $\spn\{h_0,\ldots,h_{n-1}\}$ aims to approximate the given set $\mathcal{F}$ in $L_\infty$-norm.
For more information on the EIM we refer the reader to the papers~\cite{paetal2004,paetal2007,paetal2007b}.

\subsection{The Natural Greedy Algorithm as a generalization of the Empirical Interpolation Method}
We prove here that the EIM is a particular realization of the NGA in $L_\infty$-space, in the sense that both algorithms produce the same approximating subspaces and the same (up to the sign) basis sequence.
More specifically, let $\{f_n\}_{n=0}^\infty$ be the sequence of elements selected by EIM from the set $\mathcal{F} \subset L_\infty(\Omega)$, and let $\{h_n\}_{n=0}^\infty$ be the corresponding basis sequence.
We will use induction to show that there is a viable realization of NGA which picks exactly the same elements $f_0,f_1,f_2,\ldots \in \mathcal{F}$ and that $\mathcal{R}_n(f) = f - \mathcal{I}_n(f)$ for any $f \in \mathcal{F}$ and any $n \ge 0$.

Indeed, at the first iteration EIM selects an element $f_0 \in \mathcal{F}$ and a point $z_0 \in \Omega$ given by
\[
	f_0 = \mathop{\mathrm{argmax}}_{f \in \mathcal{F}} \n{f}_{L_\infty(\Omega)},\ \ 
	z_0 = \mathop{\mathrm{argmax}}_{z \in \Omega} |f_0(z)|,
	\ \text{ and }\ 
	h_0 = \frac{f_0}{|f_0(z_0)|}.
\]
The same element $f_0$ can be selected by NGA as well, and by using the norming functional~\eqref{norming_functional_linf}, given in this case as
\[
	F_{g_0}^{L_\infty}(f) = \operatorname{sgn}(g_0(z_0)) \, f(z_0),
\]
with $g_0 = f_0 / \n{f_0}_{L_\infty(\Omega)} = \pm h_0$, we conclude that for any $f \in \mathcal{F}$
\[
	\mathcal{R}_1(f) = f - F_{g_0}^{L_\infty}(f) \, g_0
	= f - \frac{\operatorname{sgn}(f_0(z_0)) \, f(z_0)}{\n{f_0}_{L_\infty(\Omega)}} f_0
	= f - \mathcal{I}_1(f),
\]
which proves the base of induction.
Assume that the assumption holds for $m > 0$, i.e. NGA selected elements $f_0,\ldots,f_{m-1}$ and $\mathcal{R}_n(f) = f - \mathcal{I}_n(f)$ (and thus $g_n = \pm h_n$) for any $0 \le n \le m$.
Then the element $f_m$ given by
\[
	f_m = \mathop{\mathrm{argmax}}_{f \in \mathcal{F}} \n{f - \mathcal{I}_m(f)}_{L_\infty(\Omega)}
	= \mathop{\mathrm{argmax}}_{f \in \mathcal{F}} \n{\mathcal{R}_m(f)}_{L_\infty(\Omega)},
\]
can be selected by NGA, and hence $g_m = \mathcal{R}_m(f_m) / \n{\mathcal{R}_m(f_m)}_{L_\infty(\Omega)}$.
Let $z_m \in \Omega$ be the $(m+1)$-st interpolation point selected by EIM, then from the induction assumption we obtain
\[
	z_m = \mathop{\mathrm{argmax}}_{z \in \Omega} |\br{f_m - \mathcal{I}_m(f_m)}(z)|
	= \mathop{\mathrm{argmax}}_{z \in \Omega} |\mathcal{R}_m(f_m)(z)|
	= \mathop{\mathrm{argmax}}_{z \in \Omega} |g_m(z)|,
\]
and thus the norming functional $F_{g_m}^{L_\infty}$ is given by
\[
	F_{g_m}^{L_\infty}(f) = \operatorname{sgn}(g_m(z_m)) \, f(z_m).
\]
Therefore, from the semi-orthogonal relation~\eqref{nga_F_k(R_m)=0} between the reduced basis $g_0,\ldots,g_m$ constructed by NGA and the corresponding operator $\mathcal{R}_{m+1}$, we deduce that for any $f \in \mathcal{F}$
\[
	F_{g_n}^{L_\infty}(\mathcal{R}_{m+1}(f)) = \operatorname{sgn}(g_n(z_n)) \, \big(\mathcal{R}_{m+1}(f)\big)(z_n) = 0
	\ \text{ for any }\ 0 \le n \le m,
\]
i.e. $\mathcal{R}_{m+1}(f)(z_n) = 0$ for any $0 \le n \le m$.
Due to the fact that $\mathcal{R}_{m+1}(f) = f - \sum_{n=0}^m \alpha_n g_n$ for some coefficients $\{\alpha_n\}_{n=0}^m$, we deduce
\begin{equation}\label{nga_linear_system}
	\sum_{n=0}^k \alpha_n g_n(z_k) = f(z_k)
	\ \text{ for any }\ 
	0 \leq k \leq m.
\end{equation}
Since the vectors $f_0,\ldots,f_m$ are linearly independent and the systems~\eqref{eim_linear_system} and~\eqref{nga_linear_system} have the same the right-hand sides, we conclude that $\sum_{n=0}^m \alpha_n g_n = \sum_{n=0}^m \beta_n h_n$, i.e., both EIM and NGA construct the same approximation for an element $f \in \mathcal{F}$, and thus
\[
	\mathcal{R}_{m+1}(f) = f - \mathcal{I}_{m+1}(f),
\]
which proves the induction hypothesis.

Therefore, we have shown that in $L_\infty(\Omega)$, the EIM abd the NGA construct the same reduced basis (up to a sign, i.e., $h_n = \pm g_n$) and the same approximating subspaces, with the possible exception when the maximum of the norm $\n{f - \mathcal{I}_n(f)}_{L_\infty(\Omega)}$ is attained on multiple elements $f \in \mathcal{F}$ or the maximal value of $|h_n(z)|$ is attained at multiple points $z \in \Omega$.
We note that such situations are highly unlikely to occur naturally and in our numerical examples EIM and NGA generate identical approximating subspaces in $L_\infty$-spaces, which is why we do not compare algorithms in this setting.

We also note that even though the EIM has been studied extensively in various applications (see, e.g.,~\cite{paetal2007,paetal2007b,paetal2015}), it is only studied as an interpolation procedure in the sense that the only known results on the convergence properties of EIM are the ones that describe the interpolation properties of the basis $\{h_k\}_{k=0}^{n-1}$ with respect to the points $\{z_k\}_{k=0}^{n-1}$, and there are no results on the approximation properties of the corresponding subspaces $V_n = \spn\{h_0,\ldots,h_{n-1}\}$, which can be measured as
\[
	\sigma_n = \sigma_n(\mathcal{F},L_\infty(\Omega)) = \sup_{f\in\mathcal{F}} \n{f - \proj(f,V_n)}_{L_\infty(\Omega)}.
\]
However, since the EIM is a realization of the NGA, we deduce that the results from Section~\ref{section_nga} hold for the EIM as well.
In particular, for any compact set $\mathcal{F} \subset L_\infty(\Omega)$ we have $\sigma_n(\mathcal{F},L_\infty(\Omega)) \to 0$ as $n \to \infty$, and various comparisons between $\sigma_n(\mathcal{F},L_\infty(\Omega))$ and the Kolmogorov $n$-widths $d_n(\mathcal{F},L_\infty(\Omega))$ can also be stated.


\section{Numerical experiments}\label{section_numerics}
In this section we use several numerical examples to test and compare our proposed NGA with OGA, EIM, and POD for constructing reduced bases as described in Sections~\ref{section_ga}--\ref{section_other_rb}.
The experiments are performed in Octave~4.4.1 and the source code will be made available at the time of publication.

The approximation accuracy of an algorithm is measured in the following way: for a given Banach space $(\mathcal{X},\n{\cdot})$ and a compact set $\mathcal{F} \subset \mathcal{X}$ we construct a reduced basis $g_0,\ldots,g_{M-1}$ and then find for each element $f \in \mathcal{F}$ the norm of the remainder of the best approximation from $V_M = \spn\{g_0,\ldots,g_{M-1}\}$, given by the distance from $f$ to $V_M$
\[
	\dist(f,V_M) = \inf_{\alpha_0,\ldots,\alpha_{M-1} \in \mathbb{R}} \n{f - \sum_{n=0}^{M-1} \alpha_n g_n}.
\]
For completeness of the presented results we calculate and provide two approximation errors on the set $\mathcal{F}$: the average and maximal, i.e.,
\begin{align*}
	\text{error}_{\text{avg}} (\mathcal{F},V_M)
	&= \underset{f \in \mathcal{F}}{\text{avg}}\, \dist(f,V_M)
	\ \text{ and}
	\\
	\text{error}_{\text{max}} (\mathcal{F},V_M)
	&= \sup_{f \in \mathcal{F}} \dist(f,V_M)
\end{align*}
respectively.
In cases when the set $\mathcal{F}$ is infinite, we sample it to obtain a finite training set denoted $\mathcal{F}_{tr} \subset \mathcal{F}$ that is assumed to be well-representative of $\mathcal{F}$.
The approximation error in such cases is calculated on the training set $\mathcal{F}_{tr}$, namely the average and maximal errors given by
\begin{align*}
	\text{error}_{\text{avg}} (\mathcal{F}_{tr},V_M) 
	&= \frac{1}{|\mathcal{F}_{tr}|} \sum_{f \in \mathcal{F}_{tr}} \dist(f,V_M)
	\ \text{ and}
	\\
	\text{error}_{\text{max}} (\mathcal{F}_{tr},V_M) 
	&= \max_{f \in \mathcal{F}_{tr}} \dist(f,V_M)
\end{align*}
respectively.

\subsection{Computational complexity of reduced bases algorithms}
While the approximation accuracy is very likely to be the most important characteristic of a reduced basis, in some applications the choice of an algorithm is determined from the perspective of its computational complexity.
In this subsection we discuss such computational aspects of the various reduced basis algorithms.

We note that, unlike the greedy algorithms, both EIM and POD do not possess `weak' versions that simplify their realization (at least in their default formulations), hence we consider here the full greedy algorithms, i.e., the greedy step is performed for each element of the set $\mathcal{F}$.
In such case the NGA gains an additional advantage due to the fact that operators $\mathcal{R}_n$ defined by~\eqref{nga_R_n} can be represented in the hierarchical form, i.e.,
\[
	\mathcal{R}_n = r_{n-1} \circ \mathcal{R}_{n-1},
\]
thus allowing one to obtain the value of $\mathcal{R}_n(f)$ by applying the operator $r_{n-1}$ to the vector $\mathcal{R}_{n-1}(f)$ found on the previous iteration, which further reduces the computational cost of the algorithm.
Namely, on the $n$-th iteration of the NGA one has to compute a one-dimensional projection onto the previously constructed element $g_{n-1}$, as opposed to the OGA, where each iteration requires calculating a full $n$-dimensional projection onto $V_n$.

We compare the computational costs of the aforementioned algorithms in the following setting: let $\mathcal{F}$ be a compact subset of a Banach space $\mathcal{X}$, and $\mathcal{F}_{tr} \subset \mathcal{F}$ be a training set.
We consider a discretization $\mathcal{X}_h$ of the space $\mathcal{X}$ as required for the realization of POD.
Denote $N_{tr} = |\mathcal{F}_{tr}|$, $N_h = \dim \mathcal{X}_h$, and let $\varepsilon > 0$ be the accuracy with which every minimization problem in the realization of the OGA is solved.
Then, taking into account the hierarchical nature of operators $\mathcal{R}_n$, we estimate the order of the number of floating-point operations required by a realization of each algorithm, shown in the Table below:
\begin{center}
	\begin{tabular}{c|c|c}
	Algorithm & Cost of $m$-th iteration & Cost of $M$ first iterations
	\\\hline
	OGA & $(\log\varepsilon)^{m-1} \, N_{tr} N_h$ & $(\log\varepsilon)^{M-1} \, N_{tr} N_h$
	\\\hline
	NGA & $ N_{tr} N_h$ & $M  \, N_{tr} N_h$
	\\\hline
	EIM & $N_{tr} N_h$ & $M \, N_{tr} N_h$
	\\\hline
	POD & {\bf ---} & $N_{tr}^2 N_h$
	\end{tabular}
\end{center}

It might seem that the simplicity of each realization of the NGA comes at the cost of the reduced quality of the constructed approximating subspaces $V_n$; however, based on our observations, this does not appear to be the case and the approximation accuracy of NGA is comparable to other algorithms.
Note also that while POD and EIM appear to be computationally straightforward, they come with their own sets of drawbacks.
Namely, POD is not iterative and requires performing the singular value decomposition, which, depending on the matrix dimensions, can be very demanding in terms of memory allocation and might even be unfeasible.
The EIM is iterative and memory-efficient, however, on each iteration it requires finding the maximum of the function $|f - \mathcal{I}_n(f)|$ that is generally not convex and becomes more convoluted with each iteration.
Additionally, EIM depends on the smoothness of the training set and thus we expect it to be sensitive to the noisy data.
Finally, reduced bases constructed by POD and EIM aim to approximate in $L_2$-norm and $L_\infty$-norm respectively and offer no flexibility in that regard.

We also note that the stated computational cost estimates are quite theoretical and do not necessarily represent the real behavior of the algorithms in all settings.
In order to estimate the practical complexity we measure the time that each algorithm has spent on the central processing unit (e.g. cputime), which is directly proportional to the number of operations that must be performed in order to construct the reduced basis.

\noindent
Then to compare the effectiveness of the algorithms on concrete numerical examples we introduce the `quality' of a reduced basis~--- a characteristic that combines the approximation accuracy and the computational cost of constructing the basis.
Namely, we measure the approximation error provided by the reduced basis on the training set and calculate the `quality' of the reduced basis as the inverse of the product of the approximation error and the construction time, i.e.,
\[
	\text{quality}(\mathcal{F}_{tr},V_M) = \frac{1}{\text{error}(\mathcal{F}_{tr},V_M) \times \text{cputime}(V_M)}.
\]
Note that since we provide two values for the approximation error, two values for the `quality' are obtained for each algorithm, namely, the average and minimal, given by 
\begin{align*}
	\text{quality}_{\text{avg}} (\mathcal{F}_{tr},V_M)
	&= \frac{1}{\text{error}_{\text{avg}} (\mathcal{F}_{tr},V_M) 
	\times \text{cputime}(V_M)},
	\\
	\text{quality}_{\text{min}} (\mathcal{F}_{tr},V_M)
	&= \frac{1}{\text{error}_{\text{max}} (\mathcal{F}_{tr},V_M) 
	\times \text{cputime}(V_M)}
\end{align*}
respectively.

\subsection{Orthogonal Greedy Algorithm vs. Natural Greedy Algorithm}
In this subsection we concentrate on comparing the approximation properties of the OGA and the NGA but disregard for the discussion on computational cost until the next subsection.
The purpose of the presented numerical examples is to demonstrate that the reduced bases generated by these algorithms are comparably efficient in approximating the training set despite the differences in constructing the bases.

In order to obtain an unbiased comparison that is independent of the choice of training set, we approximate randomly generated synthetic data.
For each parameter we execute $10$ simulations and present the ratio of OGA/NGA approximation error distribution over those simulations.

Specifically, we consider a discrete Banach space $\ell_p^{(N_h)}$ of dimensionality $N_h = 1000$ with the canonical basis $\{e_k\}_{k=1}^{N_h}$, and a set $\{h_m\}_{m=1}^d$ of cardinality $d = 100$, randomly generated as 
\[
	h_m = \sum_{k=1}^{N_h} a_k^m e_k,
	\text{ where }\ 
	a_k^m \sim \mathcal{N}(0,1)
	\ \text{ for all }\ 
	1 \le k \le N_h,\ 1 \le m \le d,
\]
where $\mathcal{N}(0,1)$ denotes the standard normal distribution.
We then obtain the training set $\mathcal{F}_{tr} = \{f_n\}_{n=1}^{N_{tr}}$ by taking $N_{tr} = 1000$ linear combinations of vectors $h_1,\ldots,h_d$ with weights uniformly distributed on the interval $(-1,1)$, i.e.,
\[
	f_n = \sum_{m=1}^d c_m^n h_m,
	\text{ where }\ 
	c_m^n \sim \mathcal{U}(-1,1),
	\ \text{ for all }\ 
	1 \le m \le d,\ 1 \le n \le N_{tr}.
\]
Thus we generate a training set $\mathcal{F}_{tr} \in \ell_p^{(N_h)}$ (which consists of $N_{tr}$ vectors from the subspace $\spn\{h_1,\ldots,h_d\} \subset \ell_p^{(N_h)}$) and use it to construct reduced bases by applying the OGA and the NGA, and calculate the approximation errors (both average and maximal) on the training set after each iteration of the respective algorithms.
The approximation errors provided by the reduced basis constructed by the OGA is then divided by the errors provided by the reduced basis constructed by the NGA, and the graph of the ratios versus the cardinality of reduced bases for different values of $p$ is presented in Figure~\ref{fig:rnd}.

This experiment demonstrates how the geometry of the space affects the behavior of NGA.
In particular, we note that in the case $p = 3$ the approximating properties of NGA and OGA are virtually indistinguishable, due to the fact that the space $\ell_3$ is sufficiently close to the Hilbert space $\ell_2$, where the algorithms coincide.
In settings $p = 1, 5$, where the algorithms are more distinct, the distribution of the approximation errors still appears to be consistent in the sense that both algorithms perform similarly with NGA sometimes outperforming OGA.
In settings $p = 7, 10$ we observe that in terms of accuracy NGA is consistently outperformed by OGA but by a very slim margin.
Finally, in the case $p = \infty$ the ratio of the maximal approximation errors seems to be quite irregular, likely due to the non-smoothness of the space $\ell_\infty$ (even though it is not the issue in the case $p = 1$).

We also note that while the distribution of the ratio of the maximal errors appears to be influenced by the geometry of the space, the ratio of the average errors is consistent regardless of the setting.
We believe that this phenomenon is caused by the nature of operators $\mathcal{R}_n$: while they seem to be a reasonable substitute for computing projections for a large set of elements, the norm of the remainder of an individual element is not necessarily decreased, thus resulting in an irregular behavior of the maximal approximation error.

\begin{figure}
	\centering
	\includegraphics[width=\linewidth]{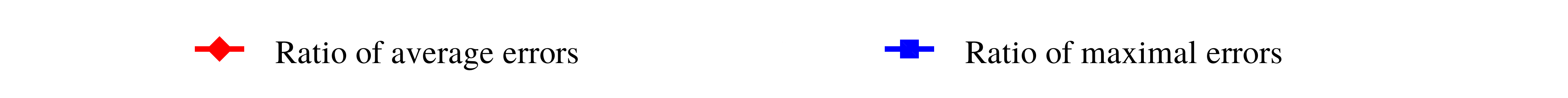}
	\begin{subfigure}{.49\linewidth}
		\includegraphics[width=\linewidth]{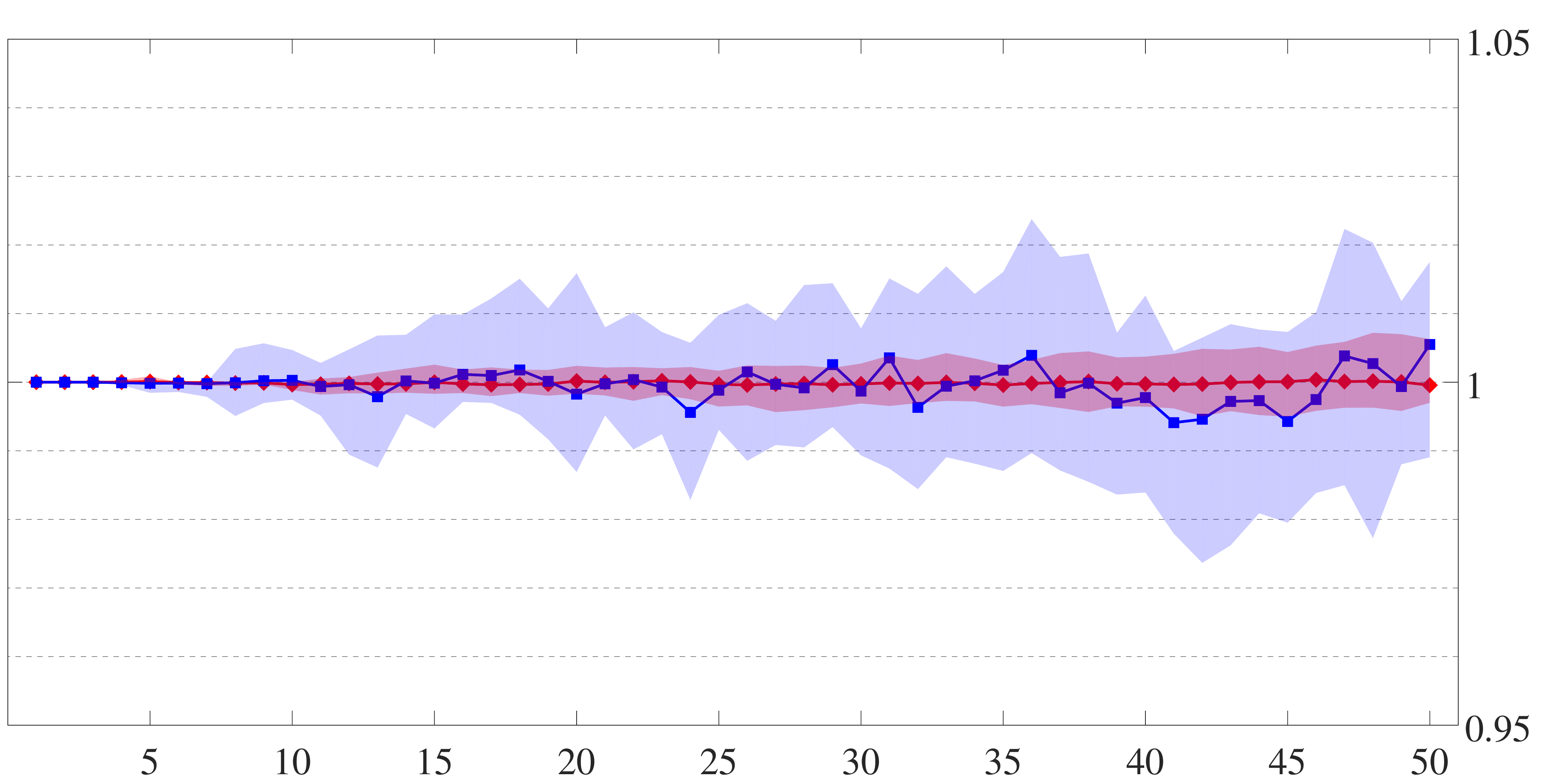}
		\caption{$p = 1$}
		\label{rnd1}
	\end{subfigure}
	\begin{subfigure}{.49\linewidth}
		\includegraphics[width=\linewidth]{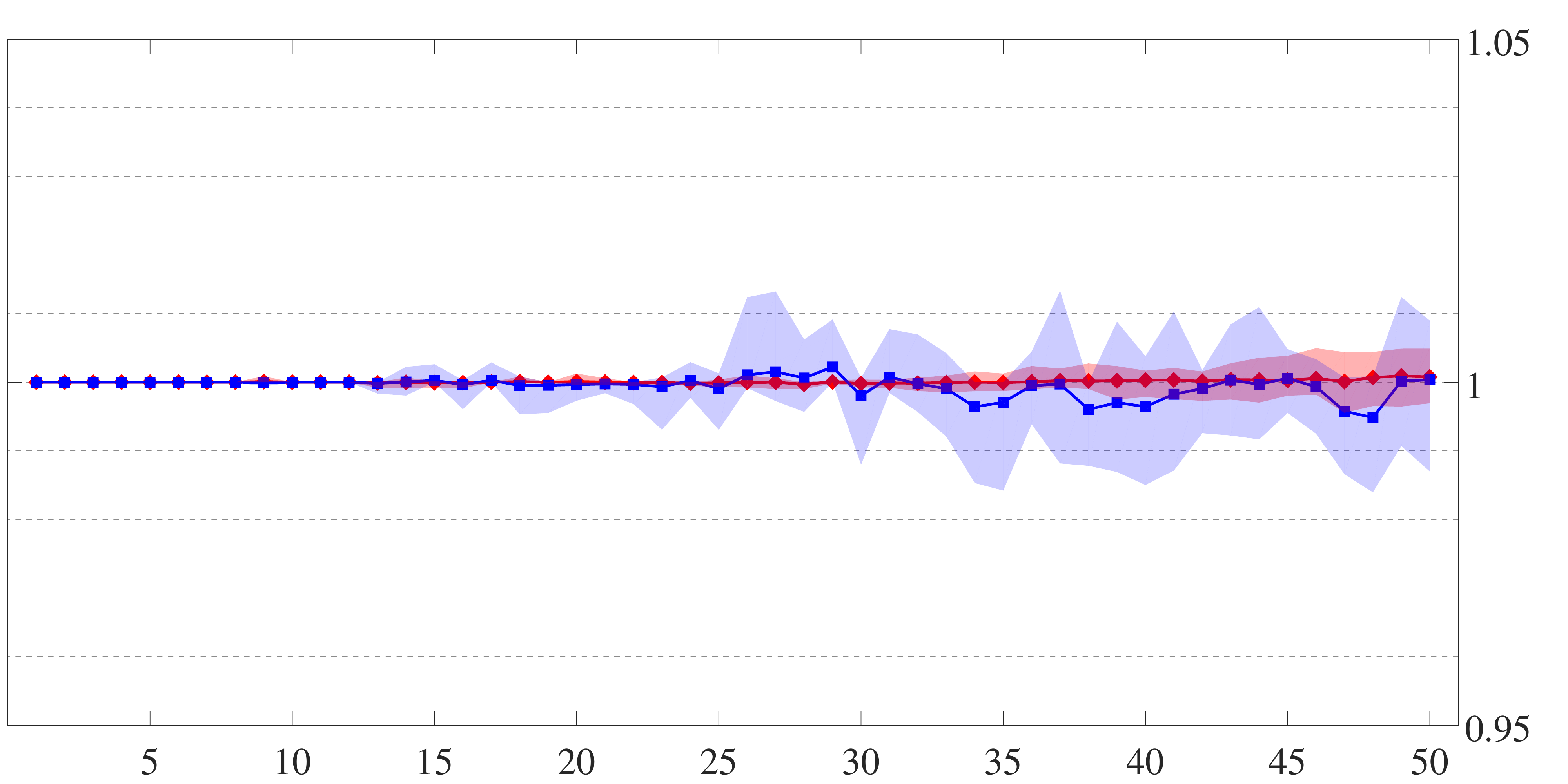}
		\caption{$p = 3$}
		\label{rnd3}
	\end{subfigure}
	\begin{subfigure}{.49\linewidth}
		\includegraphics[width=\linewidth]{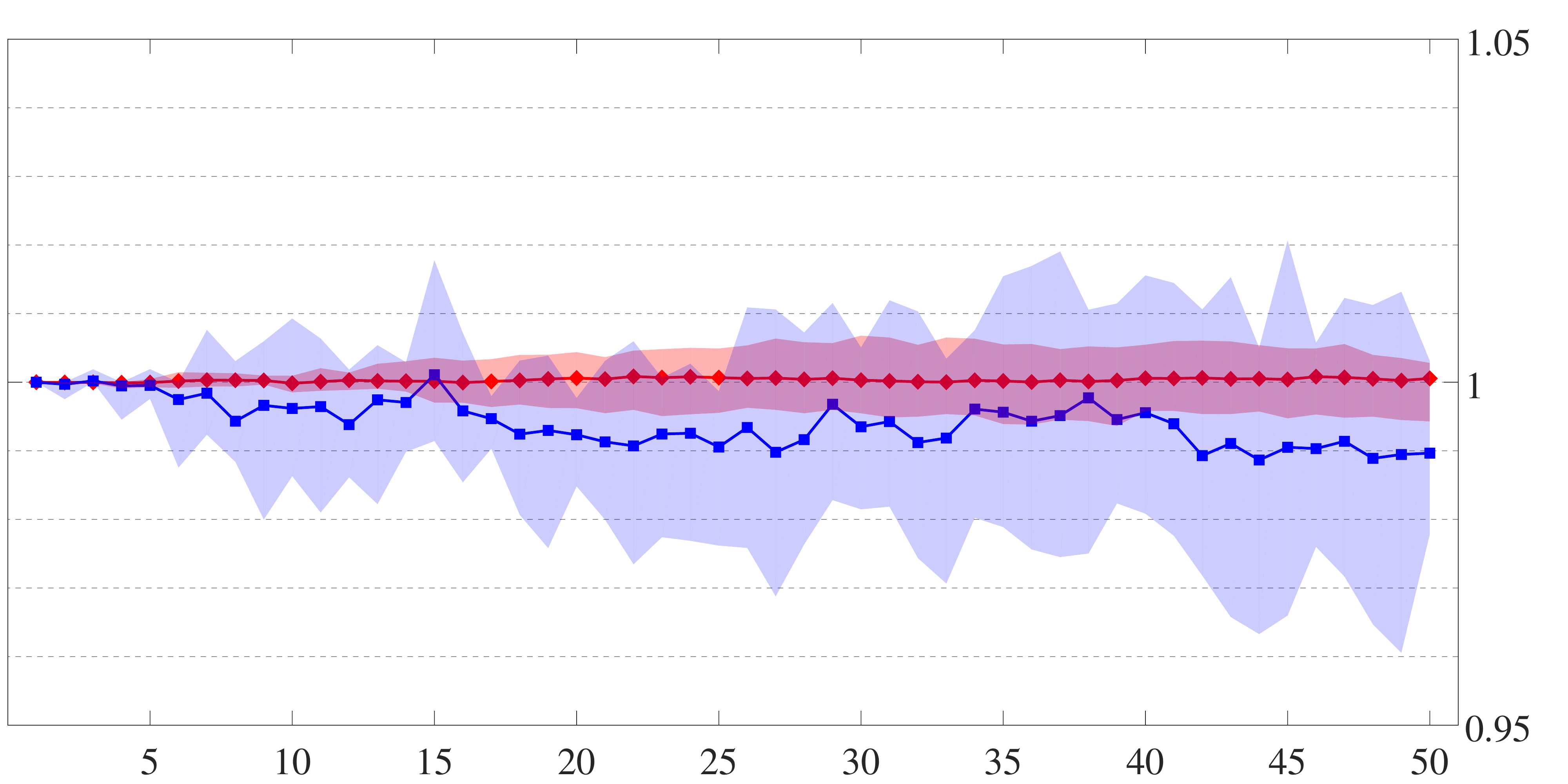}
		\caption{$p = 5$}
		\label{rnd5}
	\end{subfigure}
	\begin{subfigure}{.49\linewidth}
		\includegraphics[width=\linewidth]{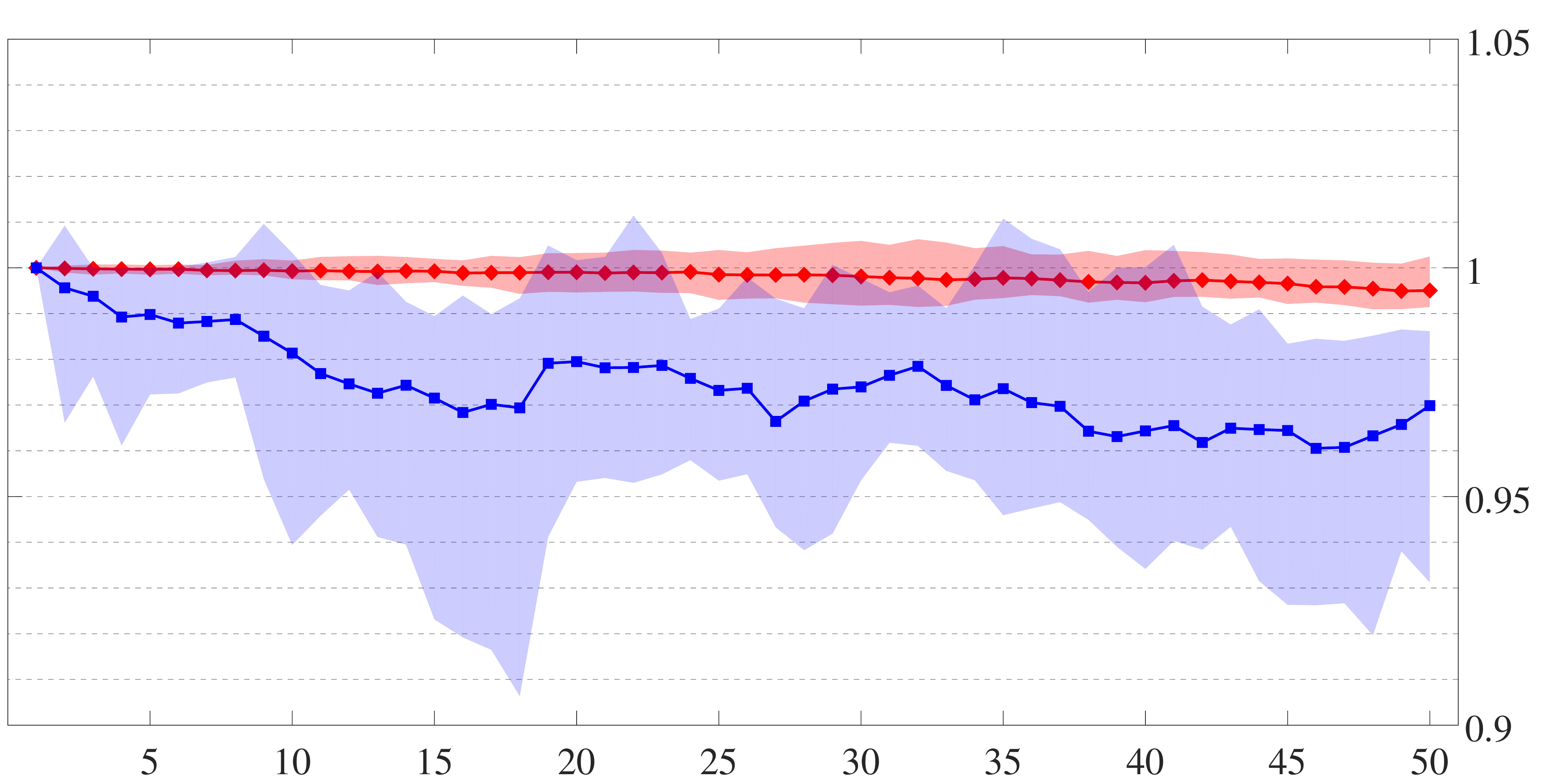}
		\caption{$p = 7$}
		\label{rnd7}
	\end{subfigure}
	\begin{subfigure}{.49\linewidth}
		\includegraphics[width=\linewidth]{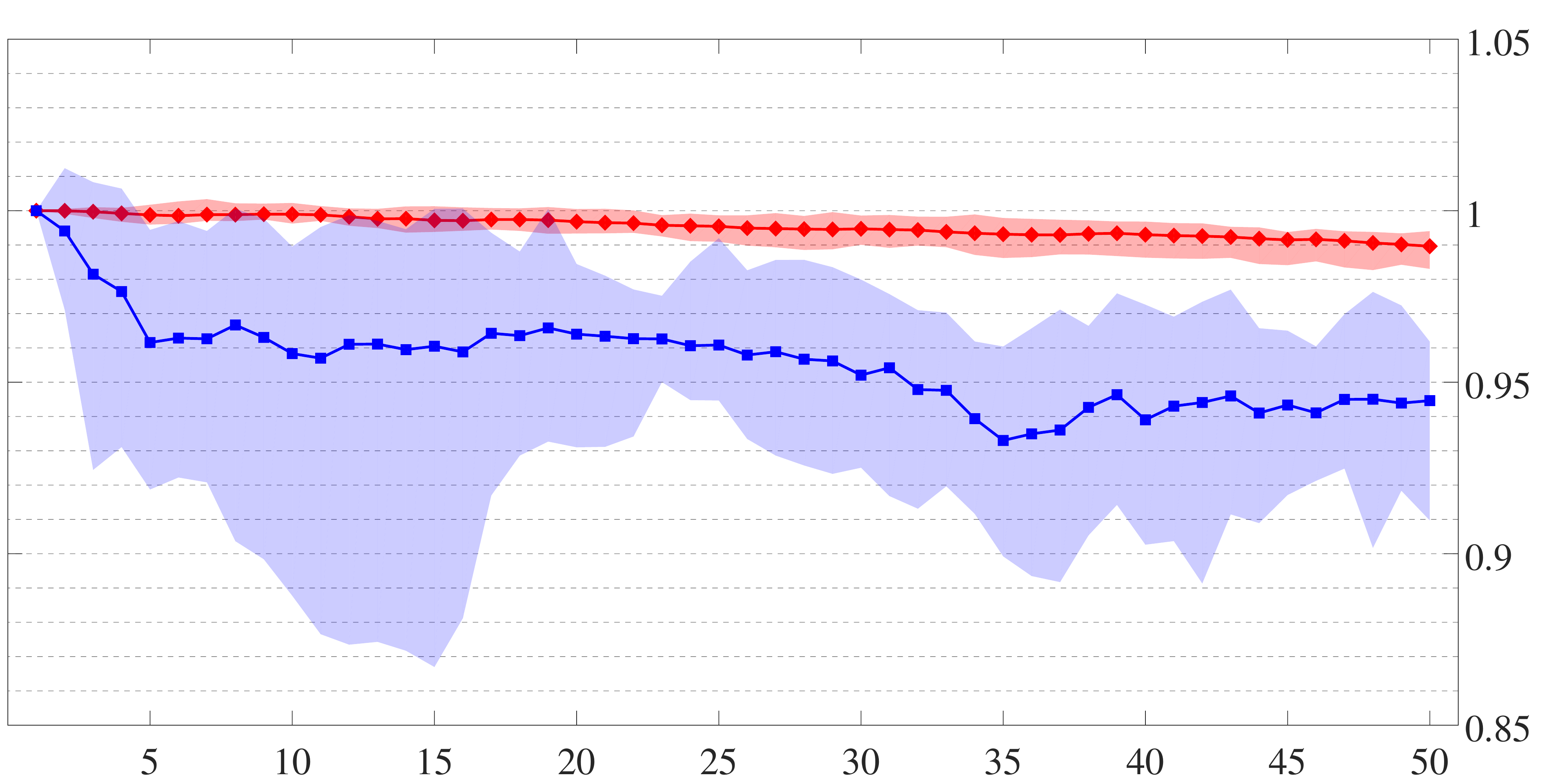}
		\caption{$p = 10$}
		\label{rnd10}
	\end{subfigure}
	\begin{subfigure}{.49\linewidth}
		\includegraphics[width=\linewidth]{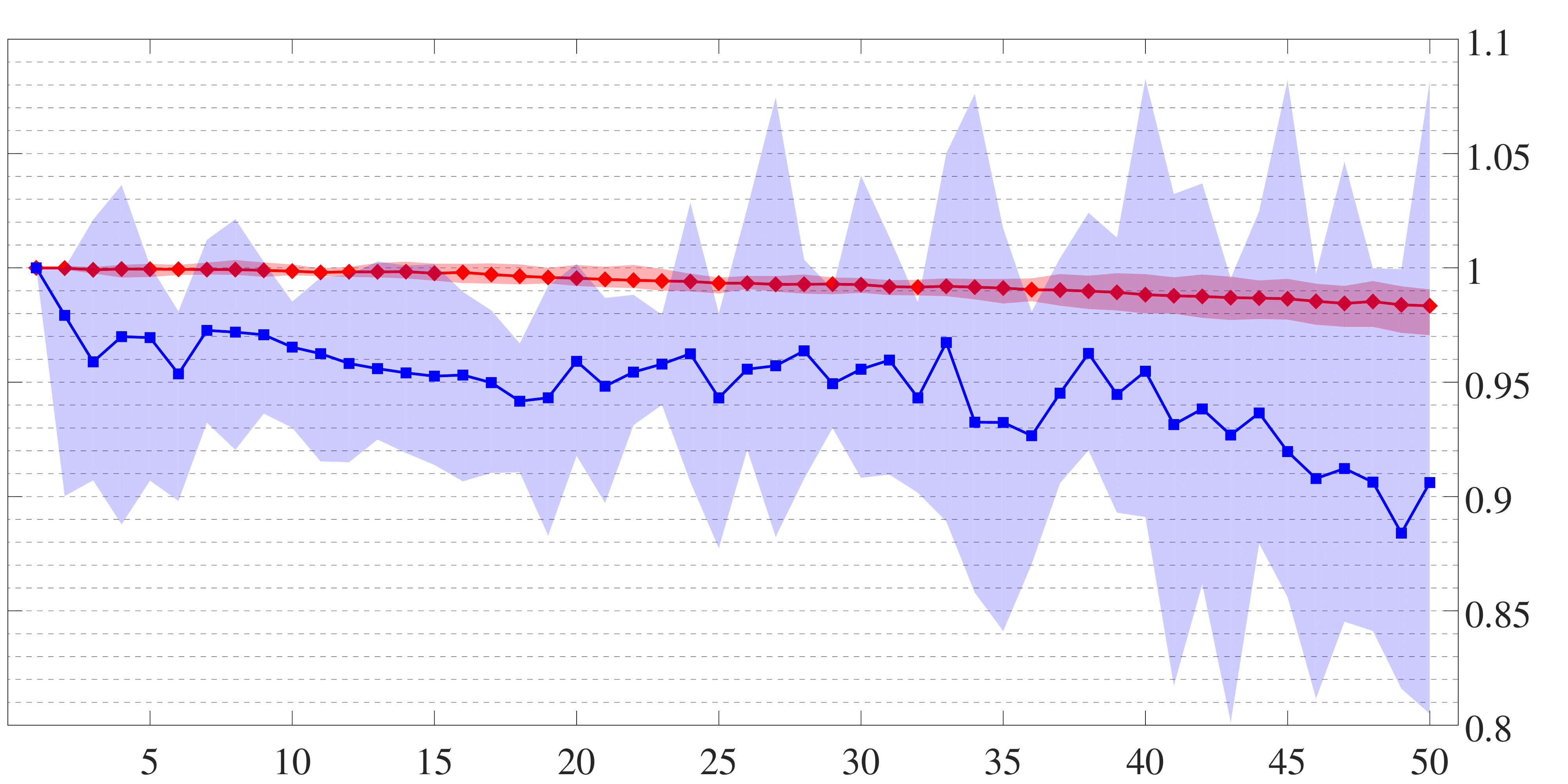}
		\caption{$p = \infty$}
		\label{rndinf}
	\end{subfigure}
	\caption{OGA/NGA approximation error ratio on a randomly generated data in $\ell_p$-norm with different values of $p$, distributed over $10$ simulations with $N_h = 1000,\ d = 100,\ N_{tr} = 1000$.}
	\label{fig:rnd}
\end{figure}

\subsection{Reduced bases in various $\boldsymbol{p}$-norms}
In the following examples we consider the problem of approximating a parametric family of functions.
More specifically, we are given a parametric family of functions $\mathcal{F}(x,\mu) \subset L_p(\Omega \times \mathcal{D})$ that depends on the parameter $\mu \in \mathcal{D} \subset \mathbb{R}^d$, $d \in \mathbb{N}$.
We then discretize the spatial domain $\Omega$ at $N_h$ uniform points and uniformly sample the parametric domain $\mathcal{D}$ at $N_{tr}$ points $\{\mu_1,\ldots,\mu_{N_{tr}}\}$ to form the discrete training set $\mathcal{F}_{tr}(x) = \{\mathcal{F}(x,\mu_n)\}_{n=1}^{N_{tr}}$ in $\ell_p^{(N_h)}$~--- a discretized version of the space $L_p(\Omega)$~--- and use it to construct the reduced bases.
The approximation error and the `quality' of each of the resulting reduced bases is calculated and presented.

In this subsection we observe the performance of the reduced bases constructed by OGA, NGA, EIM, and POD with respect to the $p$-norm for different values of $p$.
Due to the fact that OGA and NGA coincide when $p = 2$, and NGA and EIM coincide when $p = \infty$, we consider the values $p = 1, 3, 10$.

\begin{figure}
    \includegraphics[width=.49\linewidth]{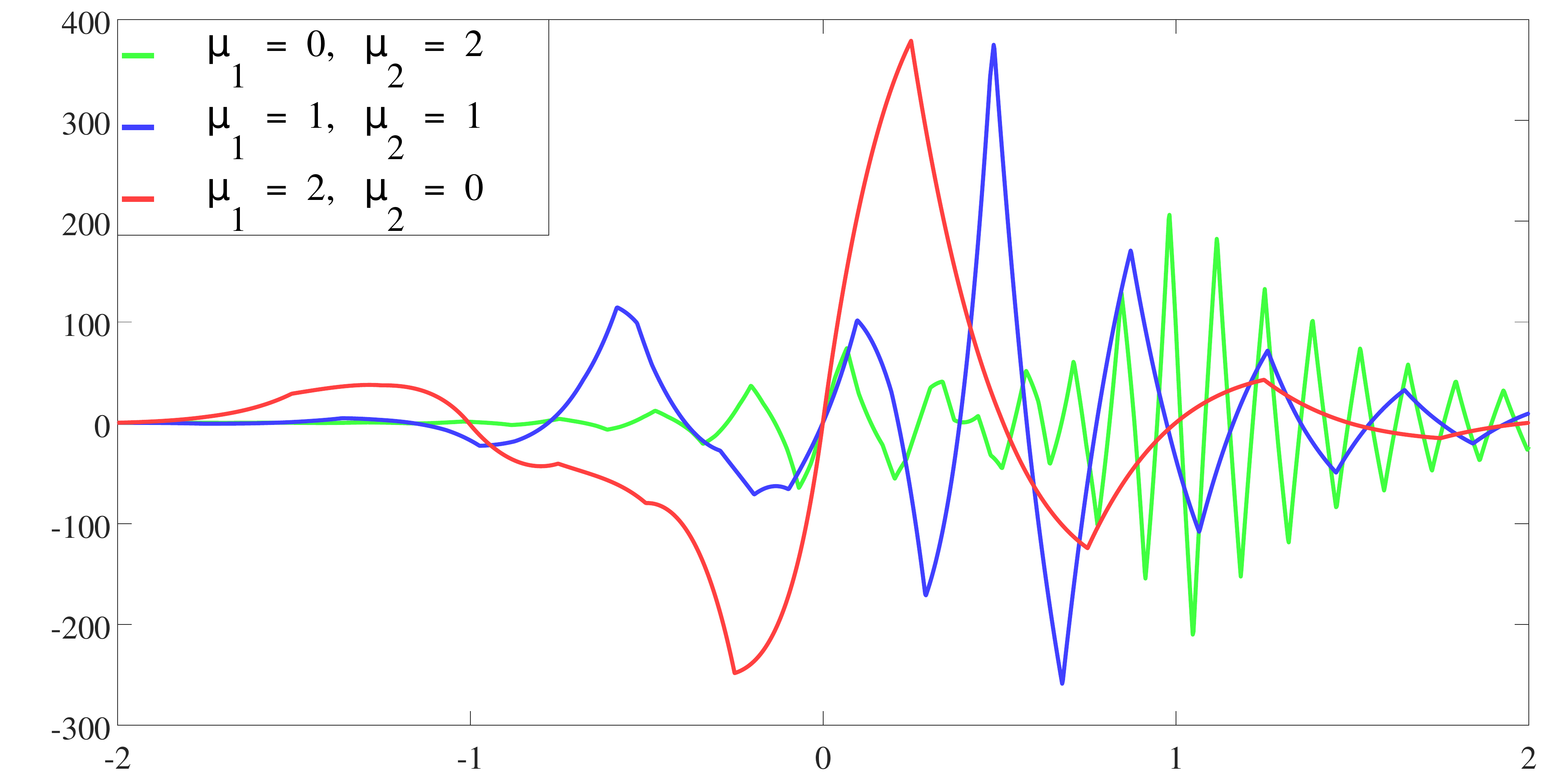}
    \includegraphics[width=.49\linewidth]{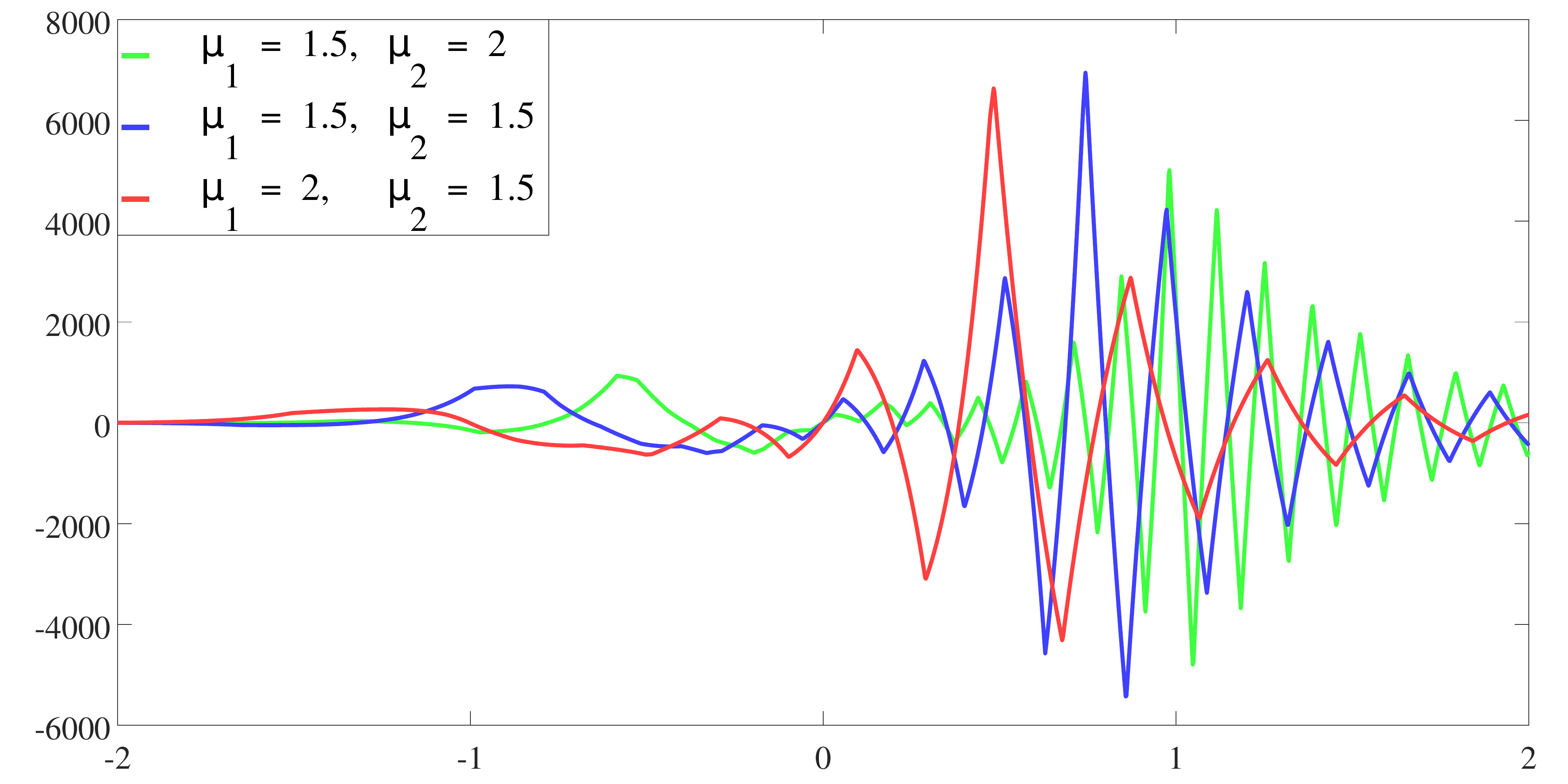}
    \caption{Elements of the parametric family $\mathcal{F}(x,\mu_1,\mu_2)$ defined by~\eqref{15d_F_tr}.}
    \label{fig:F_tr_parametric}
\end{figure}

\begin{figure}
    \hfill
    \begin{subfigure}{.49\linewidth}
        \resizebox{.9\linewidth}{!}{%
            \begin{tabular}{|c|c|c|c|c|}
                \hline$m$ & OGA & NGA & EIM & POD \\\hline
                3 & 2.482e-02 & 2.482e-02 & 2.518e-02 & 2.502e-02 \\\hline
                6 & 1.666e-02 & 1.666e-02 & 1.693e-02 & 1.722e-02 \\\hline
                9 & 1.166e-02 & 1.166e-02 & 1.159e-02 & 1.144e-02 \\\hline
                12 & 8.598e-03 & 8.598e-03 & 8.536e-03 & 8.238e-03 \\\hline
                15 & 6.560e-03 & 6.560e-03 & 6.393e-03 & 4.980e-03 \\\hline
                18 & 3.368e-03 & 3.368e-03 & 5.361e-03 & 3.117e-03 \\\hline
                21 & 1.811e-03 & 1.810e-03 & 2.581e-03 & 1.724e-03 \\\hline
                24 & 1.126e-03 & 9.912e-04 & 1.226e-03 & 8.468e-04 \\\hline
                27 & 6.041e-04 & 6.153e-04 & 6.430e-04 & 5.176e-04 \\\hline
                30 & 3.754e-04 & 3.673e-04 & 5.095e-04 & 3.505e-04 \\\hline
            \end{tabular}}
        \caption{average approximation error}
    \end{subfigure}
    \begin{subfigure}{.49\linewidth}
        \resizebox{.9\linewidth}{!}{%
            \begin{tabular}{|c|c|c|c|c|}
                \hline$m$ & OGA & NGA & EIM & POD \\\hline
                3 & 5.355e-01 & 5.355e-01 & 6.712e-01 & 3.707e-01 \\\hline
                6 & 3.143e-01 & 3.143e-01 & 4.327e-01 & 2.120e-01 \\\hline
                9 & 1.935e-01 & 1.935e-01 & 2.120e-01 & 1.080e-01 \\\hline
                12 & 1.025e-01 & 1.025e-01 & 9.914e-02 & 8.294e-02 \\\hline
                15 & 5.589e-02 & 5.589e-02 & 4.929e-02 & 4.238e-02 \\\hline
                18 & 3.046e-02 & 3.046e-02 & 4.921e-02 & 2.389e-02 \\\hline
                21 & 1.417e-02 & 1.417e-02 & 2.368e-02 & 1.282e-02 \\\hline
                24 & 9.037e-03 & 9.103e-03 & 9.398e-03 & 5.914e-03 \\\hline
                27 & 4.664e-03 & 4.847e-03 & 4.517e-03 & 2.434e-03 \\\hline
                30 & 2.911e-03 & 2.936e-03 & 3.493e-03 & 1.768e-03 \\\hline
            \end{tabular}}
        \caption{maximal approximation error}
    \end{subfigure}
    \begin{subfigure}{.49\linewidth}
        \includegraphics[width=\linewidth]{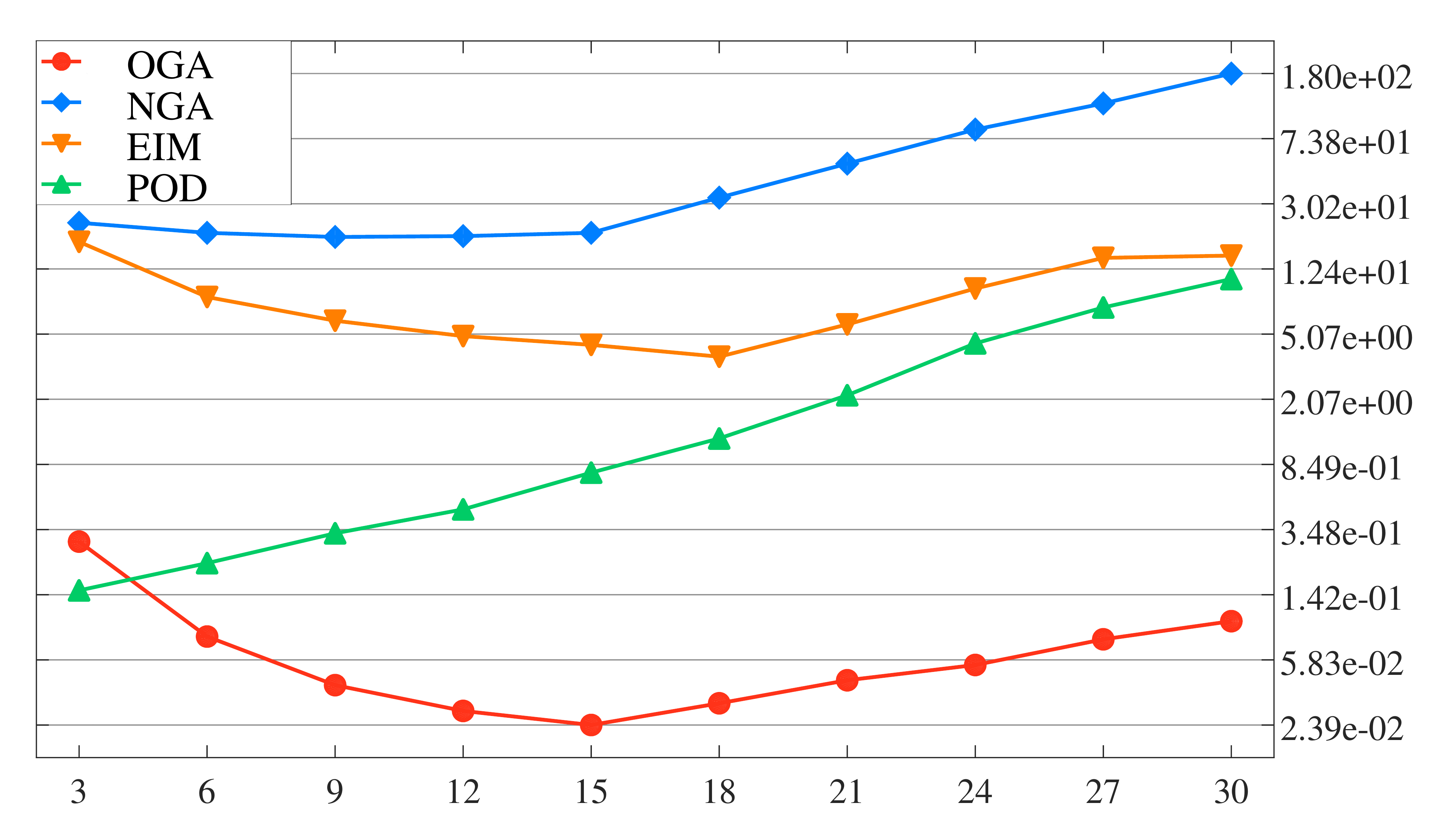}
        \caption{average quality}
    \end{subfigure}
    \begin{subfigure}{.49\linewidth}
        \includegraphics[width=\linewidth]{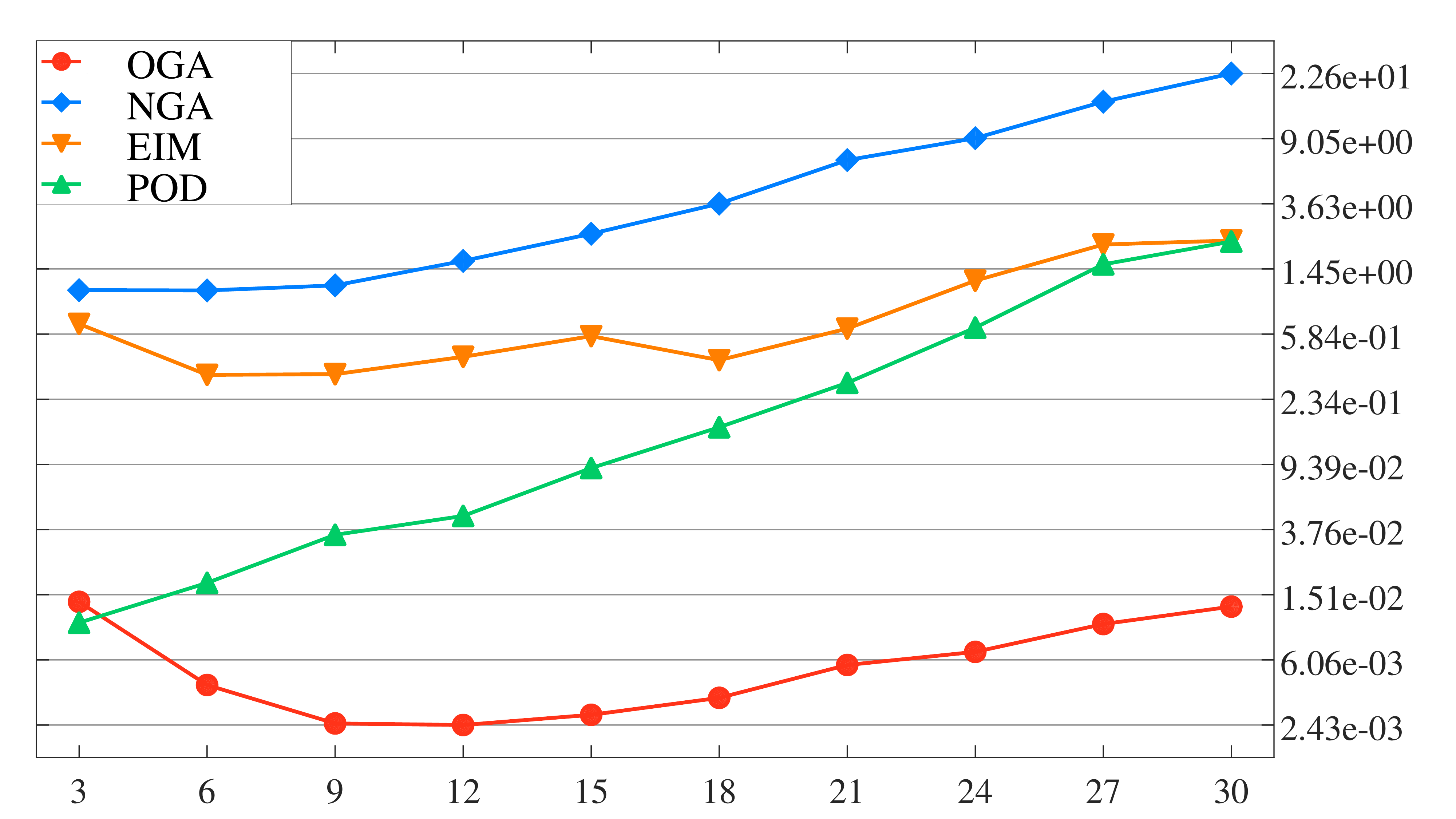}
        \caption{minimal quality}
    \end{subfigure}
    \caption{Performance of the OGA, NGA, EIM, and POD reduced bases for approximating the set $\mathcal{F}_{tr}$ (sampled from the parametric family~\eqref{15d_F_tr}) in $L_1[-2,2]$.}
    \label{15d1}
\end{figure}

\begin{figure}
    \hfill
    \begin{subfigure}{.49\linewidth}
        \resizebox{.9\linewidth}{!}{%
            \begin{tabular}{|c|c|c|c|c|}
                \hline$m$ & OGA & NGA & EIM & POD \\\hline
                3 & 1.981e-02 & 1.981e-02 & 2.074e-02 & 1.938e-02 \\\hline
                6 & 1.373e-02 & 1.381e-02 & 1.387e-02 & 1.236e-02 \\\hline
                9 & 9.833e-03 & 9.833e-03 & 9.428e-03 & 8.709e-03 \\\hline
                12 & 7.173e-03 & 7.173e-03 & 7.119e-03 & 5.994e-03 \\\hline
                15 & 5.527e-03 & 5.527e-03 & 5.359e-03 & 3.738e-03 \\\hline
                18 & 2.752e-03 & 2.752e-03 & 4.582e-03 & 2.264e-03 \\\hline
                21 & 1.548e-03 & 1.548e-03 & 2.070e-03 & 1.296e-03 \\\hline
                24 & 1.038e-03 & 1.165e-03 & 1.202e-03 & 7.577e-04 \\\hline
                27 & 7.000e-04 & 7.000e-04 & 6.965e-04 & 4.831e-04 \\\hline
                30 & 4.309e-04 & 4.309e-04 & 5.575e-04 & 3.303e-04 \\\hline
            \end{tabular}}
        \caption{average approximation error}
    \end{subfigure}
    \begin{subfigure}{.49\linewidth}
        \resizebox{.9\linewidth}{!}{%
            \begin{tabular}{|c|c|c|c|c|}
                \hline$m$ & OGA & NGA & EIM & POD \\\hline
                3 & 3.436e-01 & 3.436e-01 & 4.479e-01 & 3.074e-01 \\\hline
                6 & 2.247e-01 & 2.427e-01 & 2.824e-01 & 1.367e-01 \\\hline
                9 & 1.405e-01 & 1.405e-01 & 1.560e-01 & 7.830e-02 \\\hline
                12 & 8.080e-02 & 8.080e-02 & 8.126e-02 & 5.329e-02 \\\hline
                15 & 4.461e-02 & 4.461e-02 & 4.423e-02 & 3.054e-02 \\\hline
                18 & 2.831e-02 & 2.831e-02 & 4.418e-02 & 1.759e-02 \\\hline
                21 & 1.477e-02 & 1.477e-02 & 1.724e-02 & 9.377e-03 \\\hline
                24 & 8.611e-03 & 8.842e-03 & 9.223e-03 & 5.464e-03 \\\hline
                27 & 4.548e-03 & 4.548e-03 & 5.278e-03 & 2.647e-03 \\\hline
                30 & 3.342e-03 & 3.342e-03 & 3.967e-03 & 1.727e-03 \\\hline
            \end{tabular}}
        \caption{maximal approximation error}
    \end{subfigure}
    \begin{subfigure}{.49\linewidth}
        \includegraphics[width=\linewidth]{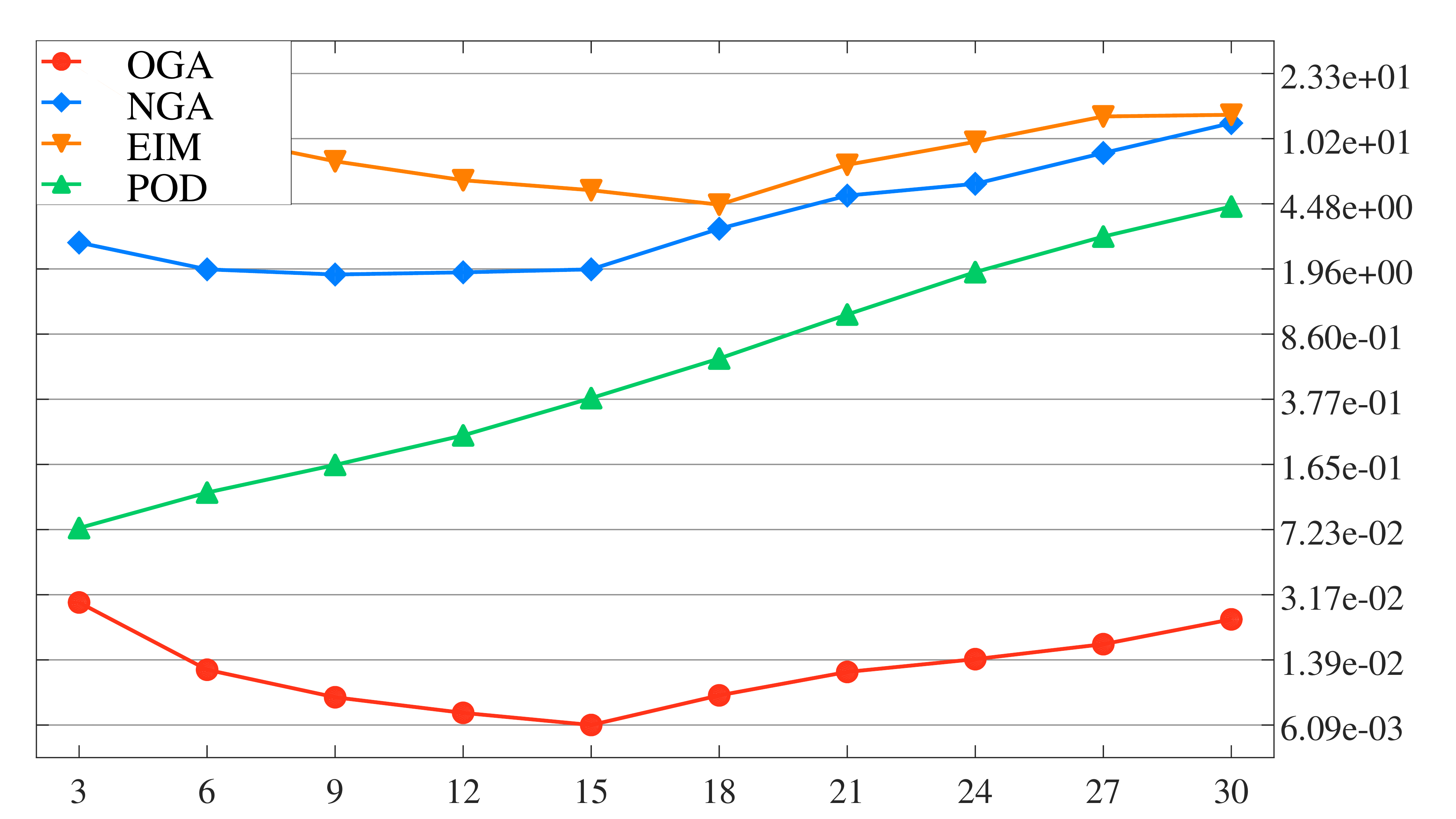}
        \caption{average quality}
    \end{subfigure}
    \begin{subfigure}{.49\linewidth}
        \includegraphics[width=\linewidth]{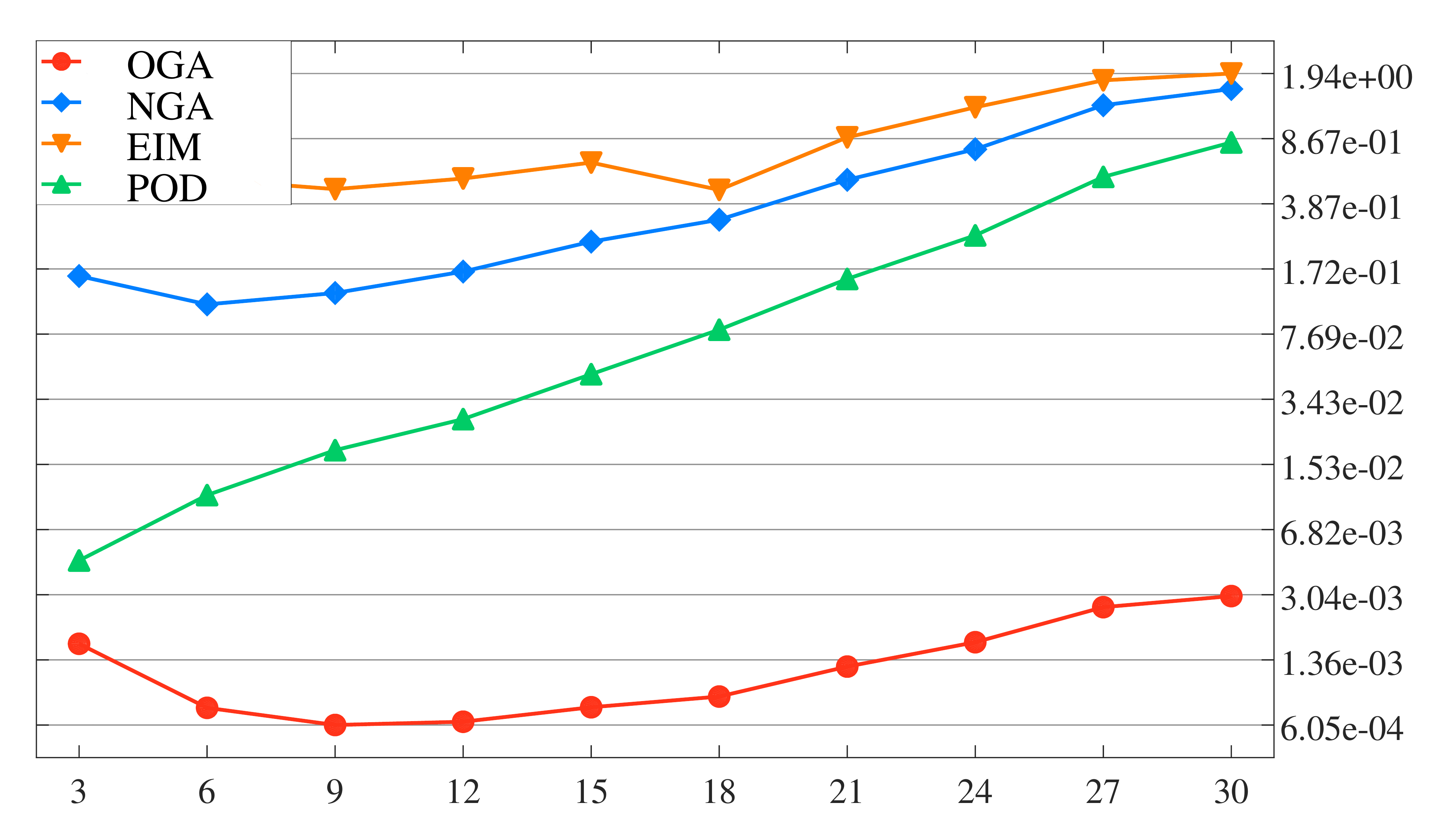}
        \caption{minimal quality}
    \end{subfigure}
    \caption{Performance of the OGA, NGA, EIM, and POD reduced bases for approximating the set $\mathcal{F}_{tr}$ (sampled from the parametric family~\eqref{15d_F_tr}) in $L_3[-2,2]$.}
    \label{15d3}
\end{figure}

\begin{figure}
    \hfill
    \begin{subfigure}{.49\linewidth}
        \resizebox{.9\linewidth}{!}{%
            \begin{tabular}{|c|c|c|c|c|}
                \hline$m$ & OGA & NGA & EIM & POD \\\hline
                3 & 1.822e-02 & 1.822e-02 & 1.822e-02 & 1.685e-02 \\\hline
                6 & 1.098e-02 & 1.100e-02 & 1.105e-02 & 1.007e-02 \\\hline
                9 & 7.359e-03 & 7.425e-03 & 7.235e-03 & 6.790e-03 \\\hline
                12 & 5.162e-03 & 5.107e-03 & 5.366e-03 & 4.411e-03 \\\hline
                15 & 4.193e-03 & 3.980e-03 & 4.060e-03 & 2.696e-03 \\\hline
                18 & 2.318e-03 & 2.290e-03 & 3.449e-03 & 1.567e-03 \\\hline
                21 & 1.151e-03 & 1.234e-03 & 1.487e-03 & 9.082e-04 \\\hline
                24 & 7.819e-04 & 9.499e-04 & 9.108e-04 & 5.816e-04 \\\hline
                27 & 5.322e-04 & 5.177e-04 & 5.667e-04 & 3.859e-04 \\\hline
                30 & 3.846e-04 & 3.569e-04 & 4.490e-04 & 2.591e-04 \\\hline
            \end{tabular}}
        \caption{average approximation error}
    \end{subfigure}
    \begin{subfigure}{.49\linewidth}
        \resizebox{.9\linewidth}{!}{%
            \begin{tabular}{|c|c|c|c|c|}
                \hline$m$ & OGA & NGA & EIM & POD \\\hline
                3 & 3.268e-01 & 3.268e-01 & 3.268e-01 & 2.451e-01 \\\hline
                6 & 1.671e-01 & 1.819e-01 & 1.929e-01 & 1.090e-01 \\\hline
                9 & 1.054e-01 & 1.194e-01 & 1.085e-01 & 5.987e-02 \\\hline
                12 & 6.762e-02 & 6.704e-02 & 5.834e-02 & 3.586e-02 \\\hline
                15 & 3.459e-02 & 3.770e-02 & 3.460e-02 & 2.129e-02 \\\hline
                18 & 2.054e-02 & 2.054e-02 & 3.456e-02 & 1.341e-02 \\\hline
                21 & 1.179e-02 & 1.141e-02 & 1.191e-02 & 7.064e-03 \\\hline
                24 & 5.508e-03 & 7.440e-03 & 7.312e-03 & 4.243e-03 \\\hline
                27 & 3.434e-03 & 3.552e-03 & 4.462e-03 & 2.259e-03 \\\hline
                30 & 2.382e-03 & 2.678e-03 & 3.265e-03 & 1.512e-03 \\\hline
             \end{tabular}}
        \caption{maximal approximation error}
    \end{subfigure}
    \begin{subfigure}{.49\linewidth}
        \includegraphics[width=\linewidth]{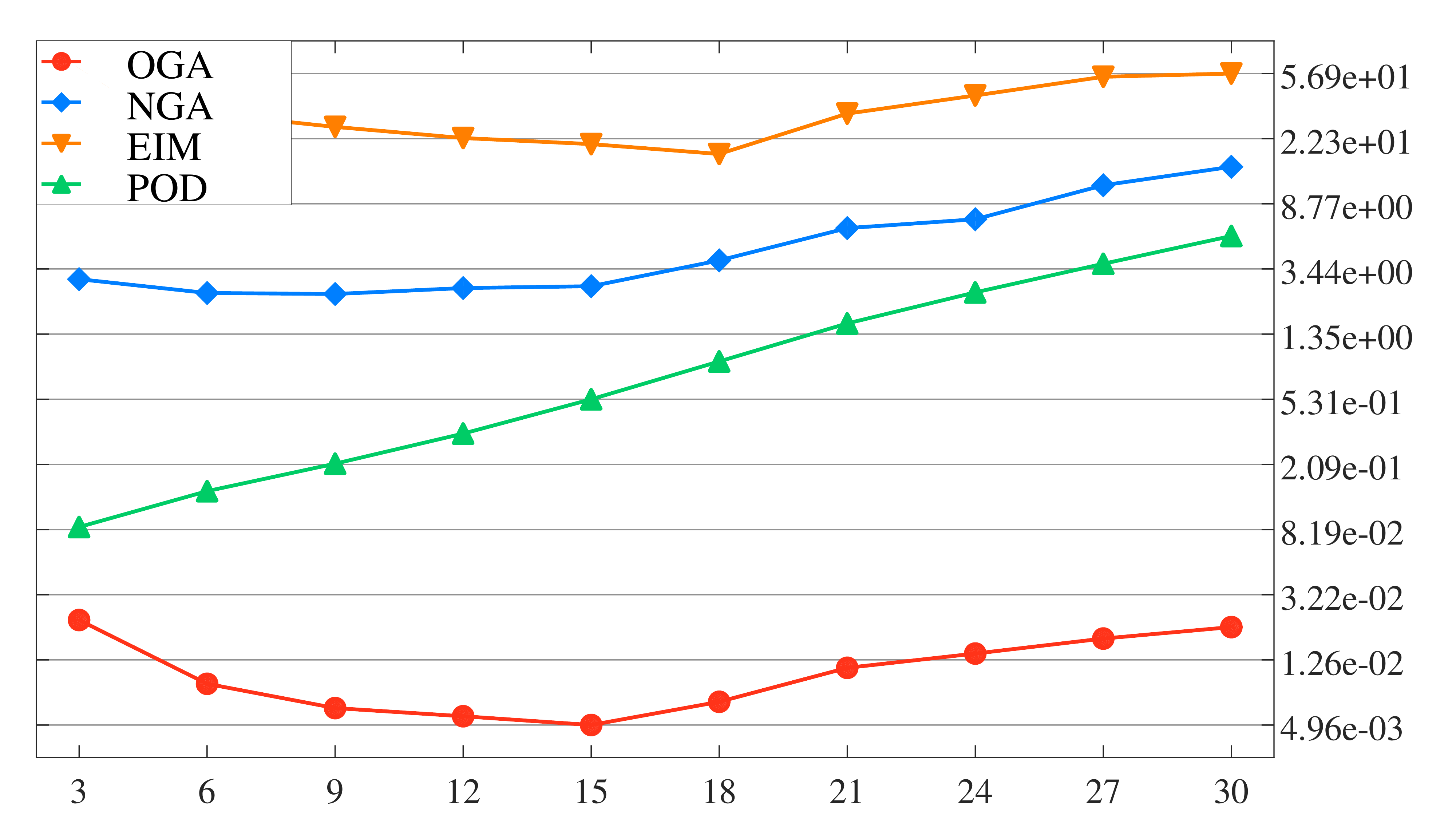}
        \caption{average quality}
    \end{subfigure}
    \begin{subfigure}{.49\linewidth}
        \includegraphics[width=\linewidth]{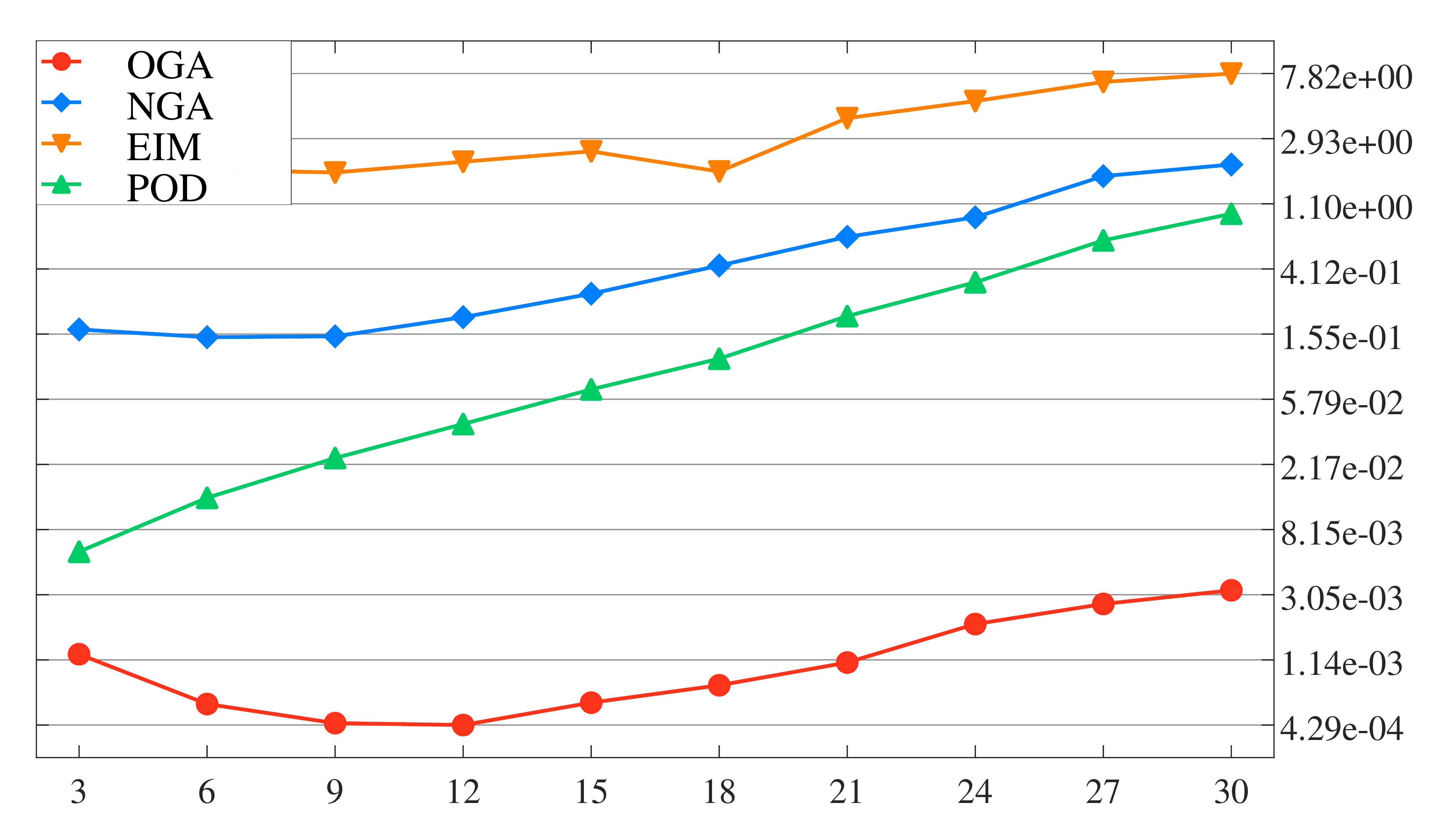}
        \caption{minimal quality}
    \end{subfigure}
    \caption{Performance of the OGA, NGA, EIM, and POD reduced bases for approximating the set $\mathcal{F}_{tr}$ (sampled from the parametric family~\eqref{15d_F_tr}) in $L_{10}[-2,2]$.}
    \label{15d10}
\end{figure}

Consider the parametric family $\mathcal{F} \subset C(\Omega \times \mathcal{D})$ given by
\begin{align}
    \label{15d_F_tr}
	\mathcal{F}(x,\mu_1,\mu_2) = e^{x + 2\mu_1 + 3\mu_2}
	    \bigg( &\arcsin\Big( \sin \big( 2\pi e^{\mu_1} \, x \big) \Big) \, e^{-\pi \left|x-\frac{\mu_1}{2}\right|}
    \\
    \nonumber
	&\ + \arcsin\Big( \sin \big( e^{\pi-\mu_2} \, x \big) \Big) \, e^{-\pi \left| x + \frac{\mu_2}{2} \right|} \bigg),
\end{align}
where $x \in \Omega = [-2,2]$ and $(\mu_1,\mu_2) \in \mathcal{D} = [0,2] \times [0,2]$.
Some elements of the family $\mathcal{F}$ are presented in Figure~\ref{fig:F_tr_parametric}.
We generate the discrete training set $\mathcal{F}_{tr}$ by sampling the spatial domain $\Omega$ at $N_h = 100,000$ uniform points, and the parametric domain $\mathcal{D}$ on the uniform grid of the size $N_{tr} = 32 \times 32$.
Approximation errors and the quality of reduced bases for each of the cases $p = 1,3,10$ are reported in Figures~\ref{15d1}, \ref{15d3}, and~\ref{15d10} respectively.

We note that in the case $p = 1$, presented in Figure~\ref{15d1}, all algorithms perform similarly in terms of approximation accuracy, however the computational complexity of NGA is the smallest (even smaller than that of EIM) due to the simplicity of norming functionals in case $p = 1$ given by the formula~\eqref{norming_functional_l1}.

In case $p = 3$, presented in Figure~\ref{15d3}, the approximation accuracy of the reduced bases is still similar, however in terms of `quality' NGA is slightly outperformed by EIM.

Note that both OGA and NGA construct the same reduced basis for this problem setting due to the fact that the geometry of the space $L_3[-2,2]$ is sufficiently close to that of the Hilbert space $L_2[-2,2]$.
In case $p = 10$, presented in Figure~\ref{15d10}, NGA loses in `quality' to both EIM and POD, likely due to the more complex computation of norming functionals.

\subsection{Reduced bases for perturbed data}
In this subsection we demonstrate how the reduced basis algorithms are affected by the `noisy' data.
Namely, in each of the following examples we construct reduced bases for a parametric family and for a perturbed version of the same family, and observe how the various perturbations affect the performance of algorithms.
Such formulation is natural in practical applications since the available data is typically measured not precisely but with some inaccuracies.
In particular, we will see that greedy algorithms (OGA and NGA) appear to be more robust to perturbations in terms of approximation accuracy.
\\
As shown in the previous subsection, the approximation properties of the reduced bases are not changed significantly for the various values of $p$; thus in this section we restrict ourselves to the case $p = 1$ for the simplicity of calculating the approximation error.

\subsubsection*{One-dimensional parametric family}
\begin{figure}
    \includegraphics[width=.49\linewidth]{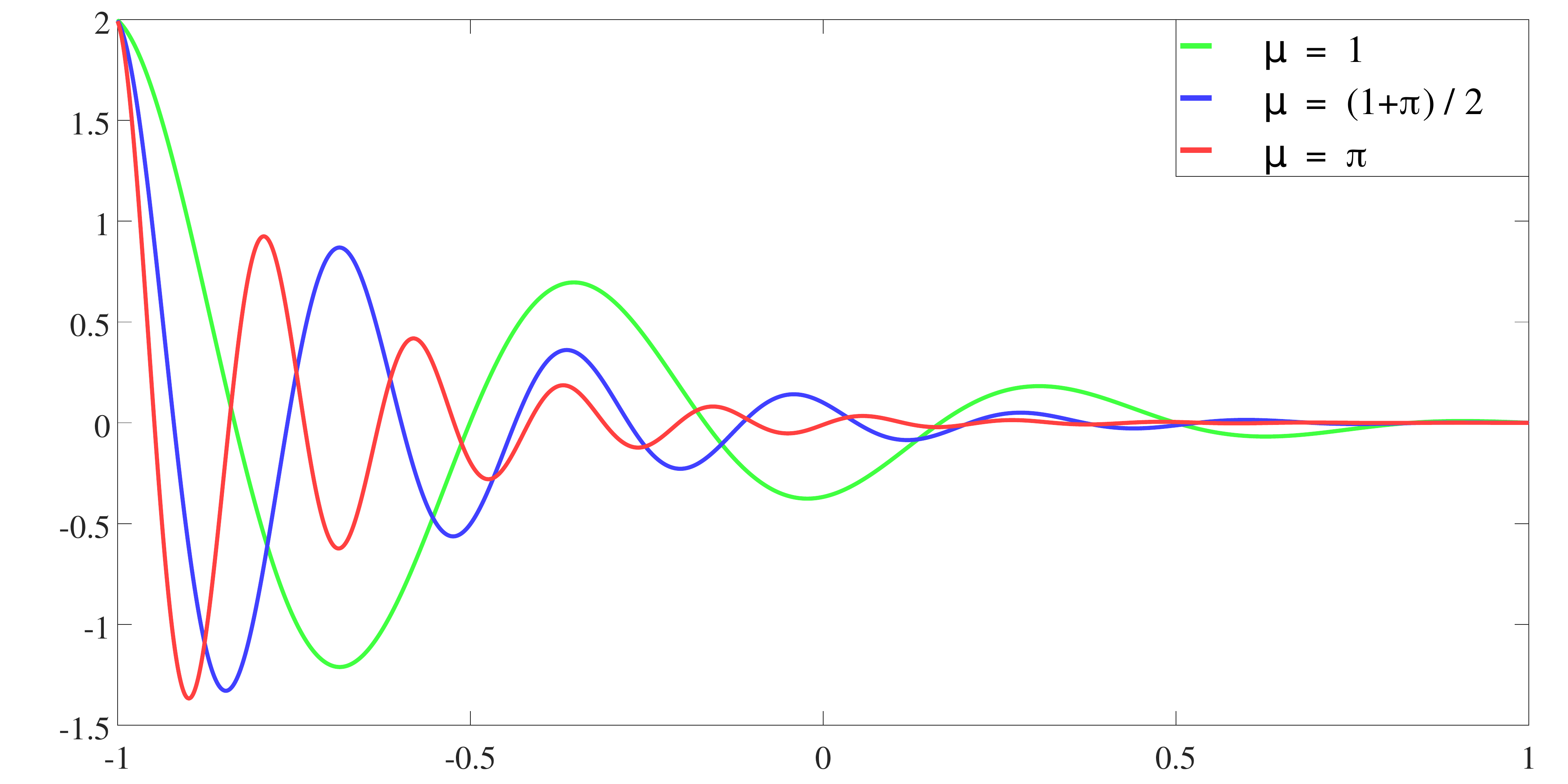}
    \includegraphics[width=.49\linewidth]{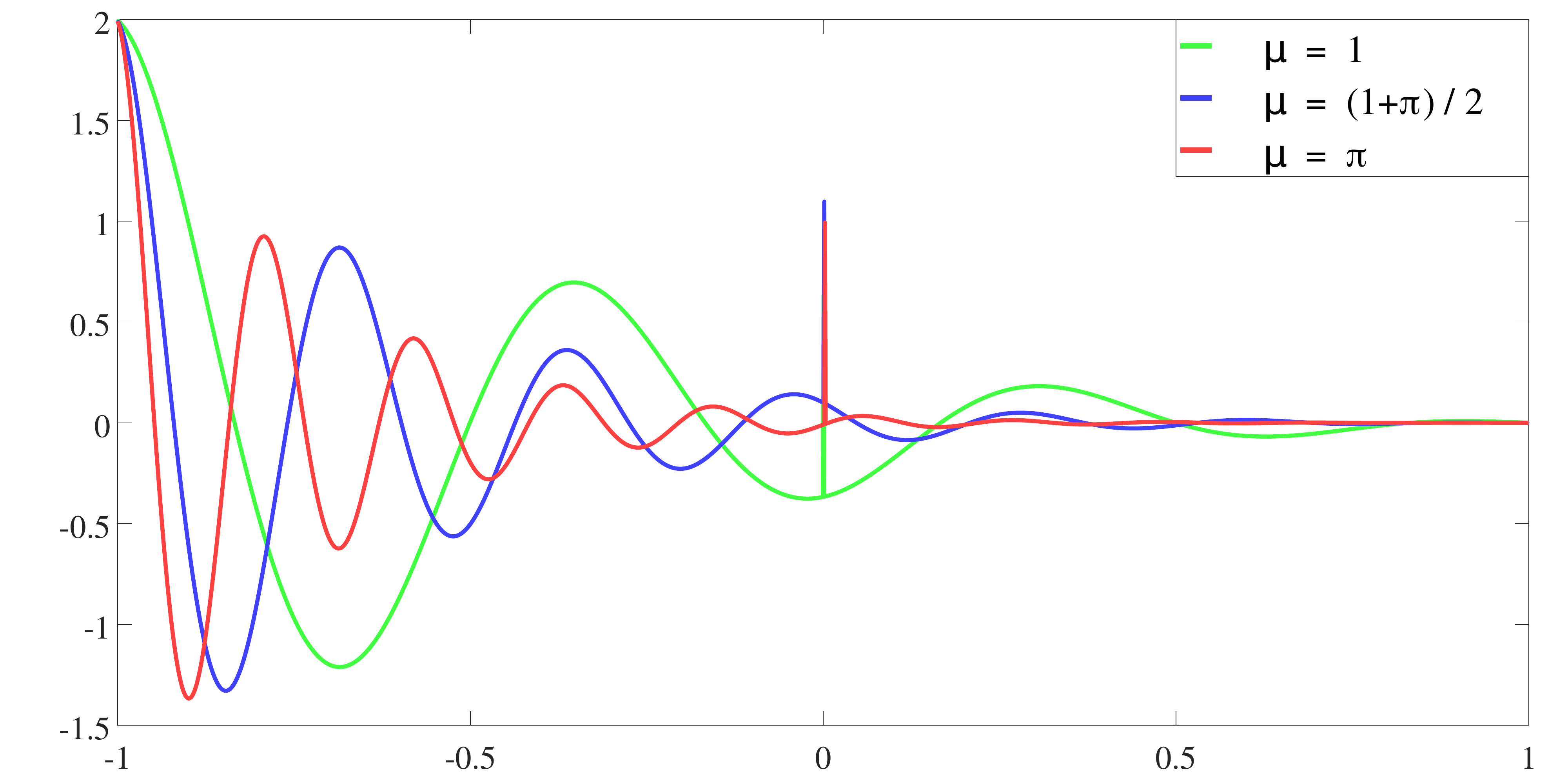}
    \caption{Elements of the families $\mathcal{F}^1(x,\mu)$ (left) and $\widetilde{\mathcal{F}}^1(x,\mu)$ (right), defined by~\eqref{1d_F_tr} and~\eqref{1dn_F_tr} respectively.}
    \label{fig:1d1_F_tr}
\end{figure}

\begin{figure}
    \hfill
    \begin{subfigure}{.49\linewidth}
        \resizebox{.9\linewidth}{!}{%
            \begin{tabular}{|c|c|c|c|c|}
                \hline$m$ & OGA & NGA & EIM & POD \\\hline
                3 & 3.798e-01 & 3.816e-01 & 3.900e-01 & 2.641e-01 \\\hline
                6 & 1.227e-01 & 1.295e-01 & 1.288e-01 & 8.637e-02 \\\hline
                9 & 3.382e-02 & 3.404e-02 & 3.160e-02 & 2.310e-02 \\\hline
                12 & 7.713e-03 & 7.179e-03 & 6.639e-03 & 4.292e-03 \\\hline
                15 & 7.077e-04 & 6.418e-04 & 5.332e-04 & 4.034e-04 \\\hline
                18 & 2.168e-05 & 2.715e-05 & 3.411e-05 & 1.635e-05 \\\hline
                21 & 4.153e-07 & 4.941e-07 & 4.024e-07 & 3.154e-07 \\\hline
                24 & 8.191e-09 & 4.488e-09 & 9.883e-09 & 2.982e-09 \\\hline
                27 & 3.510e-11 & 2.533e-11 & 3.364e-11 & 1.605e-11 \\\hline
                30 & 2.238e-14 & 3.040e-14 & 2.349e-14 & 2.742e-14 \\\hline
            \end{tabular}}
        \caption{average approximation error}
    \end{subfigure}
    \begin{subfigure}{.49\linewidth}
        \resizebox{.9\linewidth}{!}{%
            \begin{tabular}{|c|c|c|c|c|}
                \hline$m$ & OGA & NGA & EIM & POD \\\hline
                3 & 5.831e-01 & 5.860e-01 & 7.482e-01 & 6.205e-01 \\\hline
                6 & 2.940e-01 & 2.756e-01 & 2.898e-01 & 2.617e-01 \\\hline
                9 & 1.288e-01 & 1.284e-01 & 1.104e-01 & 8.285e-02 \\\hline
                12 & 1.966e-02 & 1.963e-02 & 1.780e-02 & 1.797e-02 \\\hline
                15 & 2.780e-03 & 2.235e-03 & 2.176e-03 & 2.134e-03 \\\hline
                18 & 1.082e-04 & 9.846e-05 & 1.914e-04 & 1.002e-04 \\\hline
                21 & 2.492e-06 & 2.649e-06 & 1.868e-06 & 2.093e-06 \\\hline
                24 & 4.994e-08 & 3.349e-08 & 6.512e-08 & 1.929e-08 \\\hline
                27 & 2.199e-10 & 1.174e-10 & 1.264e-10 & 9.134e-11 \\\hline
                30 & 1.829e-13 & 1.442e-13 & 1.017e-13 & 1.460e-13 \\\hline
            \end{tabular}}
        \caption{maximal approximation error}
    \end{subfigure}
    \begin{subfigure}{.49\linewidth}
        \includegraphics[width=\linewidth]{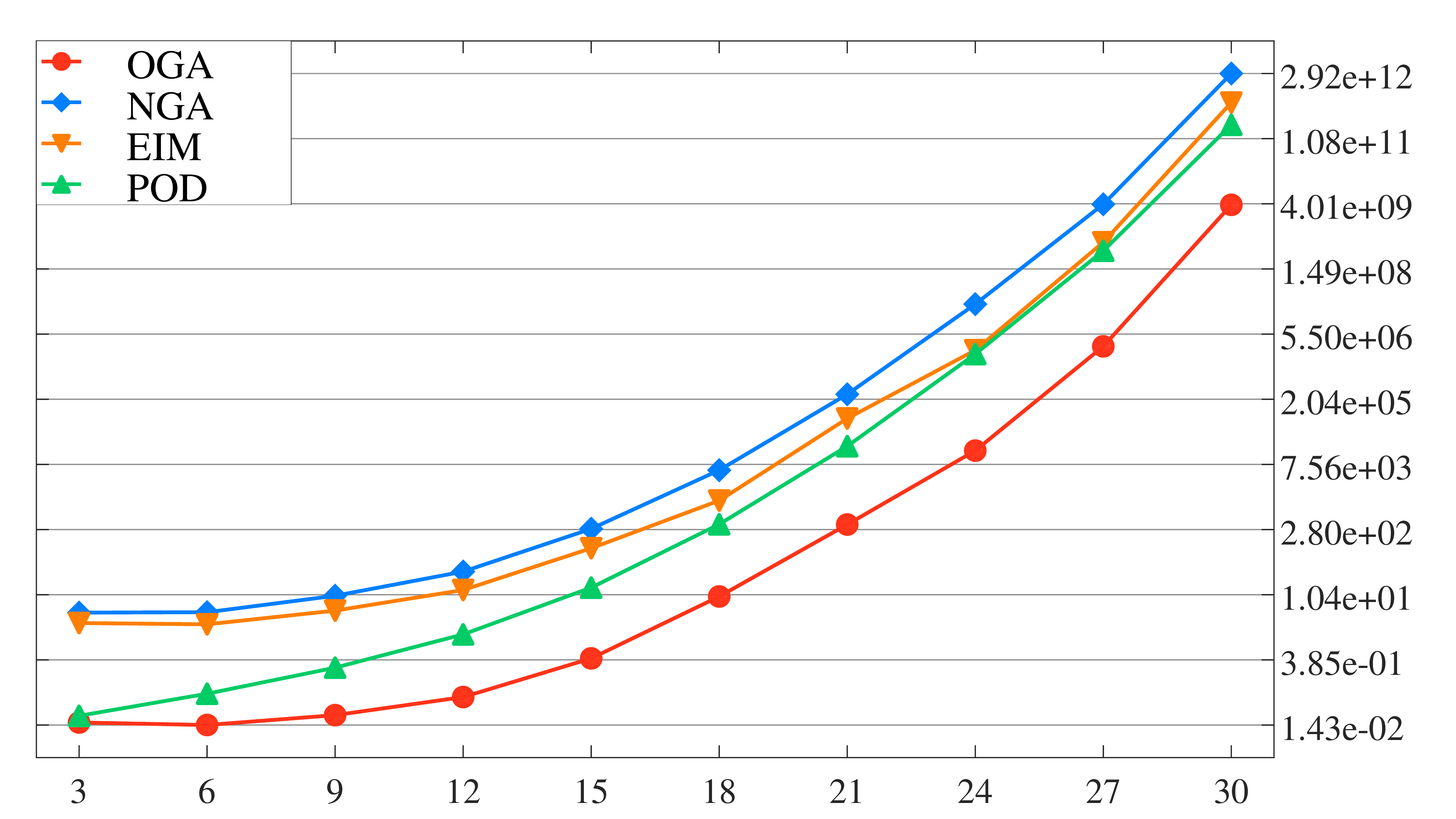}
        \caption{average quality}
    \end{subfigure}
    \begin{subfigure}{.49\linewidth}
        \includegraphics[width=\linewidth]{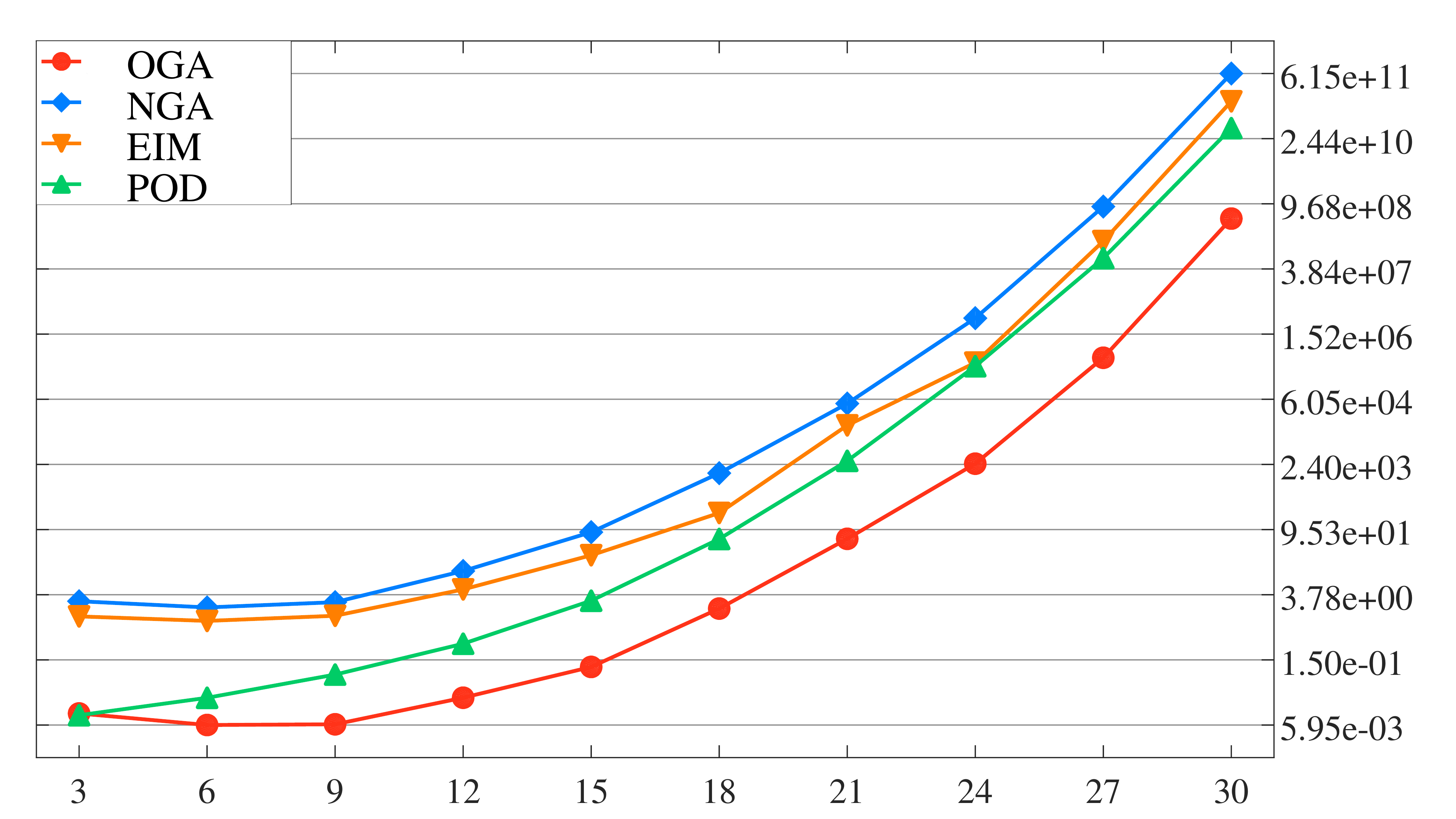}
        \caption{minimal quality}
    \end{subfigure}
    \caption{Performance of the OGA, NGA, EIM, and POD reduced bases for approximating the set $\mathcal{F}^1_{tr}$ (sampled from the one-dimensional parametric family~\eqref{1d_F_tr}) in $L_1[-1,1]$.}
    \label{fig:1d1_appr_qual}
\end{figure}

\begin{figure}
    \hfill
    \begin{subfigure}{.49\linewidth}
        \resizebox{.9\linewidth}{!}{%
            \begin{tabular}{|c|c|c|c|c|}
                \hline$m$ & OGA & NGA & EIM & POD \\\hline
                3 & 3.800e-01 & 3.812e-01 & 3.905e-01 & 2.644e-01 \\\hline
                6 & 1.235e-01 & 1.292e-01 & 1.764e-01 & 8.658e-02 \\\hline
                9 & 3.434e-02 & 3.486e-02 & 5.051e-02 & 2.338e-02 \\\hline
                12 & 8.113e-03 & 7.491e-03 & 3.661e-02 & 4.534e-03 \\\hline
                15 & 9.584e-04 & 1.024e-03 & 3.564e-02 & 7.028e-04 \\\hline
                18 & 3.055e-04 & 3.461e-04 & 1.330e-02 & 5.182e-04 \\\hline
                21 & 1.968e-04 & 2.082e-04 & 6.758e-03 & 2.302e-04 \\\hline
                24 & 1.543e-04 & 1.591e-04 & 5.532e-03 & 2.103e-04 \\\hline
                27 & 1.268e-04 & 1.308e-04 & 5.067e-03 & 1.797e-04 \\\hline
                30 & 1.082e-04 & 1.157e-04 & 4.830e-03 & 1.658e-04 \\\hline
            \end{tabular}}
        \caption{average approximation error}
    \end{subfigure}
    \begin{subfigure}{.49\linewidth}
        \resizebox{.9\linewidth}{!}{%
            \begin{tabular}{|c|c|c|c|c|}
                \hline$m$ & OGA & NGA & EIM & POD \\\hline
                3 & 5.832e-01 & 5.852e-01 & 7.484e-01 & 6.208e-01 \\\hline
                6 & 2.948e-01 & 2.682e-01 & 3.932e-01 & 2.622e-01 \\\hline
                9 & 1.323e-01 & 1.340e-01 & 1.240e-01 & 8.362e-02 \\\hline
                12 & 2.038e-02 & 1.976e-02 & 1.132e-01 & 1.910e-02 \\\hline
                15 & 2.940e-03 & 3.311e-03 & 1.132e-01 & 3.004e-03 \\\hline
                18 & 9.030e-04 & 1.230e-03 & 6.363e-02 & 2.459e-03 \\\hline
                21 & 5.403e-04 & 5.341e-04 & 3.087e-02 & 1.113e-03 \\\hline
                24 & 3.561e-04 & 3.901e-04 & 2.949e-02 & 8.831e-04 \\\hline
                27 & 2.793e-04 & 2.990e-04 & 2.013e-02 & 6.987e-04 \\\hline
                30 & 2.532e-04 & 2.960e-04 & 1.975e-02 & 5.394e-04 \\\hline
            \end{tabular}}
        \caption{maximal approximation error}
    \end{subfigure}
    \begin{subfigure}{.49\linewidth}
        \includegraphics[width=\linewidth]{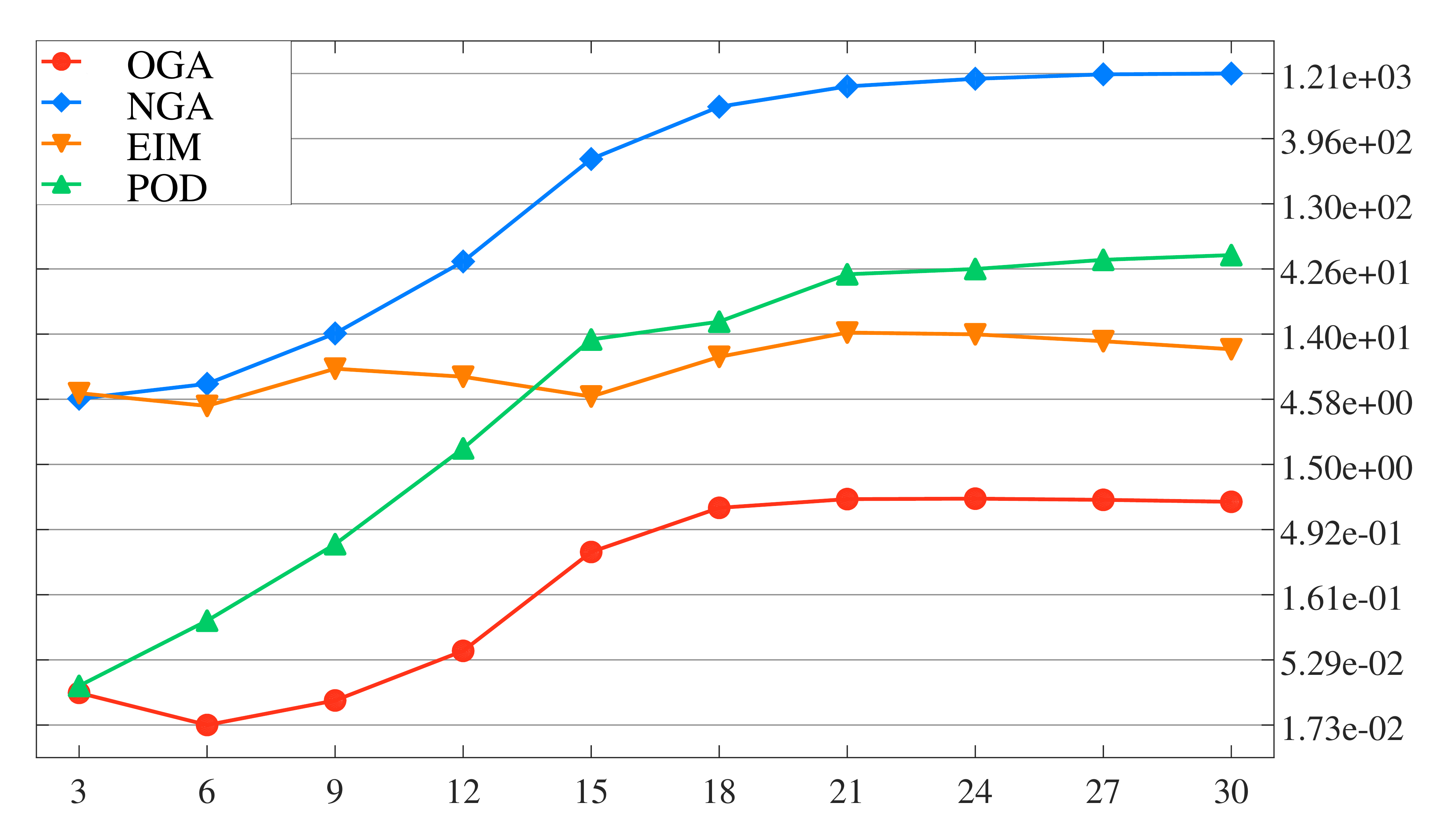}
        \caption{average quality}
    \end{subfigure}
    \begin{subfigure}{.49\linewidth}
        \includegraphics[width=\linewidth]{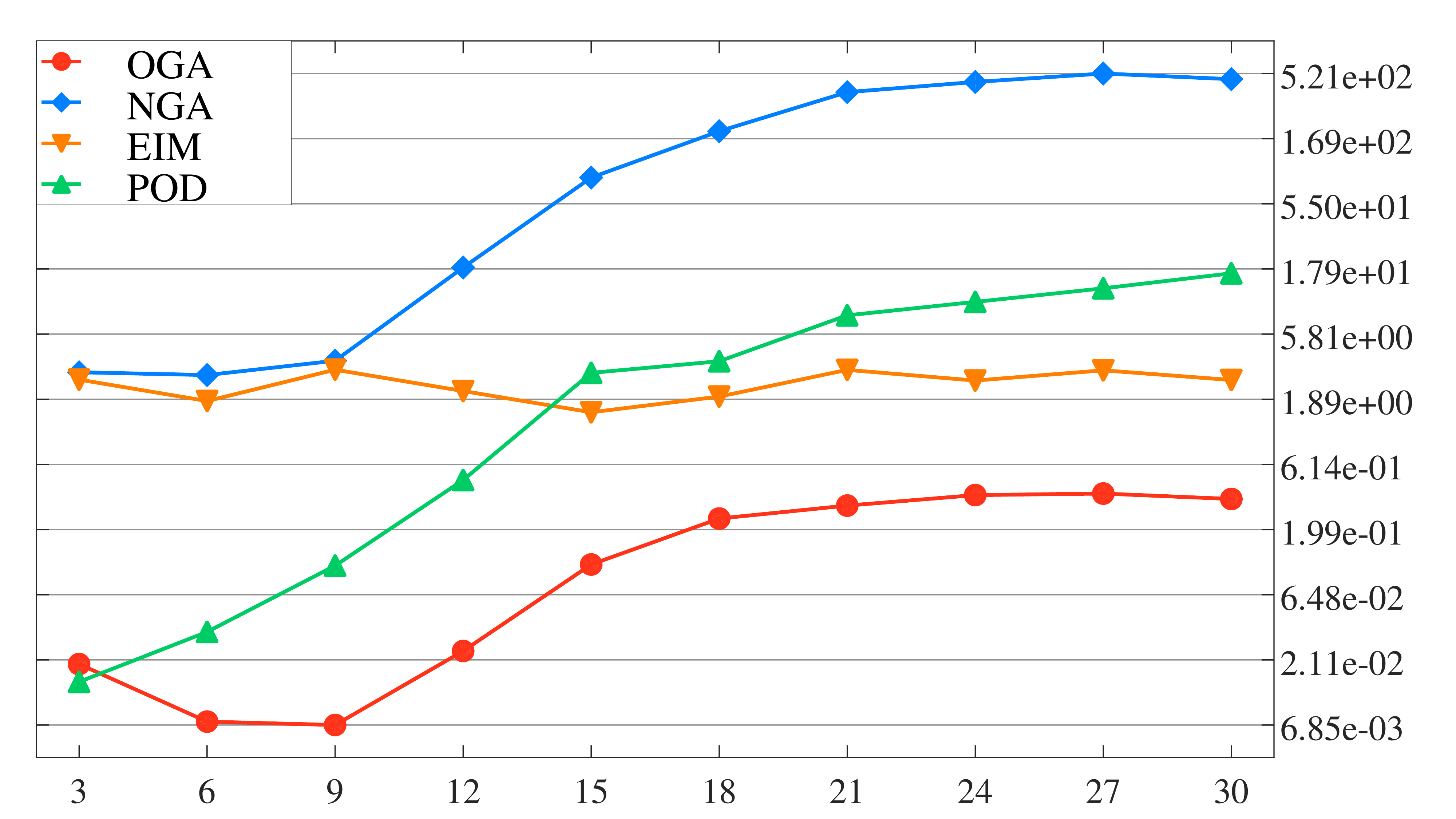}
        \caption{minimal quality}
    \end{subfigure}
    \caption{Performance of the OGA, NGA, EIM, and POD reduced bases for approximating the set $\widetilde{\mathcal{F}}^1_{tr}$ (sampled from the one-dimensional parametric family~\eqref{1dn_F_tr}) in $L_1[-1,1]$.}
    \label{fig:1d1n_appr_qual}
\end{figure}

The following parametric function was used in~\cite{ngpape2008} to compare the interpolation properties of the reduced bases generated by EIM and POD.
We compare those reduced bases with the ones constructed by the OGA and the NGA and present the approximating properties and `qualities' of all four algorithms.
Namely, we consider the parametric family $\mathcal{F}^1(x,\mu) \subset L_1(\Omega^1 \times \mathcal{D}^1)$ given by
\begin{equation}
    \label{1d_F_tr}
	\mathcal{F}^1(x,\mu) = (1-x) \, \cos(3\pi \mu (x+1)) \, \exp(-\mu(1+x)),
\end{equation}
where $x \in \Omega^1 = [-1,1]$ and $\mu \in \mathcal{D}^1 = [1,\pi]$.

To obtain a perturbed parametric family $\widetilde{\mathcal{F}}^1$ we add an indicator function of the set $\br{\frac{\mu-1}{1000}, \frac{\mu}{1000}}$ to $\mathcal{F}^1$ to serve as a perturbation:
\begin{equation}
    \label{1dn_F_tr}
	\widetilde{\mathcal{F}}^1(x,\mu) = (1-x) \, \cos(3\pi \mu (x+1)) \, \exp(-\mu(1+x)) + \left\{\frac{\mu-1}{1000} < x < \frac{\mu}{1000} \right\},
\end{equation}
where $x \in \Omega^1 = [-1,1]$ and $\mu \in \mathcal{D}^1 = [1,\pi]$.

Some elements of the families $\mathcal{F}^1$ and $\widetilde{\mathcal{F}}^1$ are shown in Figure~\ref{fig:1d1_F_tr}.
We generate the discrete training sets $\mathcal{F}^1_{tr}$ and $\widetilde{\mathcal{F}}^1_{tr}$ by sampling the spatial domain $\Omega$ at $N_h = 100,000$ uniform points, and the parametric domain $\mathcal{D}$ at $N_{tr} = 500$ uniform points.

The approximation error and the quality of the constructed reduced bases are reported in Figures~\ref{fig:1d1_appr_qual} and~\ref{fig:1d1n_appr_qual} respectively.
In this experiment we observe that for the set $\mathcal{F}^1_{tr}$ all the algorithms perform similarly, with only the OGA lagging in terms of quality due to the higher computational cost.
However, for the perturbed set $\widetilde{\mathcal{F}}^1_{tr}$, the approximation accuracy of the EIM is significantly lower than that of the other algorithms, due to the construction of the interpolation points $z_n$, $n \ge 0$, and the magnitude of perturbation.
Moreover, the OGA and NGA appear to be more resistant to such perturbation, which results in an even better approximation error than the one provided by the POD.

\subsubsection*{Two-dimensional parametric family}
Consider the family $\mathcal{F}^2 \in L_1(\Omega^2 \times \mathcal{D}^2)$:
\begin{equation}
    \label{2d_F_tr}
	\mathcal{F}^2(x_1,x_2,\mu_1,\mu_2) = \sin(x_1\mu_1) \, \cos(x_2\mu_2) \, \exp(|x_1|\mu_1 + |x_2|\mu_2),
\end{equation}
where $(x_1,x_2) \in \Omega^2 = [0,1] \times [0,1]$ and $(\mu_1,\mu_2) \in \mathcal{D}^2 = [\pi/3, 2\pi] \times [\pi/3, 2\pi]$.
We generate the discrete training set $\mathcal{F}^2_{tr}$ by sampling the spatial domain $\Omega^2$ on the uniform grid of the size $N_h = 300 \times 300$, and the parametric domain $\mathcal{D}^2$ at the uniform grid of the size $N_{tr} = 25 \times 25$.

To obtain a noisy training set $\widetilde{\mathcal{F}}^2_{tr}$ we change $1\%$ of the coordinates of the training set $\mathcal{F}^2_{tr}$ by random values raging from $0$ up to $\operatorname{avg}(\mathcal{F}^2_{tr})$ (the average value of all the coordinates of the training set).
Thus, our `noise' randomly changes a large number ($1\%$) of coordinates of the training set by a small value (up to $\operatorname{avg}(\mathcal{F}^2_{tr})$).

The approximation error and the quality of the constructed reduced bases constructed for the sets $\mathcal{F}^2_{tr}$ and $\widetilde{\mathcal{F}}^2_{tr}$ are reported in Figures~\ref{fig:2d1_appr_qual} and~\ref{fig:2d1n_appr_qual} respectively.
We observe that for the set $\mathcal{F}^2_{tr}$ the approximation accuracies of the NGA and POD are better than those of the OGA and EIM.
In terms of quality, the POD is close to the OGA at low iterations due to the necessity to perform the singular value decomposition on the large amount of data.
For the noisy set $\widetilde{\mathcal{F}}^2_{tr}$, the quality of the algorithms does not change significantly, however, the approximation accuracies of the OGA and NGA are slightly better than that of the POD.

\begin{figure}
    \hfill
    \begin{subfigure}{.49\linewidth}
        \resizebox{.9\linewidth}{!}{%
            \begin{tabular}{|c|c|c|c|c|}
                \hline$m$ & OGA & NGA & EIM & POD \\\hline
                3 & 1.092e-03 & 1.092e-03 & 1.092e-03 & 1.157e-03 \\\hline
                6 & 2.552e-04 & 2.552e-04 & 2.552e-04 & 2.561e-04 \\\hline
                9 & 1.171e-04 & 1.171e-04 & 1.441e-04 & 1.256e-04 \\\hline
                12 & 3.862e-05 & 3.863e-05 & 3.872e-05 & 3.946e-05 \\\hline
                15 & 2.099e-05 & 2.102e-05 & 2.359e-05 & 2.049e-05 \\\hline
                18 & 1.009e-05 & 1.009e-05 & 1.602e-05 & 1.101e-05 \\\hline
                21 & 6.361e-06 & 6.126e-06 & 8.995e-06 & 5.551e-06 \\\hline
                24 & 3.739e-06 & 3.880e-06 & 4.509e-06 & 4.053e-06 \\\hline
                27 & 2.779e-06 & 2.640e-06 & 3.411e-06 & 2.468e-06 \\\hline
                30 & 1.893e-06 & 1.491e-06 & 1.942e-06 & 1.433e-06 \\\hline
            \end{tabular}}
        \caption{average approximation error}
    \end{subfigure}
    \begin{subfigure}{.49\linewidth}
        \resizebox{.9\linewidth}{!}{%
            \begin{tabular}{|c|c|c|c|c|}
                \hline$m$ & OGA & NGA & EIM & POD \\\hline
                3 & 9.519e-02 & 9.519e-02 & 9.519e-02 & 6.013e-02 \\\hline
                6 & 8.730e-03 & 8.727e-03 & 8.745e-03 & 5.162e-03 \\\hline
                9 & 4.422e-03 & 4.424e-03 & 5.425e-03 & 5.042e-03 \\\hline
                12 & 7.167e-04 & 7.167e-04 & 7.181e-04 & 5.587e-04 \\\hline
                15 & 2.744e-04 & 2.744e-04 & 4.038e-04 & 3.034e-04 \\\hline
                18 & 1.102e-04 & 1.102e-04 & 2.601e-04 & 1.516e-04 \\\hline
                21 & 6.585e-05 & 1.093e-04 & 1.739e-04 & 6.838e-05 \\\hline
                24 & 4.680e-05 & 4.359e-05 & 6.777e-05 & 3.704e-05 \\\hline
                27 & 2.233e-05 & 2.690e-05 & 3.812e-05 & 1.845e-05 \\\hline
                30 & 1.709e-05 & 1.879e-05 & 2.427e-05 & 1.313e-05 \\\hline
            \end{tabular}}
        \caption{maximal approximation error}
    \end{subfigure}
    \begin{subfigure}{.49\linewidth}
        \includegraphics[width=\linewidth]{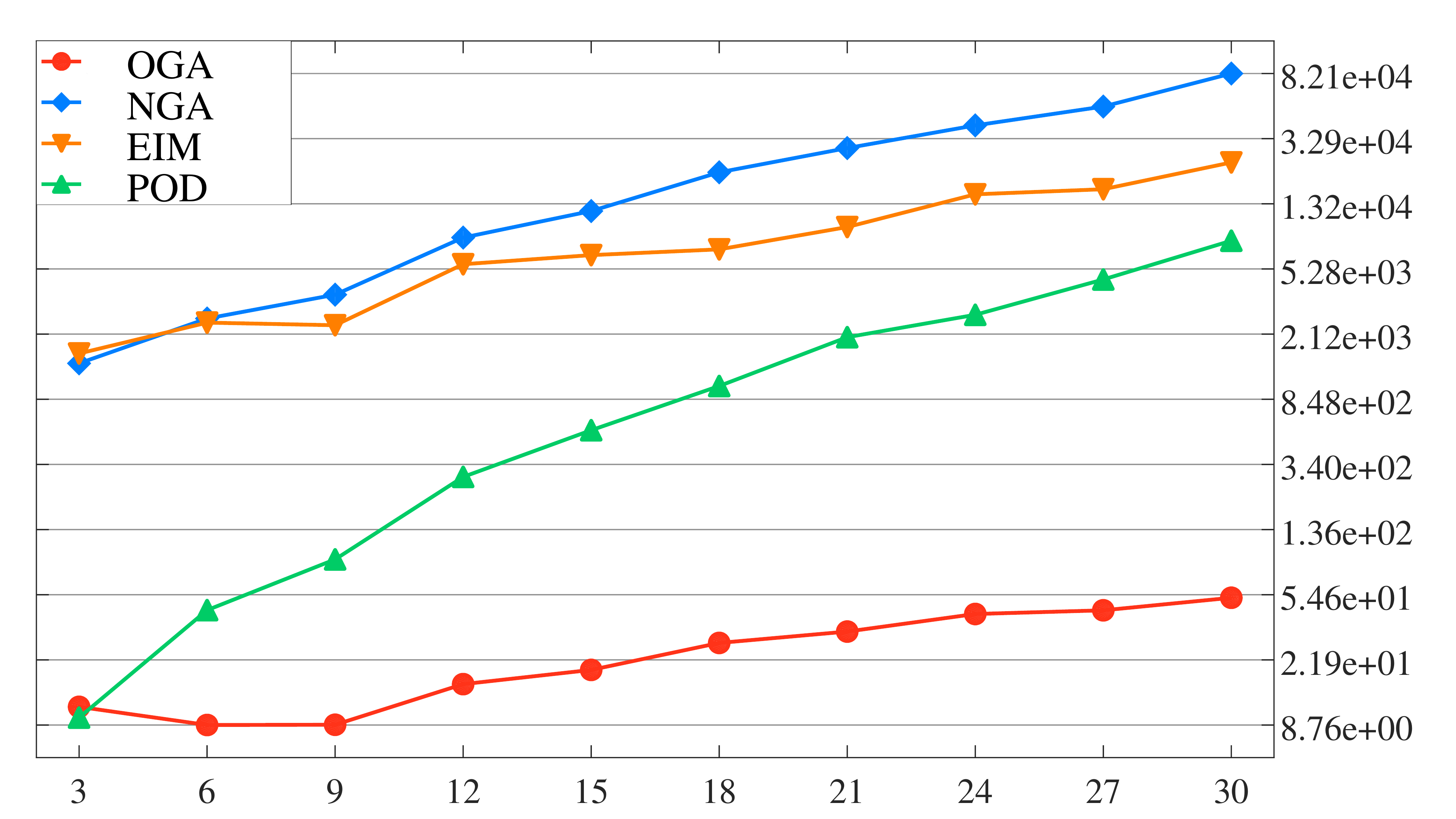}
        \caption{average quality}
    \end{subfigure}
    \begin{subfigure}{.49\linewidth}
        \includegraphics[width=\linewidth]{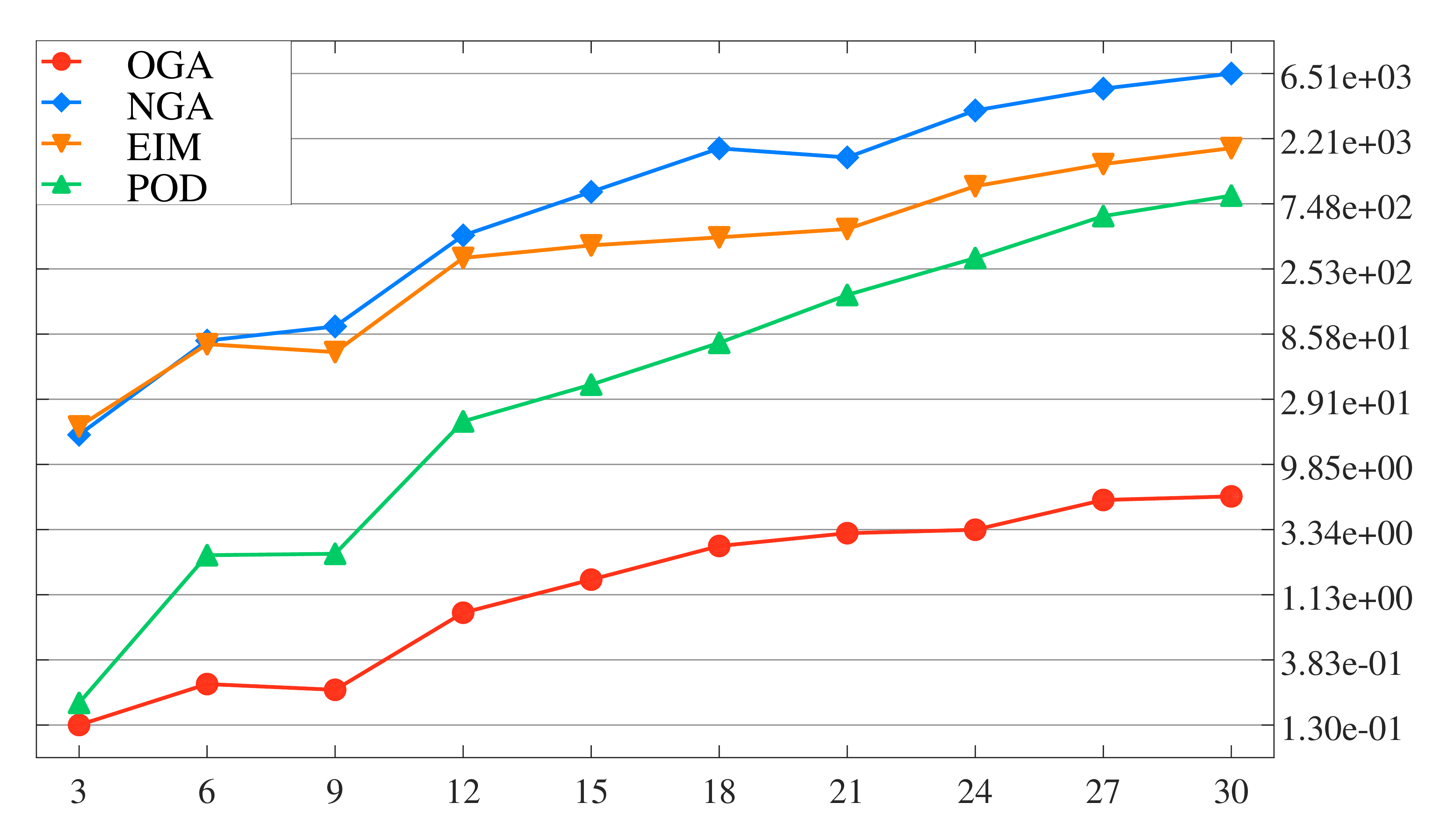}
        \caption{minimal quality}
    \end{subfigure}
    \caption{Performance of the OGA, NGA, EIM, and POD reduced bases for approximating the set $\mathcal{F}^2_{tr}$ (sampled from the two-dimensional parametric family~\eqref{2d_F_tr}) in $L_1([0,1]^2)$.}
    \label{fig:2d1_appr_qual}
\end{figure}

\begin{figure}
    \hfill
    \begin{subfigure}{.49\linewidth}
        \resizebox{.9\linewidth}{!}{%
            \begin{tabular}{|c|c|c|c|c|}
                \hline$m$ & OGA & NGA & EIM & POD \\\hline
                3 & 1.093e-03 & 1.093e-03 & 1.093e-03 & 1.158e-03 \\\hline
                6 & 2.560e-04 & 2.560e-04 & 2.561e-04 & 2.569e-04 \\\hline
                9 & 1.187e-04 & 1.187e-04 & 1.458e-04 & 1.273e-04 \\\hline
                12 & 4.166e-05 & 4.167e-05 & 4.158e-05 & 4.223e-05 \\\hline
                15 & 2.569e-05 & 2.571e-05 & 3.026e-05 & 2.575e-05 \\\hline
                18 & 1.882e-05 & 1.888e-05 & 1.995e-05 & 1.848e-05 \\\hline
                21 & 1.553e-05 & 1.593e-05 & 1.968e-05 & 1.579e-05 \\\hline
                24 & 1.364e-05 & 1.377e-05 & 1.869e-05 & 1.472e-05 \\\hline
                27 & 1.207e-05 & 1.293e-05 & 1.784e-05 & 1.423e-05 \\\hline
                30 & 1.132e-05 & 1.247e-05 & 1.776e-05 & 1.420e-05 \\\hline
            \end{tabular}}
        \caption{average approximation error}
    \end{subfigure}
    \begin{subfigure}{.49\linewidth}
        \resizebox{.9\linewidth}{!}{%
            \begin{tabular}{|c|c|c|c|c|}
                \hline$m$ & OGA & NGA & EIM & POD \\\hline
                3 & 9.521e-02 & 9.521e-02 & 9.525e-02 & 6.015e-02 \\\hline
                6 & 8.772e-03 & 8.779e-03 & 8.773e-03 & 5.172e-03 \\\hline
                9 & 4.434e-03 & 4.434e-03 & 5.434e-03 & 5.042e-03 \\\hline
                12 & 7.432e-04 & 7.426e-04 & 7.427e-04 & 5.813e-04 \\\hline
                15 & 3.654e-04 & 3.654e-04 & 4.419e-04 & 3.244e-04 \\\hline
                18 & 1.852e-04 & 2.011e-04 & 2.025e-04 & 1.606e-04 \\\hline
                21 & 1.398e-04 & 1.695e-04 & 2.022e-04 & 1.369e-04 \\\hline
                24 & 1.151e-04 & 1.455e-04 & 1.990e-04 & 1.101e-04 \\\hline
                27 & 9.667e-05 & 1.129e-04 & 1.981e-04 & 1.011e-04 \\\hline
                30 & 7.804e-05 & 9.912e-05 & 1.978e-04 & 1.009e-04 \\\hline
            \end{tabular}}
        \caption{maximal approximation error}
    \end{subfigure}
    \begin{subfigure}{.49\linewidth}
        \includegraphics[width=\linewidth]{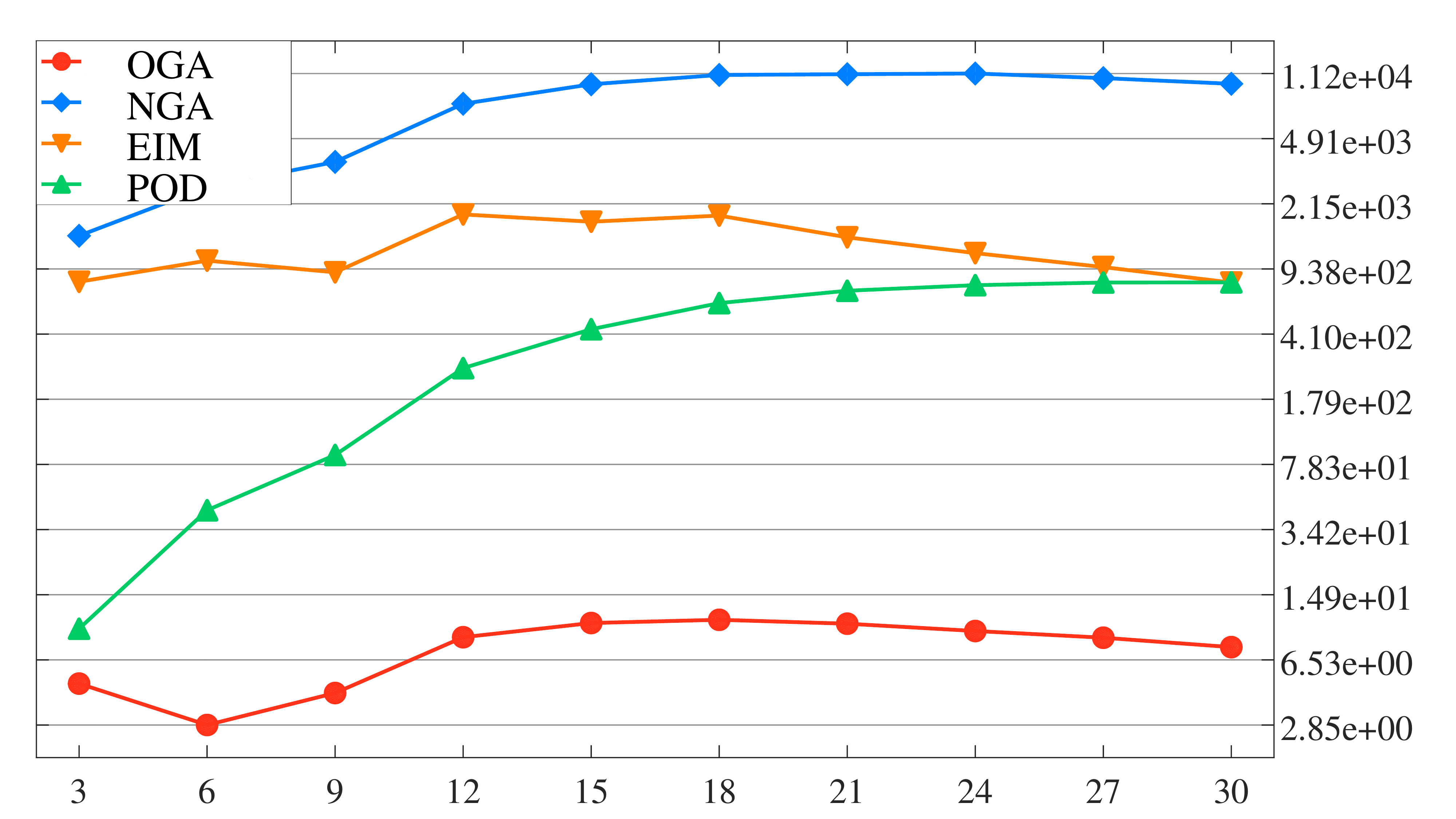}
        \caption{average quality}
    \end{subfigure}
    \begin{subfigure}{.49\linewidth}
        \includegraphics[width=\linewidth]{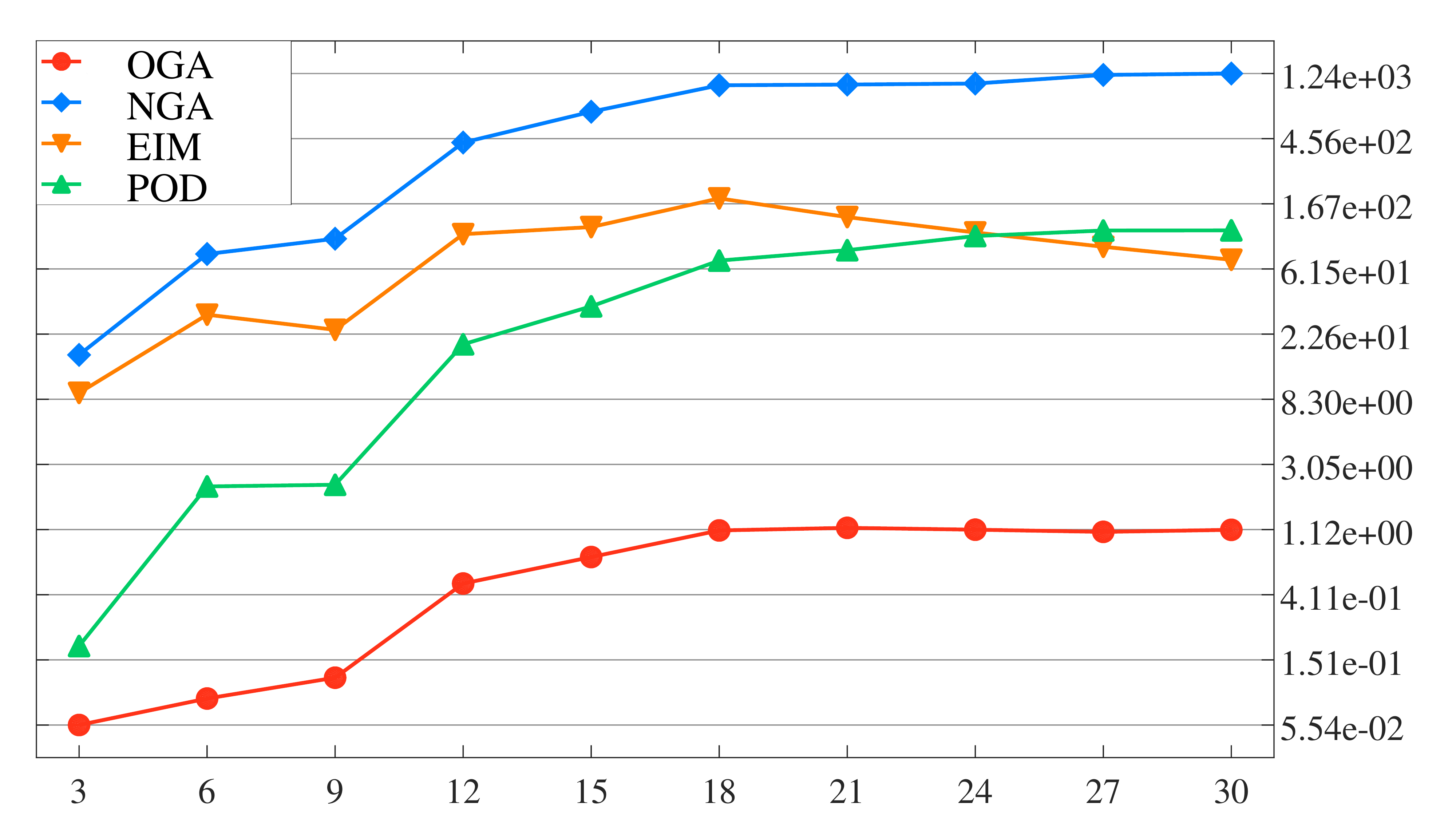}
        \caption{minimal quality}
    \end{subfigure}
    \caption{Performance of the OGA, NGA, EIM, and POD reduced bases for approximating the set $\widetilde{\mathcal{F}}^2_{tr}$ (sampled from the two--dimensional parametric family~\eqref{2d_F_tr} with small additive noise) in $L_1([0,1]^2)$.}
    \label{fig:2d1n_appr_qual}
\end{figure}

\subsubsection*{Three-dimensional parametric family}
Consider the family $\mathcal{F}^3 \in L_1(\Omega^3 \times \mathcal{D}^3)$:
\begin{multline}
    \label{3d_F_tr}
	\mathcal{F}^3(x_1,x_2,x_3,\mu_1,\mu_2,\mu_3) = (1-x_1)(1-x_2)(1-x_3) 
	\\
	\times \sin\big(\pi (x_1\mu_1 + x_2\mu_2 + x_3\mu_3)\big) \, \exp\big((x_1+\mu_1)(x_2+\mu_2)(x_3+\mu_3)\big),
\end{multline}
where $(x_1,x_2,x_3) \in \Omega^3 = [0,1] \times [0,1] \times [0,1]$ and $(\mu_1,\mu_2,\mu_3) \in \mathcal{D}^3 = [0,1] \times [0,1] \times [0,1]$.
We generate the discrete training set $\mathcal{F}^3_{tr}$ by sampling the spatial domain $\Omega$ on the uniform grid of the size $N_h = 50 \times 50 \times 50$, and the parametric domain $\mathcal{D}$ at the uniform grid of the size $N_{tr} = 8 \times 8 \times 8$.

For the corresponding noisy training set $\widetilde{\mathcal{F}}^3_{tr}$ we change $.001\%$ of the coordinates the training set $\mathcal{F}^3_{tr}$ by random values raging from $0$ up to $\max(\mathcal{F}^3_{tr})$ (the maximal value of the coordinates of the training set).
Thus, our `noise' randomly changes a small number ($.001\%$) of coordinates of the training set by a large value (up to $\max(\mathcal{F}^3_{tr})$).

The approximation error and the quality of the constructed reduced bases constructed for the sets $\mathcal{F}^3_{tr}$ and $\widetilde{\mathcal{F}}^3_{tr}$ are reported in Figures~\ref{fig:3d1_appr_qual} and~\ref{fig:3d1n_appr_qual} respectively.
Similarly to the previous examples, we observe that the OGA and NGA are more resistant to perturbations, thus providing a better approximation accuracy for the noisy set $\widetilde{\mathcal{F}}^3_{tr}$, even though the POD is more accurate on the unperturbed set $\mathcal{F}^3_{tr}$.

\begin{figure}
    \hfill
    \begin{subfigure}{.49\linewidth}
        \resizebox{.9\linewidth}{!}{%
            \begin{tabular}{|c|c|c|c|c|}
                \hline$m$ & OGA & NGA & EIM & POD \\\hline
                3 & 1.001e-02 & 1.001e-02 & 7.688e-03 & 6.142e-03 \\\hline
                6 & 1.965e-03 & 1.965e-03 & 1.591e-03 & 1.053e-03 \\\hline
                9 & 7.788e-04 & 7.787e-04 & 8.540e-04 & 4.911e-04 \\\hline
                12 & 4.436e-04 & 4.493e-04 & 4.314e-04 & 1.809e-04 \\\hline
                15 & 2.584e-04 & 2.393e-04 & 2.612e-04 & 8.728e-05 \\\hline
                18 & 8.960e-05 & 8.444e-05 & 8.371e-05 & 3.439e-05 \\\hline
                21 & 3.247e-05 & 3.270e-05 & 3.254e-05 & 1.393e-05 \\\hline
                24 & 1.915e-05 & 1.960e-05 & 1.864e-05 & 8.206e-06 \\\hline
                27 & 9.967e-06 & 8.704e-06 & 1.084e-05 & 4.931e-06 \\\hline
                30 & 5.917e-06 & 6.285e-06 & 8.008e-06 & 3.056e-06 \\\hline
            \end{tabular}}
        \caption{average approximation error}
    \end{subfigure}
    \begin{subfigure}{.49\linewidth}
        \resizebox{.9\linewidth}{!}{%
            \begin{tabular}{|c|c|c|c|c|}
                \hline$m$ & OGA & NGA & EIM & POD \\\hline
                3 & 3.942e-02 & 3.942e-02 & 3.469e-02 & 2.746e-02 \\\hline
                6 & 5.452e-03 & 5.452e-03 & 5.844e-03 & 5.301e-03 \\\hline
                9 & 3.749e-03 & 3.749e-03 & 3.911e-03 & 2.665e-03 \\\hline
                12 & 8.980e-04 & 8.865e-04 & 9.765e-04 & 9.524e-04 \\\hline
                15 & 5.367e-04 & 4.741e-04 & 6.343e-04 & 6.320e-04 \\\hline
                18 & 2.850e-04 & 2.695e-04 & 2.676e-04 & 1.913e-04 \\\hline
                21 & 6.952e-05 & 8.273e-05 & 1.089e-04 & 5.437e-05 \\\hline
                24 & 4.908e-05 & 4.731e-05 & 5.559e-05 & 3.117e-05 \\\hline
                27 & 2.412e-05 & 2.381e-05 & 2.539e-05 & 1.250e-05 \\\hline
                30 & 1.606e-05 & 1.555e-05 & 2.128e-05 & 9.645e-06 \\\hline
            \end{tabular}}
        \caption{maximal approximation error}
    \end{subfigure}
    \begin{subfigure}{.49\linewidth}
        \includegraphics[width=\linewidth]{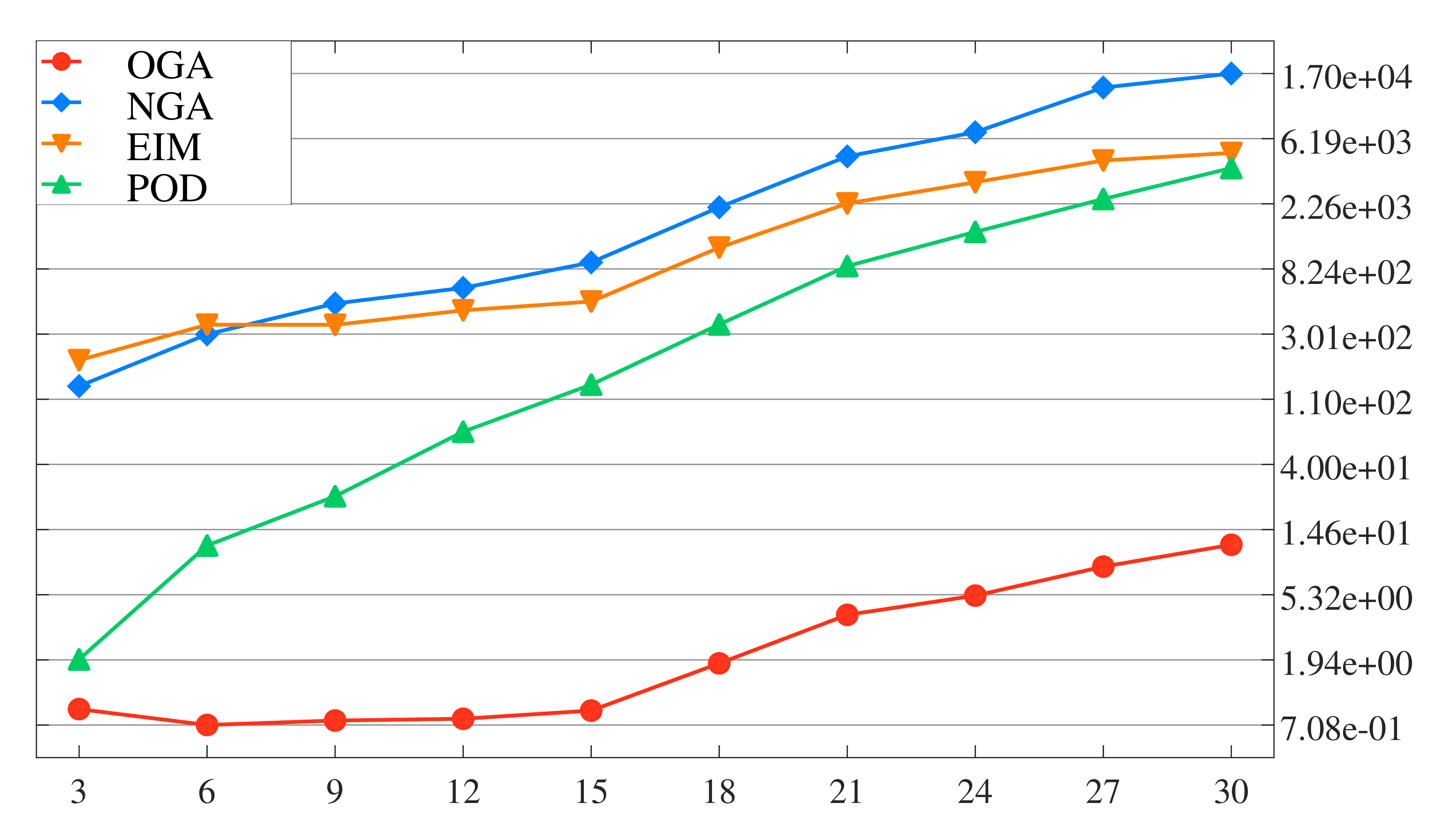}
        \caption{average quality}
    \end{subfigure}
    \begin{subfigure}{.49\linewidth}
        \includegraphics[width=\linewidth]{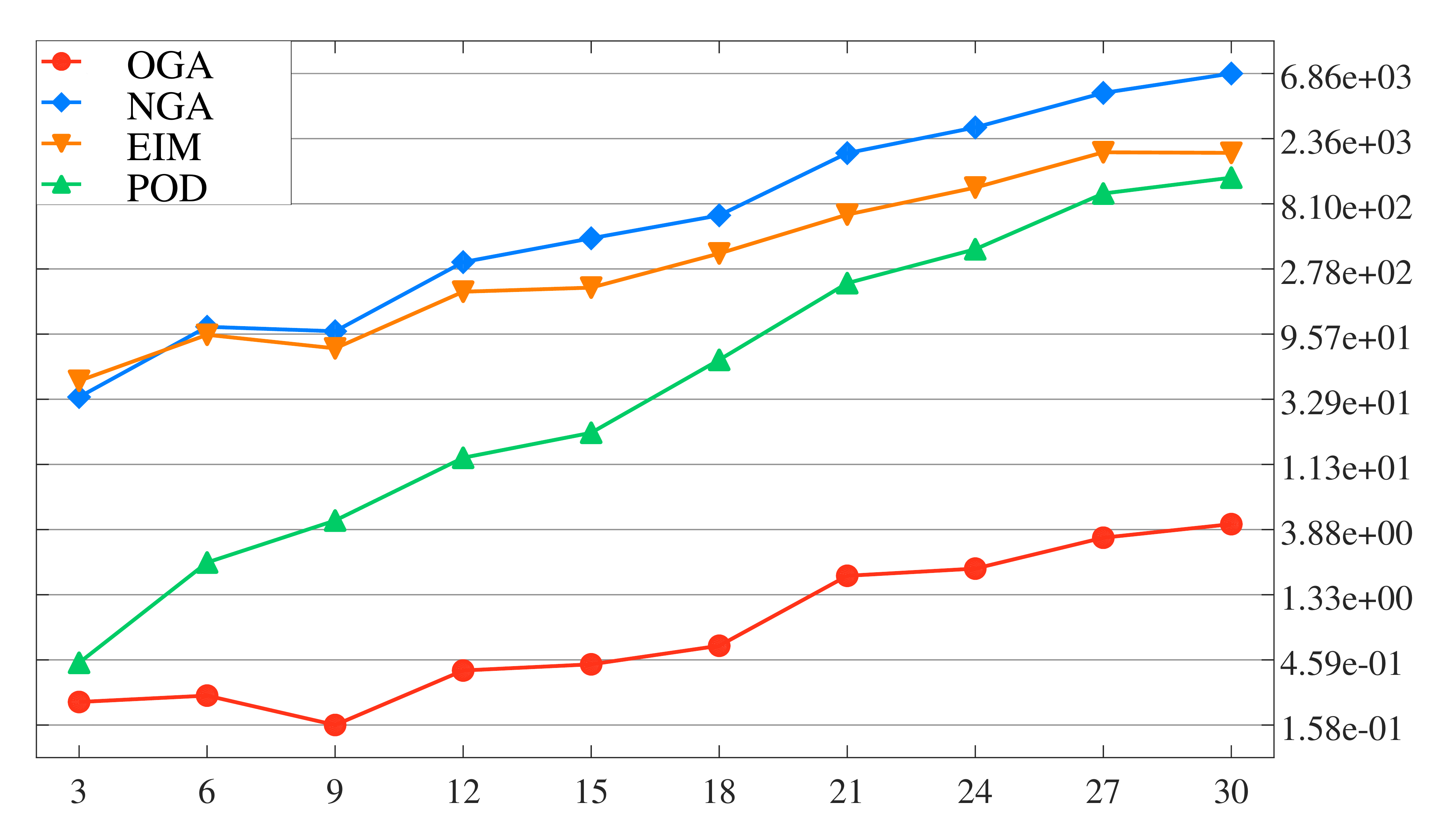}
        \caption{minimal quality}
    \end{subfigure}
    \caption{Performance of the OGA, NGA, EIM, and POD reduced bases for approximating the set $\mathcal{F}^3_{tr}$ (sampled from the three-dimensional parametric family~\eqref{3d_F_tr}) in $L_1([0,1]^3)$.}
    \label{fig:3d1_appr_qual}
\end{figure}

\begin{figure}
    \hfill
    \begin{subfigure}{.49\linewidth}
        \resizebox{.9\linewidth}{!}{%
            \begin{tabular}{|c|c|c|c|c|}
                \hline$m$ & OGA & NGA & EIM & POD \\\hline
                3 & 1.001e-02 & 1.001e-02 & 8.062e-03 & 6.151e-03 \\\hline
                6 & 1.977e-03 & 1.977e-03 & 2.154e-03 & 1.067e-03 \\\hline
                9 & 7.926e-04 & 7.929e-04 & 1.340e-03 & 5.102e-04 \\\hline
                12 & 4.443e-04 & 4.716e-04 & 9.804e-04 & 2.434e-04 \\\hline
                15 & 2.637e-04 & 2.653e-04 & 3.775e-04 & 2.203e-04 \\\hline
                18 & 1.373e-04 & 1.363e-04 & 3.298e-04 & 1.765e-04 \\\hline
                21 & 8.198e-05 & 9.419e-05 & 2.785e-04 & 1.449e-04 \\\hline
                24 & 6.426e-05 & 7.786e-05 & 2.683e-04 & 1.416e-04 \\\hline
                27 & 6.145e-05 & 6.744e-05 & 2.416e-04 & 1.348e-04 \\\hline
                30 & 5.684e-05 & 6.367e-05 & 2.308e-04 & 1.288e-04 \\\hline
            \end{tabular}}
        \caption{average approximation error}
    \end{subfigure}
    \begin{subfigure}{.49\linewidth}
        \resizebox{.9\linewidth}{!}{%
            \begin{tabular}{|c|c|c|c|c|}
                \hline$m$ & OGA & NGA & EIM & POD \\\hline
                3 & 3.943e-02 & 3.943e-02 & 4.005e-02 & 2.770e-02 \\\hline
                6 & 5.481e-03 & 5.481e-03 & 5.173e-03 & 5.349e-03 \\\hline
                9 & 3.764e-03 & 3.768e-03 & 4.274e-03 & 2.818e-03 \\\hline
                12 & 9.248e-04 & 9.102e-04 & 4.058e-03 & 1.113e-03 \\\hline
                15 & 5.366e-04 & 5.014e-04 & 2.502e-03 & 1.102e-03 \\\hline
                18 & 3.321e-04 & 3.286e-04 & 2.419e-03 & 1.061e-03 \\\hline
                21 & 1.656e-04 & 2.068e-04 & 2.329e-03 & 9.938e-04 \\\hline
                24 & 1.398e-04 & 1.599e-04 & 2.329e-03 & 9.833e-04 \\\hline
                27 & 1.073e-04 & 1.526e-04 & 2.290e-03 & 9.449e-04 \\\hline
                30 & 9.411e-05 & 1.484e-04 & 1.825e-03 & 9.438e-04 \\\hline
            \end{tabular}}
        \caption{maximal approximation error}
    \end{subfigure}
    \begin{subfigure}{.49\linewidth}
        \includegraphics[width=\linewidth]{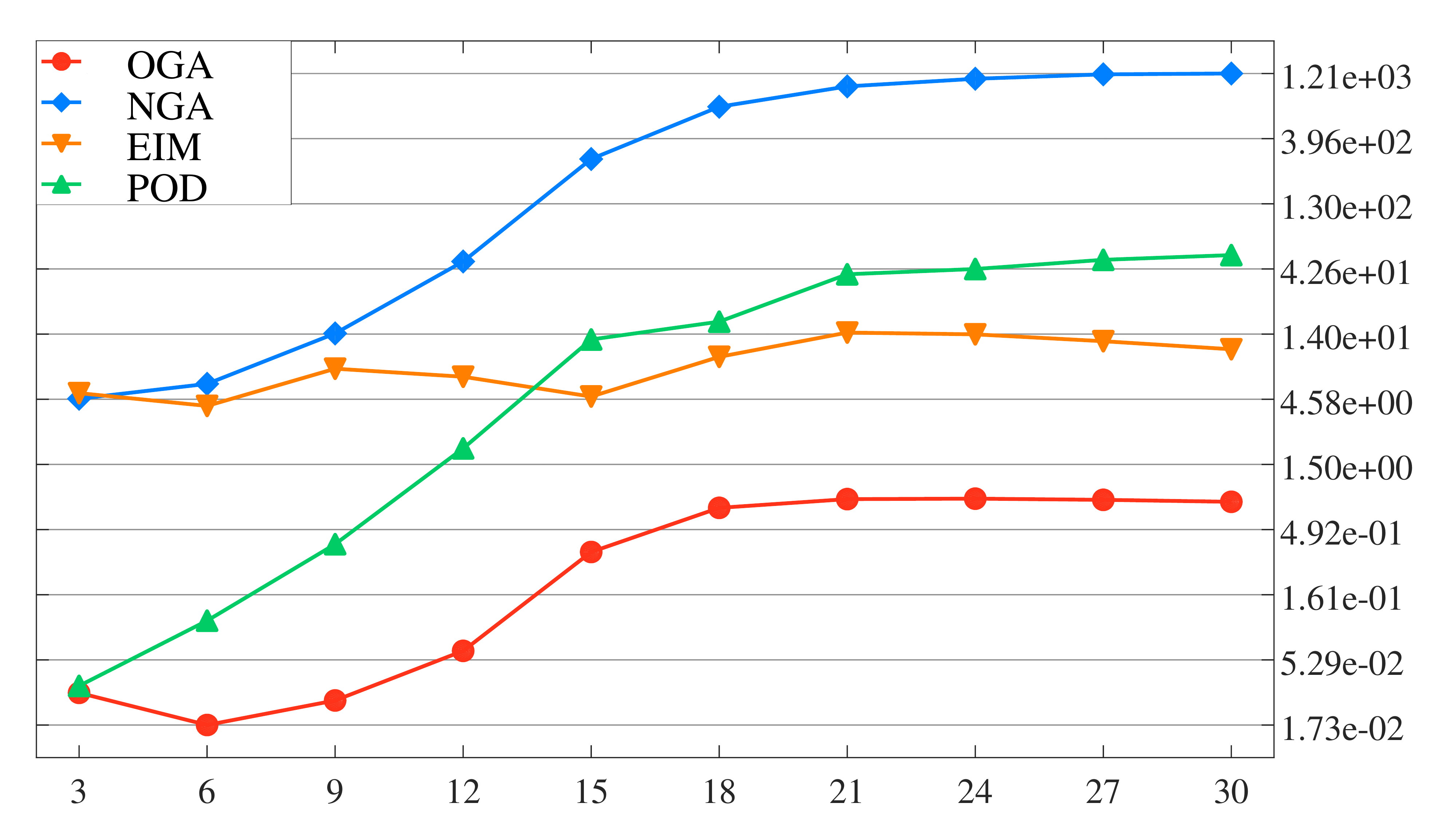}
        \caption{average quality}
    \end{subfigure}
    \begin{subfigure}{.49\linewidth}
        \includegraphics[width=\linewidth]{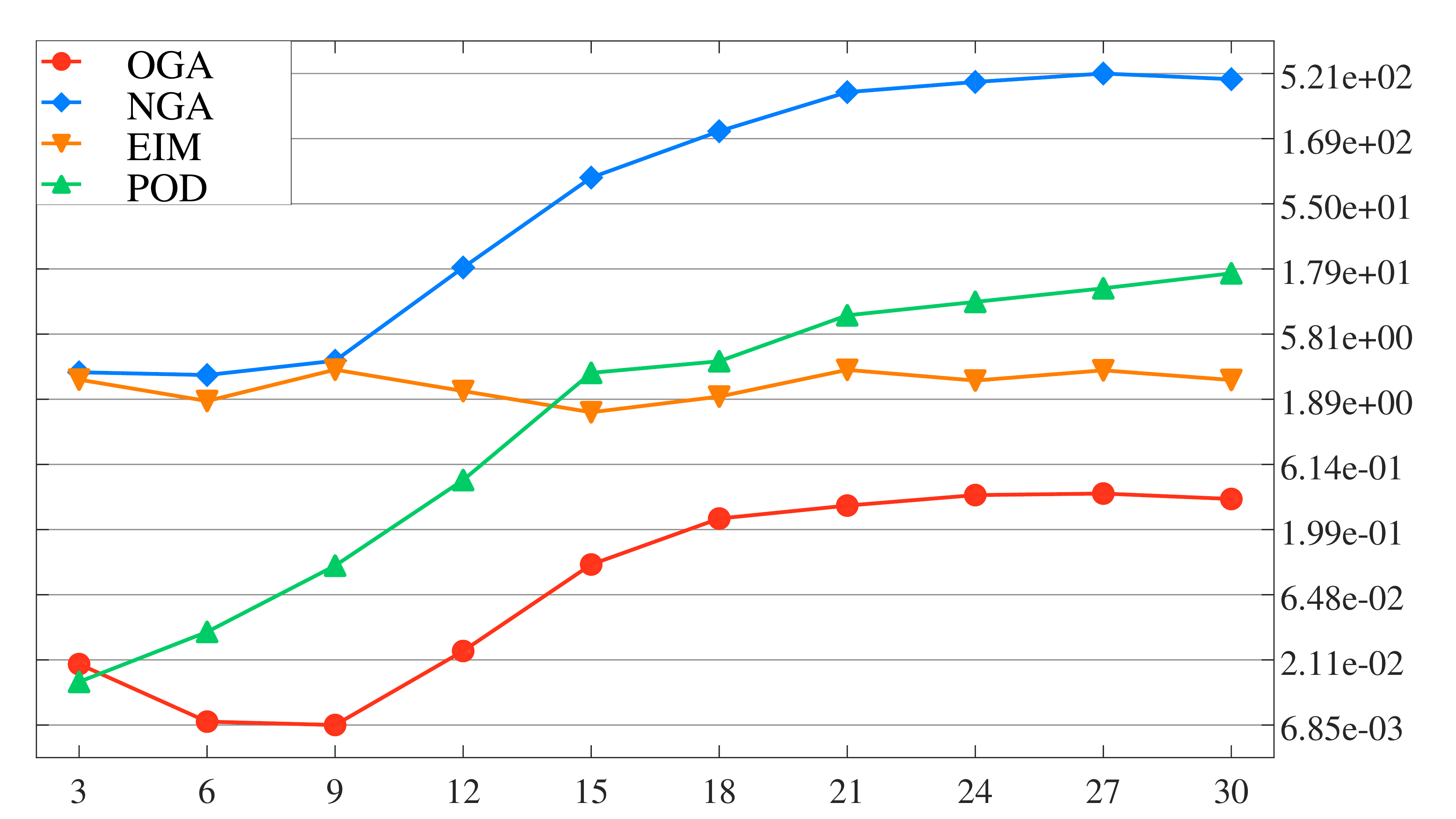}
        \caption{minimal quality}
    \end{subfigure}
    \caption{Performance of the OGA, NGA, EIM, and POD reduced bases for approximating the set $\widetilde{\mathcal{F}}^3_{tr}$ (sampled from the three-dimensional parametric family~\eqref{3d_F_tr} with localized additive noise) in $L_1([0,1]^3)$.}
    \label{fig:3d1n_appr_qual}
\end{figure}

\subsection{Norms of the operators $\boldsymbol{\mathcal{R}_n}$}\label{section_remarks}
We conclude this section by commenting on the utilization of the operators $\mathcal{R}_n$ in place of the actual remainder of the orthogonal projection, i.e., computing $\n{\mathcal{R}_n(\cdot)}$ instead of $\dist(\cdot,V_n)$.
Recall that generally the norm of $\mathcal{R}_n$ is larger than $1$, thus the norm of the resulting remainder can be larger than the norm of the projected element.
Namely, estimate~\eqref{estimate_norm_R_n_V} provides
\[
	\n{\mathcal{R}_n}_{V^*} \le B_n^g = \min\{R^n, C_g + 1\},
\]
where $R(\mathcal{X}) > 1$ is the parameter that was introduced in Lemma~\ref{estimate_norm_R_n}, and $C_g(\mathcal{F,X})$ is the basis constant of the reduced basis constructed by the NGA.
Since $C_g$ is unknown, it might seem that our theoretical estimates for the convergence rates of the NGA essentially contain an exponential factor.
For instance, Theorem~\ref{theorem_nga_indirect} provides the convergence rate
\[
	\sigma_n(\mathcal{F,X}) \le \sqrt{2} B_{n/2}^g(\mathcal{F,X}) \min_{0 < m < n} \gamma_{n+m}(\mathcal{X}) \, d_m^{1-m/n}(\mathcal{F,X}),
\]
where the factor $B_{n/2}^g$, as can be seen from the proof, comes from the estimate
\[
	\br{\prod_{m=0}^{n-1} \n{\mathcal{R}_m}_{V_n^*}}^{1/n} 
	\le \min\{R(\mathcal{X})^{(n-1)/2}, C_g(\mathcal{F,X})+1\}
	\le B_{n/2}^g.
\]
To clarify that such exponential growth does not actually occur, we note that
\[
	\br{\prod_{m=0}^{n-1} \n{\mathcal{R}_m}_{V_n^*}}^{1/n} \le \max_{0 \le m < n} \n{\mathcal{R}_m}_{V_n^*}
\]
and calculate the values of $\max_{0 \le m < n} \n{\mathcal{R}_m}_{V_n^*}$ (with various values of $n$) that correspond to the numerical examples given in this section.
We obtain the norm of operator $\mathcal{R}_m$ on the subspace $V_n$ by using the property $\mathcal{R}_m(V_m) = 0$ (see~\eqref{nga_R_n(V_n)=0}) and solving the optimization problem
\begin{align*}
	\n{\mathcal{R}_m}_{V_n^*} 
	= \max_{f \in V_n} \left\{ \frac{\n{\mathcal{R}_m(f)}}{\n{f}} \right\}
	&= \max_{\alpha_0,\ldots,\alpha_{n-1} \in \mathbb{R}} \left\{ \frac{\n{\mathcal{R}_m\br{\sum_{k=0}^{n-1} \alpha_k g_k}}}{\n{\sum_{k=0}^{n-1} \alpha_k g_k}} \right\}
    \\
	&= \max_{\alpha_0,\ldots,\alpha_{n-1} \in \mathbb{R}} \left\{ \frac{\n{\sum_{k=m}^{n-1} \alpha_k g_k}}{\n{\sum_{k=0}^{n-1} \alpha_k g_k}} \right\}.
\end{align*}
In Table~\ref{tbl:norm_R_m_max} we present the obtained values for $\max_{0 \le m < n} \n{\mathcal{R}_m}_{V_n^*}$ and compare them to the most optimistic theoretical estimate for the exponential factor $R^{(n-1)/2}$ provided by Lemma~\ref{estimate_norm_R_n} for the Hilbert space setting by the equality $R(\mathcal{H}) = (1+\sqrt{5})/2$.

\begin{table}
	\begin{tabular}{cc|c|c|c|c|c|}
	\\\cline{2-7}
	& \multicolumn{6}{|c|}{Dimensionality of the subspace $V_n$}
	\\\cline{2-7}
	\multicolumn{1}{c|}{} & $5$ & $10$ & $15$ & $20$ & $25$ & $30$
	\\\hline
	\multicolumn{1}{|c|}{Theoretical estimate} & $2.6180$ & $8.7186$ & $29.034$ & $96.690$ & $322.00$ & $1072.3$
	\\\hline
	\multicolumn{1}{|c|}{Figure~\ref{15d1}} & $1.1857$ & $1.3313$ & $1.4260$ & $1.5389$ & $1.7963$ & $2.1759$
	\\\hline
	\multicolumn{1}{|c|}{Figure~\ref{15d3}} & $1.0242$ & $1.0623$ & $1.0978$ & $1.1189$ & $1.1328$ & $1.1445$
	\\\hline
	\multicolumn{1}{|c|}{Figure~\ref{15d10}} & $1.3707$ & $2.0628$ & $2.3644$ & $2.6191$ & $2.6894$ & $2.9247$
	\\\hline
	\multicolumn{1}{|c|}{Figure~\ref{fig:1d1_appr_qual}} & $1.1369$ & $1.8069$ & $2.6031$ & $3.4946$ & $3.9994$ & $4.3324$
	\\\hline
	\multicolumn{1}{|c|}{Figure~\ref{fig:1d1n_appr_qual}} & $1.1377$ & $1.8124$ & $2.5270$ & $2.7774$ & $2.7996$ & $2.7991$
	\\\hline
	\multicolumn{1}{|c|}{Figure~\ref{fig:2d1_appr_qual}} & $1.3498$ & $1.5669$ & $1.6344$ & $2.7012$ & $3.2748$ & $3.2714$
	\\\hline
	\multicolumn{1}{|c|}{Figure~\ref{fig:2d1n_appr_qual}} & $1.3492$ & $1.5636$ & $1.5849$ & $2.8762$ & $2.8712$ & $2.8759$
	\\\hline
	\multicolumn{1}{|c|}{Figure~\ref{fig:3d1_appr_qual}} & $1.1349$ & $1.2804$ & $1.6615$ & $1.8075$ & $1.8230$ & $1.8328$
	\\\hline
	\multicolumn{1}{|c|}{Figure~\ref{fig:3d1n_appr_qual}} & $1.1352$ & $1.2806$ & $1.6187$ & $2.2997$ & $2.2397$ & $2.3585$
	\\\hline
	\end{tabular}
    \caption{Theoretical and practical estimates of the maximum of the norms of the operators $\mathcal{R}_m$ for $0 \le m < n = \dim V_n$.}
    \label{tbl:norm_R_m_max}
\end{table}

\subsection{Concluding remarks}
In this effort we have presented the Natural Greedy Algorithm~--- a greedy-type method for constructing reduced bases in Banach spaces with arbitrary norms.
In its core the NGA relies on the novel method of computing projections onto subspaces in non-Hilbert norms, which is based on the utilization of norming functionals.
Such an approach is significantly less computationally demanding than the conventional orthogonal projection, which requires the solution of an optimization problem.
We have subjected the NGA to the rigorous mathematical analysis and established the convergence properties that are equivalent to the well-known OGA.

We have performed various numerical experiments that compare the efficiency of the NGA to that of the OGA, and the other commonly used reduced basis procedures, namely, the EIM and POD.
Our experiments demonstrate that the NGA matches the approximation accuracy that is provided by the other algorithms, while requires far less computational effort, especially when compared to the OGA.
We also observe that the NGA and the OGA appear to possess the robustness to various perturbations (which is expected from greedy algorithms), and thus perform better than the EIM and POD when applied to the noisy data.

Moreover, we note that the NGA can be viewed as a generalization of the Orthogonal Matching Pursuit to the Banach space setting, e.g., an alternative to the OGA as in a Hilbert space both algorithms coincide.
Additionally, the EIM can be viewed as a special case of the NGA since in $L_\infty$-spaces both the NGA and EIM essentially provide the same reduced basis sequences.

\subsubsection*{Acknowledgements}
This material is based upon work supported in part by: the U.S. Department of Energy, Office of Science, Early Career Research Program under award number ERKJ314; U.S. Department of Energy, Office of Advanced Scientific Computing Research under award numbers ERKJ331 and ERKJ345; the National Science Foundation, Division of Mathematical Sciences, Computational Mathematics program under contract number DMS1620280; and by the Laboratory Directed Research and Development program at the Oak Ridge National Laboratory, which is operated by UT-Battelle, LLC., for the U.S. Department of Energy under contract DE-AC05-00OR22725.

\appendix

\section{Proofs of main results}\label{section_proofs}
In this section we prove the results related to the OGA and the NGA that were stated in sections~\ref{section_ga} and~\ref{section_nga}.

\subsubsection*{Proof of Lemma~\ref{estimate_norm_R_n}}
First, we prove our estimate on the norm of operators $\mathcal{R}_n$ on $\mathcal{X}$.
\begin{proof}[Proof of Lemma~\ref{estimate_norm_R_n}]
Let $\mathcal{X}$ be a smooth Banach space (we will address a non-smooth case later).
Take any $f,g \in \mathcal{X}$ with $\n{f} = \n{g} = 1$ and denote $\alpha = F_g(f) \in [-1;1]$.
We begin by estimating the norm of
\[
	r(f) = f - F_g(f) \, g = f - \alpha g.
\]
If $\alpha = 0$ then trivially $\n{r(f)} = 1$.
Assume that $\alpha \ne 0$, then from equality~\eqref{norming_functional_smooth} we obtain
\[
	F_{\alpha g}(f) = \operatorname{sgn}(\alpha) \, F_g(f) = |\alpha|.
\] 
Recall that for any $x,y \in \mathcal{X}$ we have (see, e.g.,~\cite[Lemma~6.1]{te2011})
\[
	\n{x - y} \le \n{x} - F_x(y) + 2\n{x} \rho\br{\frac{\n{y}}{\n{x}}}.
\]
Applying this inequality with $x = \alpha g$ and $y = f$ provides
\[
	\n{r(f)} = \n{f - \alpha g}
	\le |\alpha| - F_{\alpha g}(f) + 2|\alpha| \, \rho(1/|\alpha|)
	= 2|\alpha| \, \rho(1/|\alpha|).
\]
On the other hand, by the triangle inequality we get
\[
	\n{r(f)} \le 1 + |\alpha| \le 2.
\]
Since both estimates must hold and functions $1 + x$ and $2x \, \rho(x^{-1})$ intersect once on $(0;1]$, we conclude that $\n{r(f)} \le 1 + \mu$, where $\mu \in (0;1]$ is the root of the equation $1 + x = 2x \, \rho(x^{-1})$.
Note that $\mu = \mu(\mathcal{X})$ does not depend on $f,g$, and $\alpha$, thus
\[
	\n{r}_{\mathcal{X}^*} \le 1 + \mu.
\]
For any non-smooth Banach space we have $\rho(u) = u$ and $\mu = 1$, and thus the above estimate holds.

We have proven that for any Banach space $\mathcal{X}$ and any $n \ge 0$ one has
\[
	\n{r_n}_{\mathcal{X}^*} \le R = 1 + \mu
\]
and, since $\mathcal{R}_n = r_{n-1} \circ \ldots \circ r_0$, we obtain
\[
	\n{\mathcal{R}_n}_{\mathcal{X}^*} \le \prod_{k=0}^{n-1} \n{r_k}_{\mathcal{X}^*} \le R^n.
\]
Lastly, note that if $\rho(u) \le c_\rho u^q$ for some $q \in (1;2]$, we get $\mu \le \mu_q$, where $\mu_q > 0$ is the solution of the equation $1 + x = 2c_\rho x^{1-q}$, which can be bigger than $1$.
Hence we arrive at
\[
	R = 1 + \mu \le \min\{1 + \mu_q, 2\}.
\]
\end{proof}

\subsubsection*{Proof of Lemma~\ref{lemma_nga!=ga}}
Next, we prove that the NGA is not a weak version of the OGA, but in fact is a new method of constructing reduced bases.
\begin{proof}[Proof of Lemma~\ref{lemma_nga!=ga}]
We first construct a simple setting in which $\n{\mathcal{R}_n(\cdot)}$ and $\dist(\cdot,V_n)$ are equivalent on $\mathcal{F}$ but the equivalence constant is arbitrarily large.

\noindent
Namely, let $\mathcal{X} = \ell_1$ and $\{e_n\}_{n=0}^\infty$ be the canonical basis in $\ell_1$.
Take any $0 < \epsilon < 1/2$ and consider vectors $f_*,f_0,f_1,f_2,\ldots$ given by
\[
	\begin{bmatrix}
		f_* \\ f_0 \\ f_1 \\ f_2 \\ f_3 \\ \vdots
	\end{bmatrix}
	=
	\begin{bmatrix}
		1 & 0 & 0 & 0 & 0 & 0 & \ldots \\
		-\epsilon & 1-\epsilon & 0 & 0 & 0 & 0 & \ldots \\
		-\epsilon/2 & -\epsilon/2 & 1-\epsilon & 0 & 0 & 0 & \ldots \\
		-\epsilon/4 & -\epsilon/4 & -\epsilon/2 & 1-\epsilon & 0 & 0 & \ldots \\
		-\epsilon/8 & -\epsilon/8 & -\epsilon/4 & -\epsilon/2 & 1-\epsilon & 0 & \ldots \\
		\vdots & \vdots & \vdots & \vdots & \vdots & \vdots & \ddots
	\end{bmatrix},
\]
i.e. $f_* = e_0$, and for any $n \ge 0$
\begin{equation}\label{lemma_nga!=ga_f_n}
	f_n = -\frac{\epsilon}{2^n} e_0 - \sum_{k=1}^n \frac{\epsilon}{2^{n+1-k}} e_k + (1-\epsilon) e_{n+1}.
\end{equation}
Note that the vectors $f_*,f_0,f_1,f_2,\ldots$ have norm $1$, and by using the formula $F_x(y) = \sum_{n=0}^\infty \operatorname{sgn}(x_n) \, y_n$ for a norming functional in $\ell_1$ we conclude that the sequence $\{f_n\}_{n=0}^\infty$ is semi-orthogonal in the sense that $F_{f_n}(f_m) = 0$ for any $m > n \ge 0$.

Take any positive sequence $\{a_n\}_{n=0}^\infty$ that is monotonically decreasing to $0$ and a number $0 < \alpha < a_0$, and define set $\mathcal{F} \subset \mathcal{X}$ as
\[
	\mathcal{F} = \{\alpha f_*\} \cup \{a_n f_n\}_{n=0}^\infty,
\]
which is compact since $a_n \to 0$.
We will show that by making $\alpha$ sufficiently small one can get arbitrarily large value for the ratio $\n{\mathcal{R}_n(f_*)}/\dist(f_*,V_n)$.
Denote
\[
	M = M(\alpha) = \min \left\{ n \ge 1 : \alpha \ge \frac{a_n}{(2(1-\epsilon))^n} \right\}.
\]
We will use induction to show that for any $1 \le m \le M$
\begin{equation}\label{lemma_nga!=ga_hypothesis}
	\{g_n\}_{n=0}^{m-1} = \{f_n\}_{n=0}^{m-1}
	\ \text{ and }\ 
	\mathcal{R}_m(f_*) = \alpha \, (1-\epsilon)^m \br{e_0 + \sum_{n=1}^m 2^{n-1} e_n}.
\end{equation}
The base of induction holds since $\alpha < a_0$ and thus the NGA selects
\[
	a_0 f_0 = \mathop{\mathrm{argmax}}_{f \in \mathcal{F}} \n{f}
	\text{ and }
	g_0 = f_0.
\]
Hence $F_{g_0}(f_*) = -\alpha$ and
\[
	\mathcal{R}_1(f_*) =  f_* - F_{g_0}(f_*) \, g_0 = \alpha \, (1-\epsilon) \, (e_0 + e_1).
\]
Assume that the induction holds for some $1 \le m < M$, then since $\{g_n\}_{n=0}^{m-1} = \{f_n\}_{n=0}^{m-1}$ and due to semi-orthogonality of vectors $\{f_n\}_{n=0}^\infty$ we get $\mathcal{R}_m(a_n f_n) = a_n f_n$ (and thus $\n{\mathcal{R}_m(a_n f_n)} = a_n$) for any $n \ge m$.
By the assumption we have
\[
	\mathcal{R}_m(f_*) = \alpha \, (1-\epsilon)^m \br{e_0 + \sum_{n=1}^m 2^{n-1} e_n}
\]
and hence $\n{\mathcal{R}_m(f_*)} = \alpha \, (2(1-\epsilon))^m < a_m$ since $m < M$.
Therefore on $(m+1)$-st iteration the NGA selects
\[
	a_m f_m = \mathop{\mathrm{argmax}}_{f \in \mathcal{F}} \n{\mathcal{R}_m(f)}
	\text{ and }
	g_m = f_m.
\]
Hence $F_{g_m}(\mathcal{R}_m(f_*)) = -\alpha \, (2(1-\epsilon))^m$ and by~\eqref{lemma_nga!=ga_f_n} we get
\begin{align*}
	\mathcal{R}_{m+1}(f_*) 
	&= \mathcal{R}_m(f_*) - F_{g_m}(\mathcal{R}_m(f_*)) \, g_m
	\\
	&= \alpha \, (1-\epsilon)^m \br{e_0 + \sum_{n=1}^m 2^{n-1} e_n} \\
	&\qquad\qquad + \alpha \, (2(1-\epsilon))^m \br{-\frac{\epsilon}{2^m} e_0 - \sum_{n=1}^m \frac{\epsilon}{2^{m+1-n}} e_n + (1-\epsilon) e_{m+1}}
	\\
	&= \alpha \, (1-\epsilon)^m \br{(1-\epsilon) e_0 + \sum_{n=1}^m (1-\epsilon) 2^{n-1} e_n + 2^m (1-\epsilon) e_{m+1}}
	\\
	&= \alpha \, (1-\epsilon)^m \br{(1-\epsilon) e_0 + \sum_{n=1}^{m+1} (1-\epsilon) 2^{n-1} e_n},
\end{align*}
which proves hypothesis~\eqref{lemma_nga!=ga_hypothesis}.
Therefore we obtain
\[
	\frac{\n{\mathcal{R}_m(f_*)}}{\dist(f_*,V_m)}
	\ge \frac{\n{\mathcal{R}_m(f_*)}}{\n{f_*}}
	\ge \br{2(1-\epsilon)}^m
	\text{ for any }
	m \le M,
\]
i.e., the constant $C = C(\mathcal{F,X})$ in the inequality $\dist(\cdot,V_n) \le \n{\mathcal{R}_n(\cdot)} \le C \dist(\cdot,V_n)$ satisfies
\begin{equation}\label{lemma_nga!=ga_M_alpha}
	C \ge \br{2(1-\epsilon)}^{M(\alpha)}.
\end{equation}

Next, we construct an example of such Banach space $\mathcal{X}$ and a compact set $\mathcal{F} \subset \mathcal{X}$ that $\dist(\cdot,V_n)$ and $\n{\mathcal{R}_n(\cdot)}$ are not uniformly equivalent on $\mathcal{F}$.
Consider the space
\[
	\mathcal{X} = \br{\bigoplus_{m\in\mathbb{N}} \ell_1}_1
	= \left\{ \{x_m\}_{m\in\mathbb{N}} : x_m \in \ell_1 \text{ for any } m\in\mathbb{N} \text{ and } \sum_{m\in\mathbb{N}} \n{x_m}_1 < \infty \right\}
\]
with the norm $\n{\{x_m\}_{m\in\mathbb{N}}} = \sum_{m\in\mathbb{N}} \n{x_m}_1$.
For every $m \in \mathbb{N}$ let $\{e_n^m\}_{n=0}^\infty$ be the canonical basis for the $m$-th space $\ell_1$ in the direct sum.
Choose any $0 < \epsilon < 1/2$ and for every $m \in \mathbb{N}$ define vectors $f_*^m,f_0^m,f_1^m,f_2^m,\ldots$ as in~\eqref{lemma_nga!=ga_f_n}, i.e.,
\[
	f_*^m = e_0^m
	\text{ and }
	f_n^m = -\frac{\epsilon}{2^n} e_0^m - \sum_{k=1}^n \frac{\epsilon}{2^{n+1-k}} e_k^m + (1-\epsilon) e_{n+1}^m
	\text{ for any } n \ge 0.
\]
Take two positive sequences $\{a_n\}_{n=0}^\infty$ and $\{\alpha_m\}_{m=0}^\infty$ (monotonically decreasing to $0$) and define the sequence $\{\mathcal{F}_m\}_{m \in \mathbb{N}}$ of disjointly supported subsets of $\mathcal{X}$ as
\[
	\mathcal{F}_m = \left\{ \frac{\alpha_m}{2^m} f_*^m \right\} \cup \left\{ \frac{a_n}{2^m} f_n^m \right\}_{n=0}^\infty.
\]
Finally, denote by $\mathcal{F}$ the union of these sets, i.e.,
\[
	\mathcal{F} = \bigcup_{m\in\mathbb{N}} \mathcal{F}_m \subset \mathcal{X}.
\]
Then $\mathcal{F}$ is a compact subset of $\mathcal{X}$ but $\{\n{\mathcal{R}_n(\cdot)}\}_{n=0}^\infty$ are not uniformly equivalent to $\{\dist(\cdot,V_n)\}_{n=0}^\infty$.
Indeed, assume that there exists a constant $C = C(\mathcal{F,X}) < \infty$ such that for any $f \in \mathcal{F}$ and any $n \ge 0$
\[
	\dist(f,V_n) \le \n{\mathbb{R}_n(f)} \le C \dist(f,V_n).
\]
Then condition~\eqref{lemma_nga!=ga_M_alpha} guarantees that the constant $C$ satisfies the estimate
\[
	C \ge (2(1-\epsilon))^{M(\alpha_m)}
\]
for every $m \in \mathbb{N}$.
However, since $0 < \epsilon < 1/2$, $\{a_n\}_{n=0}^\infty$ are fixed, and $\alpha_m \to 0$ as $m \to \infty$, we conclude that $M(\alpha_m) \to \infty$ as $m \to \infty$.
Therefore $(2(1-\epsilon))^{M(\alpha_m)} \to \infty$ and we obtain a contradiction with the assumption $C < \infty$.
\end{proof}

\subsubsection*{Proof of Theorems~\ref{theorem_ga_direct} and~\ref{theorem_nga_direct}}
Next, we prove the direct estimate for the convergence rate of the greedy algorithms.
Proofs of Theorems~\ref{theorem_ga_direct} and~\ref{theorem_nga_direct} follow the same lines so we state here the proof of Theorem~\ref{theorem_nga_direct} as it is the more technically involved of the two and note that the proof of Theorem~\ref{theorem_ga_direct} can be obtained from the aforementioned proof by making a few minor changes.
Namely, replacing all $\tau_k$ with $\sigma_k$, $\n{\mathcal{R}_k(\cdot)}$ with $\dist(\cdot,V_k)$, and substituting $1$ in place of $R$ provides the desired proof.
\begin{proof}[Proof of Theorem~\ref{theorem_nga_direct}]
Take any $n > 0$ and let $V_{n+1} = \spn\{f_0,\ldots,f_n\}$ be the subspace generated by the greedy algorithm after the $(n+1)$-st iteration.
Define linear functionals $\varphi_0,\ldots,\varphi_n$, where $\varphi_k : \spn\{f_k\} \to \mathbb{R}$ is such that for any $\alpha \in \mathbb{R}$
\[
	\varphi_k(\alpha f_k) = \alpha \tau_k = \alpha \n{\mathcal{R}_k(f_k)}.
\]
Then, since the greedy selection step provides $\n{\mathcal{R}_k(f_k)} = \tau_k$, we have for any $x \in \spn\{f_k\}$
\[
	|\varphi_k(x)| \le \n{\mathcal{R}_k(x)},
\]
and by the Hahn--Banach theorem for each $0 \le k \le n$ there exists a linear extension $\phi_k : V_{n+1} \to \mathbb{R}$ such that $\phi_k(x) = \varphi_k(x)$ for any $x \in \spn\{f_k\}$ and $|\phi_k(x)| \le \n{\mathcal{R}_k(x)}$ for any $x \in V_{n+1}$.
Therefore
\begin{equation}\label{theorem_nga_direct_Phi_ij}
	\phi_k(f_i) = \left\{
	\begin{array}{ll}
		0, & i < k
		\\
		\tau_k, & i = k
	\end{array}
	\right.
	\quad\text{and}\quad
	|\phi_k(f_i)| \le \tau_k, \ i > k.
\end{equation}
Note that estimate~\eqref{estimate_norm_R_n_V} provides
\begin{equation}\label{theorem_nga_direct_phi_k_norm}
	\n{\phi_k}_{V_{n+1}^*} \le \n{\mathcal{R}_k}_{V_{n+1}^*} \le B_k^g \le B_{n+1}^g.
\end{equation}
Let $\{e_k\}_{k=0}^n \in V_{n+1}$ be biorthogonal to $\{\phi_k\}_{k=0}^n$ vectors, i.e., $\phi_k(e_i) = \delta_{ik}$.
Then, since $V_{n+1} = \spn\{f_0,\ldots,f_n\}$, there exist coefficients $\{\beta_{ij}\}_{i,j=0}^n$ such that for each $0 \le i \le n$
\begin{equation}\label{theorem_nga_direct_tau_i_e_i}
	\tau_i e_i = \sum_{j=0}^n \beta_{ij} f_j.
\end{equation}
Hence for any $0 \le k \le n$ we have
\begin{equation}\label{theorem_nga_direct_tau_i_delta_ik}
	\tau_i \delta_{ik} = \tau_i \phi_k(e_i) = \sum_{j=0}^n \beta_{ij} \phi_k(f_j).
\end{equation}
Define matrices $\Sigma = (\tau_i \delta_{ij})_{i,j=0}^n$, $B = (\beta_{ij})_{i,j=0}^n$, and $\Phi = (\phi_j(f_i))_{i,j=0}^n$.
Then system~\eqref{theorem_nga_direct_tau_i_delta_ik} can be rewritten in the matrix form
\[
	\Sigma = B \Phi.
\]
Note that matrix $\Sigma$ is diagonal with $\tau_0,\ldots,\tau_n$ on the diagonal, and from~\eqref{theorem_nga_direct_Phi_ij} we get that $\Phi$ is lower-triangular with $\tau_0,\ldots,\tau_n$ on the diagonal.
Thus $B = \Sigma \Phi^{-1}$ is lower-triangular with $1$ on the diagonal.
We prove by induction that for any $j < i$ the estimate $|\beta_{ij}| \le 2^{i-j-1}$ holds.
Indeed, from~\eqref{theorem_nga_direct_tau_i_delta_ik} and the fact that matrices $B$ and $\Phi$ are lower-triangular we obtain for $k = i-1$
\[
	0 = \tau_i \delta_{i,i-1}
	= \sum_{j=0}^n \beta_{ij} \phi_k(f_j) 
	= \sum_{j=i-1}^i \beta_{ij} \phi_{i-1}(f_j)
	= \beta_{i,i-1} \phi_{i-1}(f_{i-1}) + \beta_{ii} \phi_{i-1}(f_i).
\]
By~\eqref{theorem_nga_direct_Phi_ij} we have $\phi_{i-1}(f_{i-1}) = \tau_{i-1}$ and $|\phi_{i-1}(f_i)| \le \tau_{i-1}$, thus
\[
	|\beta_{i,i-1}| \le |\beta_{ii}| \frac{|\phi_{i-1}(f_i)|}{\tau_{i-1}} \le 1,
\]
which proves the base of induction.
Assume that $|\beta_{ij}| \le 2^{i-j-1}$ for any $j = i-m+1,\ldots,i-1$.
Then for $k = i - m$ we get from~\eqref{theorem_nga_direct_tau_i_delta_ik}
\[
	0 = \tau_i \delta_{i,i-m}
	= \sum_{j=0}^n \beta_{ij} \phi_{i-m}(f_j)
	= \sum_{j=i-m}^i \beta_{ij} \phi_{i-m}(f_j).
\]
Since $\phi_{i-m}(f_{i-m}) = \tau_{i-m}$ and $|\phi_{i-m}(f_j)| \le \tau_{i-m}$ for $j = i-m+1,\ldots,i$, we obtain
\[
	|\beta_{i,i-m}| \le \sum_{j=i-m+1}^i |\beta_{ij}| \frac{|\phi_{i-m}(f_j)|}{\tau_{i-m}}
	\le 1 + \sum_{j=i-m+1}^{i-1} 2^{i-j-1}
	= 2^{m-1},
\]
which proves the hypothesis.
Therefore for any $0 < i < n$ we have
\begin{equation}\label{theorem_nga_direct_sum_b_ij}
	\sum_{j=0}^n |\beta_{ij}| \le 1 + \sum_{j=0}^{i-1} 2^{i-j-1} = 2^i.
\end{equation}
Take any $\epsilon > 0$ and let $X_n$ be an $n$-dimensional subspace that almost attains Kolmogorov $n$-width $d_n(\mathcal{F,X})$, i.e.,
\begin{equation}\label{theorem_nga_direct_X_n}
	\sup_{f \in \mathcal{F}} \dist(f,X_n) \le (1+\epsilon) \, d_n.
\end{equation}
Then there exist such $\{h_j\}_{j=0}^n \in X_n$ that $\n{f_j - h_j} \le (1 + \epsilon) \, d_n$ for any $0 \le j \le n$.
Denote $Y = \spn\{V_{n+1},X_n\}$ and by~\eqref{e_norm} introduce a Euclidean norm $\n{\cdot}_e$ on $Y$ such that for any $y \in Y$
\[
	\n{y} \le \n{y}_e \le \gamma_{2n+1} \n{y}.
\]
Let $P$ be the orthogonal projector from $Y$ onto $V_{n+1}$ in $\n{\cdot}_e$-norm and denote $W = P(X_n) \subset V_{n+1}$.
Since $\dim W \le n < n+1 = \dim V_{n+1}$, there exists a linear functional $\psi : V_{n+1} \to \mathbb{R}$ such that $\n{\psi}_{V_{n+1}^*} = 1$ and $\ker \psi = W$ (see, e.g.,~\cite[Proposition~2.8]{faetal2013}).
Then for each $0 \le j \le n$ we have
\begin{align*}
	|\psi(f_j)| 
	&= |\psi(f_j - P(h_j))|
	= |\psi(P(f_j - h_j))|
	\le \n{P(f_j - h_j)}
	\\
	&\le \n{P(f_j - h_j)}_e
	\le \n{f_j - h_j}_e
	\le \gamma_{2n+1} \n{f_j - h_j}
	\le \gamma_{2n+1} (1 + \epsilon) \, d_n,
\end{align*}
and hence by~\eqref{theorem_nga_direct_tau_i_e_i} and~\eqref{theorem_nga_direct_sum_b_ij} for any $0 \le i \le n$ we get
\[
	|\psi(\tau_i e_i)| = \left| \psi \br{\sum_{j=0}^n \beta_{ij} f_j} \right|
	\le \sum_{j=0}^n |\beta_{ij}| |\psi(f_j)|
	\le 2^i \gamma_{2n+1} (1 + \epsilon) \, d_n.
\]
Finally, since $\psi = \sum_{i=0}^n \psi(e_i) \phi_i$ and $\n{\psi}_{V_{n+1}^*} = 1$, we deduce
\[
	1 = \n{\psi}_{V_{n+1}^*} = \n{\sum_{i=0}^n \psi(e_i) \phi_i}_{V_{n+1}^*}
	\le \sum_{i=0}^n |\psi(e_i)| \n{\phi_i}_{V_{n+1}^*},
\]
and using estimates~\eqref{theorem_nga_direct_phi_k_norm} and $\sigma_n \le \sigma_i \le \tau_i$ for any $i \le n$ (see~\eqref{estimate_sigma_tau}) we arrive at
\begin{align*}
	\sigma_n 
	\le \sum_{i=0}^n \tau_i |\psi(e_i)| \n{\phi_i}_{V_{n+1}^*}
	&\le \sum_{i=0}^n 2^i \gamma_{2n+1} (1 + \epsilon) \, d_n B_{n+1}^g
    \\
	&\le B_{n+1}^g \, 2^{n+1} \gamma_{2n+1} (1 + \epsilon) \, d_n.
\end{align*}
Taking infimum over all $\epsilon > 0$ completes the proof.
\end{proof}

\subsubsection*{Proof of Theorem~\ref{theorem_nga_indirect}}
For the proof of the delayed estimate of the convergence rate we will use the following technical lemma, which is based on Hadamard's inequality.
\begin{lemma}[{\cite[Lemma~2.1]{depewo2013}}]\label{lemma_proj}
Let $A = (a_{ij})_{i,j=0}^{n-1}$ be a $n \times n$ lower triangular matrix with rows $a_0,\ldots,a_{n-1}$, $W_m$ be any $m$-dimensional subspace of $\mathbb{R}^n$, and $P_m$ be the orthogonal projection from $\mathbb{R}^n$ onto $W_m$.
Then
\[
	\prod_{i=0}^{n-1} a_{ii}^2 \le \br{\frac{1}{m} \sum_{i=0}^{n-1} \n{P_m a_i}_2^2}^m \br{\frac{1}{n-m} \sum_{i=0}^{n-1} \n{a_i-P_m a_i}_2^2}^{n-m}, 
\]
where $\n{\cdot}_2$ is the Euclidean norm of a vector in $\mathbb{R}^n$.
\end{lemma}

\begin{proof}[Proof of Theorem~\ref{theorem_nga_indirect}]
Take any $n > m > 0$ and any $\epsilon > 0$.
Let $V_n = \spn\{f_0,\ldots,f_{n-1}\}$ be the subspace generated by the greedy algorithm after the $n$-th iteration and $X_m$ be an $m$-dimensional subspace that almost attains Kolmogorov $m$-width $d_m(\mathcal{F,X})$, i.e.,
\begin{equation}\label{theorem_nga_indirect_X_m}
	\sup_{f \in \mathcal{F}} \dist(f,X_m) \le (1+\epsilon) \, d_m.
\end{equation}
Denote $Y = \spn\{V_n,X_m\}$ and introduce a Euclidean norm $\n{\cdot}_e$ on $Y$ (see~\eqref{e_norm}) such that for any $y \in Y$
\[
	\n{y} \le \n{y}_e \le \gamma_{n+m} \n{y}.
\]
Let $\{e_j\}_{j=0}^{n-1}$ be the orthonormal system in $\n{\cdot}_e$-norm obtained from $\{f_j\}_{j=0}^{n-1}$ by the Gram--Schmidt process and define the matrix $A = (a_{ij})_{i,j=0}^{n-1}$ given by $a_{ij} = e_j(f_i)$.
Then, since $\mathcal{R}_i(V_i) = 0$ and thus $\n{\mathcal{R}_i(f_i)} = \n{\mathcal{R}_i(f_i-\proj(f_i,V_i))}$, we obtain
\begin{align}
	\nonumber
	a_{ij} &= 0 \text{ for any } j > i,
	\\
	\label{theorem_nga_indirect_a_ii>sigma_i}
	a_{ii} &= e_i(f_i) = \dist_e(f_i,V_i) \ge \dist(f_i,V_i) 
	\ge \frac{\n{\mathcal{R}_i(f_i)}}{\n{\mathcal{R}_i}_{V_{i+1}^*}} 
	= \frac{\tau_i}{\n{\mathcal{R}_i}_{V_{i+1}^*}},
\end{align}
where $\dist_e(\cdot,\cdot)$ denotes the distance in $\n{\cdot}_e$-norm.
\\
Consider the subspace $W = \{e_0(h),\ldots,e_{n-1}(h)\} \subset \mathbb{R}^n$, $h \in X_m$ and take any $m$-dimensional subspace $W_m \subset \mathbb{R}^n$ containing $W$, i.e.,
\[
	W \subset W_m \subset \mathbb{R}^n
	\text{ and }
	\dim W \le \dim W_m = m < n.
\]
Denote by $P$ and $P_m$ the orthogonal projectors from $\mathbb{R}^n$ onto $W$ and $W_m$ respectively.
Let $a_0,\ldots,a_{n-1}$ denote the rows of the matrix $A$, i.e., $a_i = (e_0(f_i),\ldots,e_{n-1}(f_i))$.
Then, since $\mathcal{F}$ is contained in the unit ball of $\mathcal{X}$, we get
\begin{align*}
	\n{P_m a_i}_2^2
	&\le \n{a_i}_2^2
	= \sum_{j=1}^i a_{ij}^2
	= \sum_{j=1}^i (e_j(f_i))^2
	= \n{f_i}_e^2
	\le \gamma_{n+m}^2 \n{f_i}^2
	\le \gamma_{n+m}^2.
\end{align*}
By~\eqref{theorem_nga_indirect_X_m} there exist $\{h_i\}_{i=0}^{n-1} \in X_m$ such that $\n{f_i - h_i} \le (1+\epsilon) \, d_m$.
Denote $w_i = (e_0(h_i),\ldots,e_{n-1}(h_i)) \in W$, then
\begin{multline*}
	\n{a_i - P_m a_i}_2^2
	\le \n{a_i - P a_i}_2^2
	\le \n{a_i - w_i}_2^2
    \\
	= \n{(e_0(f_i),\ldots,e_{n-1}(f_i)) - (e_0(h_i),\ldots,e_{n-1}(h_i))}_2^2
    = \sum_{j=0}^{n-1}  (e_j(f_i-h_i))^2
	\\
	\le \n{f_i-h_i}_e^2
	\le \gamma_{n+m}^2 \n{f_i-h_i}^2
	\le \gamma_{n+m}^2 (1+\epsilon)^2 \, d_m^2.
\end{multline*}
Hence from the inequality $\sigma_n \le \tau_i$ for any $i \le n$, estimates~\eqref{estimate_norm_R_n_V} and~\eqref{theorem_nga_indirect_a_ii>sigma_i}, and Lemma~\ref{lemma_proj} we obtain
\begin{align*}
	\frac{\sigma_n^{2n}}{\min\{R^{n(n-1)},(C_g+1)^{2n}\}}
	&\le \prod_{i=0}^{n-1} \frac{\tau_i^2}{\n{\mathcal{R}_i}_{V_n^*}^2}
	= \prod_{i=0}^{n-1} a_{ii}^2 \\
	& \le \br{\frac{1}{m} \sum_{i=0}^{n-1} \n{P_m a_i}_2^2}^m \br{\frac{1}{n-m} \sum_{i=0}^{n-1} \n{a_i-P_m a_i}_2}^{n-m}
	\\
	&\le \br{\frac{n}{m} \gamma_{n+m}^2}^m \br{\frac{n}{n-m} \gamma_{n+m}^2 (1+\epsilon)^2 d_m^2}^{n-m}
	\\
	&= \frac{n^n}{m^m (n-m)^{n-m}} \gamma_{n+m}^{2n} \br{(1+\epsilon) \, d_m}^{2(n-m)}
	\\
	&\le 2^n \gamma_{n+m}^{2n} \br{(1+\epsilon) \, d_m}^{2(n-m)}.
\end{align*}
Since this estimate holds for any $\epsilon > 0$ and $0 < m < n$, we get
\begin{align*}
	\sigma_n 
	&\le \sqrt{2} \min\{R^{(n-1)/2},C_g+1\} \min_{0<m<n} \gamma_{n+m} \, d_m^{1-m/n}
	\\
	&\le \sqrt{2} B_{n/2}^g \min_{0<m<n} \gamma_{n+m} \, d_m^{1-m/n}.
\end{align*}
\end{proof}

\subsubsection*{Proof of Theorem~\ref{theorem_nga_special}}
Finally, we prove the estimates on convergence rates of the NGA in special cases.
\begin{proof}[Proof of Theorem~\ref{theorem_nga_special}]
\begin{enumerate}
	\item Let $d_n(\mathcal{F,X}) \le A \exp(-a n^\alpha)$ with some constants $0 < \alpha,a,A < \infty$.
	Then Theorem~\ref{theorem_nga_indirect} provides for $n = 2m$
	\[
		\sigma_{2m} 
		\le \sqrt{2} B_m^g \gamma_{3m} \sqrt{d_m}
		\le \sqrt{2A} B_m^g \gamma_{3m} \exp\br{-\frac{a}{2} m^\alpha},
	\]
	which proves the desired estimate.
	\item Proof of this part essentially uses the proof of Theorem~\ref{theorem_nga_indirect} with additional rigorous technical estimates, and repeats to the letter the proof of the corresponding result for the OGA presented in~\cite[Theorem~2.3]{wo2015} and~\cite[Theorem~3.1]{ng2018}.
	Since the proof of Theorem~\ref{theorem_nga_indirect} is provided in this paper, we leave the meticulous technical estimates to the interested reader.
\end{enumerate}
\end{proof}

\bibliographystyle{abbrv}      
\bibliography{NGA}   
\end{document}